\documentclass[english,10pt]{amsart}
\usepackage[english]{babel}

\usepackage[matrix,arrow]{xy}
\xyoption{all}
\usepackage{amscd,amssymb,amsfonts,amsmath}

\usepackage{graphics}
\usepackage{epsfig}
\usepackage{mathrsfs}

\newcommand\mypagesizel{
\textwidth= 6.5in
\textheight=9in
\voffset-.55in
\hoffset -0.75in
\marginparwidth=56pt
}

\newcommand{\Pic}{\textup{Pic}}

\newcommand{\Proj}{\textup{Proj}}

\newcommand{\Chow}{\textup{Chow}}
\newcommand{\codim}{\textup{codim}}

\renewcommand{\c}[0]{{\mathbb C}}  
\newcommand{\p}[0]{{\mathbb P}}
\newcommand{\rat}[0]{\operatorname{RatCurves}^n}
\newcommand{\Sing}{\textup{Sing}}

\newcommand{\NE}{\overline{\textup{NE}}}

\renewcommand{\phi}{\varphi}
\newcommand{\into}{\hookrightarrow}
\newcommand{\map}{\dashrightarrow}

\renewcommand{\le}{\leqslant}
\renewcommand{\ge}{\geqslant}

\newcommand{\bQ}{\textup{\textbf{Q}}}

\renewcommand{\O}{\mathcal{O}}

\newcommand{\sA}{\mathscr{A}}
\newcommand{\sB}{\mathscr{B}}
\newcommand{\sC}{\mathscr{C}}

\newcommand{\sE}{\mathscr{E}}
\newcommand{\sF}{\mathscr{F}}
\newcommand{\sG}{\mathscr{G}}
\newcommand{\sH}{\mathscr{H}}
\newcommand{\sI}{\mathscr{I}}

\newcommand{\sK}{\mathscr{K}}
\newcommand{\sL}{\mathscr{L}}
\newcommand{\sM}{\mathscr{M}}
\newcommand{\sN}{\mathscr{N}}
\newcommand{\sO}{\mathscr{O}}

\newcommand{\sQ}{\mathscr{Q}}

\newcommand{\sS}{\mathscr{S}}
\newcommand{\sT}{\mathscr{T}}

\newcommand{\sV}{\mathscr{V}}
\newcommand{\sW}{\mathscr{W}}

\newcommand{\rank}{\textup{rank}}

\mypagesizel

\newtheorem{thm}{Theorem}%[section]
\newtheorem*{thm*}{Theorem}

\newtheorem{lemma}[thm]{Lemma}

\newtheorem{prop}[thm]{Proposition}

\theoremstyle{definition}
\newtheorem{defn}[thm]{Definition}

\newtheorem{say}[thm]{}
\newtheorem{exmp}[thm]{Example}
\newtheorem{assumptions}[thm]{Assumptions}

\newtheorem{defn-thm}[thm]{Definition-Theorem}

\newtheorem{rem}[thm]{Remark}

\theoremstyle{remark}

\newtheorem*{not-and-def}{Notation and definitions}

\numberwithin{equation}{section}

\begin{document}

\title{Codimension $1$ Mukai foliations on complex projective manifolds}

\author{Carolina \textsc{Araujo}} 

\address{\noindent Carolina Araujo: IMPA, Estrada Dona Castorina 110, Rio de
  Janeiro, 22460-320, Brazil} 

\email{caraujo@impa.br}

\author{St\'ephane \textsc{Druel}}

\address{St\'ephane Druel: Institut Fourier, UMR 5582 du
  CNRS, Universit\'e Grenoble 1, BP 74, 38402 Saint Martin
  d'H\`eres, France} 

\email{druel@ujf-grenoble.fr}

\thanks{The first named author was partially supported by CNPq and Faperj Research 
  Fellowships}

\thanks{The second named author was partially supported by the CLASS project of the 
A.N.R}

\subjclass[2010]{14M22, 37F75}

\begin{abstract}
In this paper we classify codimension $1$ Mukai foliations on complex projective manifolds.
\end{abstract}

\maketitle

\tableofcontents

%%%%%%%%%%%%%%%%%%%%%%%%%%%%%%%%%%%%%%%%%%%%%
%
%			SECTION 1
%
%%%%%%%%%%%%%%%%%%%%%%%%%%%%%%%%%%%%%%%%%%%%%

\section{Introduction}

Codimension $1$ holomorphic foliations on complex projective spaces 
are a central theme in complex dynamics. 
An important numerical invariant of such a foliation $\sF$ is its degree $\deg(\sF)$, defined as the number 
of tangencies of a general line with $\sF$. 
Codimension $1$ foliations  on $\p^n$ with low degree present special behavior, and 
those with $\deg(\sF)\le 2$ have been classified. 
Codimension $1$ foliations  on $\p^n$ with $\deg(\sF)\le 1$ were classified in \cite{jouanolou}.
If  $\deg(\sF)=0$, then $\sF$ is induced by a pencil of hyperplane sections, i.e., it is the relative tangent 
sheaf to a linear projection $\p^n\map \p^1$. 
If  $\deg(\sF)=1$, then 
\begin{itemize}
\item either $\sF$ is induced by a pencil of hyperquadrics containing a double hyperplane, or 
\item $\sF$ is the linear pullback of a foliation on $\p^{2}$ induced by a global 
holomorphic vector field. 
\end{itemize}
Codimension $1$ foliations of degree $2$ on $\p^n$ were classified in \cite{CLN}.
The space of such foliations has $6$ irreducible components, and much is known about them. 
In particular, when $n\ge 4$, their leaves are always covered by rational curves.

In this paper we extend this classification to arbitrary complex projective manifolds.
In order to do so, we reinterpret the degree of a foliation $\sF$ on $\p^n$
as a numerical invariant defined in terms of its \emph{canonical class} $K_{\sF}:=-c_1(\sF)$. 
For a codimension $1$ foliation $\sF$ on $\p^n$, $\deg(\sF)=n-1+\deg(K_{\sF})$.
So, foliations with low degree are precisely those with $-K_{\sF}$ most positive.

\begin{defn}[\cite{fano_fols}]
A \emph{Fano foliation}   is a holomorphic foliation $\sF$ on a complex projective manifold $X$
such that $-K_{\sF}$ is ample. 
The \emph{index} $\iota_{\sF}$ of  $\sF$ is the largest integer dividing $-K_{\sF}$ in $\Pic(X)$. 
\end{defn}

It follows from \cite[Theorem 0.1]{bogomolov_mcquillan01} that the leaves 
of a Fano foliation $\sF$  are always covered by positive dimensional rationally connected algebraic subvarieties of $X$.  
Recent results suggest that the higher is the index of $\sF$, the higher is the dimension of these subvarieties.
In order to state this precisely, we define the \emph{algebraic} and \emph{transcendental} parts of
a holomorphic foliation.

\begin{defn}
Let $\sF$ be a holomorphic foliation of rank $r_{\sF}$ on a normal variety $X$.
There exists  a normal variety $Y$, unique up to birational equivalence,  a dominant rational map with connected fibers $\varphi:X\map Y$,
and a holomorphic foliation $\sG$ on $Y$ 
of rank $r_{\sG}=r_{\sF}-\big(\dim(X)-\dim(Y)\big)$ such that the following holds (see \cite[Section 2.4]{loray_pereira_touzet}).
\begin{enumerate}
	\item $\sG$ is purely transcendental, i.e., there is no positive dimensional algebraic subvariety through a general point of $Y$ that is tangent to $\sG$; and
	\item $\sF$ is the pullback of $\sG$ via $\varphi$ (see \ref{pullback_foliations} for this notion).
\end{enumerate}
The foliation on $X$ induced by $\varphi$ is called the  \emph{algebraic part} of $\sF$, and its rank is the \emph{algebraic rank} of $\sF$, 
which we denote by $r_{\sF}^a$. When $r_{\sF}^a=r_{\sF}$, we say that $\sF$ is \emph{algebraically integrable}.
\end{defn}

\begin{thm} \label{Thm:ADK}
Let $\sF$ be a Fano foliation of rank $r_{\sF}$ on a complex projective manifold $X$.
Then $\iota_{\sF}\le r_{\sF}$, and equality holds only if $X\cong \p^n$ \textup{(\cite[Theorem 1.1]{adk08})}.
In this case, by \cite[Th\'eor\`eme 3.8]{cerveau_deserti}, 
$\sF$ is induced by a  linear projection $\p^n \dashrightarrow \p^{n-r_{\sF}}$.
In particular, $r_{\sF}^a=r_{\sF}$.
\end{thm}

In analogy with the case of Fano manifolds, we define \emph{del Pezzo foliations} to be
Fano foliations $\sF$ with index $\iota_{\sF} = r_{\sF}-1\ge 1$.
Del Pezzo foliations were investigated in \cite{fano_fols} and \cite{codim_1_del_pezzo_fols}.
By \cite[Theorem 1.1]{fano_fols}, if $\sF$ is a del Pezzo foliation on a complex projective manifold $X$, then  $r_{\sF}^a=r_{\sF}$, except 
when $X\cong \p^n$ and  $\sF$ is the pullback  under a linear projection of a transcendental foliation on $\p^{n-r_{\sF}+1}$ induced by a global 
vector field, in which case $r_{\sF}^a=r_{\sF}-1$.

The following is the complete classification of codimension $1$ del Pezzo foliations
on complex projective manifolds. 
For  manifolds with Picard  number $1$, the classification was obtained in  \cite[Proposition 3.7]{lpt3fold}, while 
\cite[Theorem 1.3]{codim_1_del_pezzo_fols} deals  with mildly singular  varieties of arbitrary Picard number.

\begin{thm}[{\cite[Theorem 1.3]{codim_1_del_pezzo_fols}}] \label{Thm:codim1_dP}
Let $\sF$ be a codimension $1$ del Pezzo foliation on an $n$-dimensional complex projective manifold $X$.
\begin{enumerate}
	\item Suppose that $\rho(X)=1$. Then either $X\cong\p^n$ and $\sF$ is a degree $1$ foliation, or 
		$X\cong Q^n\subset \p^{n+1}$ and $\sF$ is induced by a pencil of hyperplane sections.
	\item Suppose that $\rho(X) \ge 2$.  Then  there exist
			\begin{itemize}
			\item an exact sequence of vector bundles on $\p^1$, $0\to \sK\to \sE\to \sV \to 0$;
			\item a foliation by curves $\sC$ on $\p_C(\sK)$, generically transverse to the natural projection $p:\p_C(\sK)\to \p^1$,
				induced by a nonzero global section of $T_{\p(\sK)}\otimes q^*\det(\sV)^*$;
			\end{itemize}	
		such that  $X\cong \p_C(\sE)$, and $\sF$ is the pullback of $\sC$ via the induced relative linear projection 
		$\p_C(\sE)\map \p_C(\sK)$.
		Moreover, one of the following holds.
		\begin{enumerate}
			\item $(\sE,\sK)\cong \big(\sO_{\p^1}(2)\oplus \sO_{\p^1}(a)^{\oplus 2},\sO_{\p^1}(a)^{\oplus 2}\big)$
				for some positive integer $a$.
			\item $(\sE,\sK)\cong \big(\sO_{\p^1}(1)^{\oplus 2}\oplus \sO_{\p^1}(a)^{\oplus 2},\sO_{\p^1}(a)^{\oplus 2}\big)$
				for some positive integer $a$.
			\item $(\sE,\sK)\cong\big(\sO_{\p^1}(1)\oplus \sO_{\p^1}(a)\oplus \sO_{\p^1}(b),\sO_{\p^1}(a)\oplus \sO_{\p^1}(b)\big)$ for distinct 
				positive integers $a$ and $b$.
		\end{enumerate}
\end{enumerate}
\end{thm}

Next we define  \emph{Mukai foliations} as Fano foliations $\sF$ with index $\iota_{\sF} = r_{\sF}-2\ge 1$.
When $X=\p^n$ and $r_{\sF}=n-1\ge 3$, the Mukai condition is equivalent to  $\deg(\sF)=2$.
%Codimension $1$ degree $2$ foliations on $\p^n$, $n\ge 3$, have been classified in \cite{CLN}. 
One checks from the classification in \cite{CLN} that $r_{\sF}^a\ge r_{\sF}-2$.

The aim of this paper is to classify codimension $1$ Mukai foliations on complex projective manifolds $X\not\cong\p^n$. 
The classification is summarized in the following two theorems.

\begin{thm}\label{main_thm_rho=1}
Let $\sF$ be a codimension $1$ Mukai foliation on an $n$-dimensional complex projective manifold $X\not\cong\p^n$ with 
$\rho(X)=1$, $n\ge 4$. Then the pair $(X,\sF)$ satisfies one of the following conditions.
\begin{enumerate}
		\item  $X\cong Q^n\subset \p^{n+1}$ and $\sF$ is one of the following.
			\begin{enumerate}
				\item $\sF$ cut out by a general pencil of  hyperquadrics of $\p^{n+1}$ containing a double hyperplane.
				In this case, $r_{\sF}^a= r_{\sF}$.
				\item $\sF$ is the pullback  under the restriction of a linear projection $\p^{n+1}\map \p^2$ of a foliation on $\p^{2}$ induced by a 
				global vector field. In this case, $r_{\sF}^a\ge r_{\sF}-1$.	
			\end{enumerate}
		\item  $X$ is a Fano manifold with $\rho(X)=1$ and index $\iota_X=n-1$, and $\sF$ is induced by a pencil in $|\sO_X(1)|$, where $\sO_X(1)$
			is the ample generator of $\Pic(X)$. In this case, $r_{\sF}^a= r_{\sF}$.	
\end{enumerate} 
\end{thm}

\begin{rem}
In case (2), by Fujita's classification (see Section~\ref{Fano_manifolds_high_index}), $X$ is isomorphic to one of the following: 
\begin{itemize}
	\item A cubic hypersurface in $\p^{n+1}$.
	\item An intersection of two hyperquadrics in $\p^{n+2}$.
	\item A linear section of the Grassmannian $G(2,5)\subset\p^9$ under the Pl\"ucker embedding.
	\item A  hypersurface of degree $4$ in the weighted projective space $\p(2,1,\ldots,1)$.
	\item A hypersurface of degree $6$ in the weighted projective space $\p(3,2,1,\ldots,1)$.
\end{itemize}
\end{rem}

\begin{thm}\label{main_thm_rho>1}
Let $X$ be an $n$-dimensional complex projective manifold with $\rho(X)>1$, $n\ge 4$.
Let $\sF$ be a codimension $1$ Mukai foliation on  $X$.
Then one of the following holds.
\begin{enumerate}

		\item $X$ admits a $\p^{n-1}$-bundle structure $\pi:X\to \p^1$, $r_{\sF}^a= r_{\sF}$, 
			and the restriction of $\sF$ to a general fiber of $\pi$ is induced by a 
			pencil of  hyperquadrics of $\p^{n-1}$ containing a double hyperplane.	

		\item There exist 
			\begin{itemize}
			\item a complete smooth curve $C$, together with an exact sequence of vector bundles on $C$
				$$
				0\ \to \ \sK \ \to \ \sE \ \to \ \sV \ \to \ 0,
				$$
				with $\sE$ ample of rank $n$, and $r:=\rank(\sK)\in \{2,3\}$; 
			\item a codimension $1$ foliation $\sG$ on $\p_C(\sK)$, generically transverse to the natural projection $p:\p_C(\sK)\to C$,
				satisfying  $\det(\sG)  \cong  p^*\big(\det(\sV)\big)\otimes \sO_{\p_C(\sK)}(r-3)$ and $r_{\sG}^a\ge r_{\sG}-1$;		
			\end{itemize}	
			such that  $X\cong \p_C(\sE)$, and $\sF$ is the pullback of $\sG$ via the induced relative linear projection 
			$\p_C(\sE)\map \p_C(\sK)$. In this case,  $r_{\sF}^a\ge r_{\sF}-1$.	
 
		\item  $X$ admits a $Q^{n-1}$-bundle structure $\pi:X\to \p^1$, $r_{\sF}^a= r_{\sF}$, 
			and the restriction of $\sF$ to a general fiber of $\pi$ is induced by a 
			pencil of  hyperplane sections of $Q^{n-1}$.
			More precisely, there exist
			\begin{itemize}
			\item an exact sequence of vector bundles on $\p^1$
				$$
				0\ \to \ \sK \ \to \ \sE \ \to \ \sV \ \to \ 0,
				$$
				with  natural projections $\pi:\p_{\p^1}(\sE)\to \p^1$
				and $q:\p_{\p^1}(\sK)\to \p^1$;
			\item  an integer $b$ and a foliation by rational curves $\sG\cong q^*\big(\det(\sV)\otimes \sO_{\p^1}(b)\big)$ on $\p_{\p^1}(\sK)$;
			\end{itemize}
			such that $X\in \big|\sO_{\p(\sE)}(2)\otimes \pi^*\sO(b)\big|$, and
			$\sF$ is the pullback of $\sG$ via the restriction to $X$ of the relative linear projection $\p_{\p^1}(\sE) \map  \p_{\p^1}(\sK)$.
			Moreover, one of the following holds.
				\begin{enumerate}
					\item $(\sE,\sK)\cong (\sO_{\p^1}(a)^{\oplus 2}\oplus\sO_{\p^1}^{\oplus 3},\sO_{\p^1}(a)^{\oplus 2})$
					for some integer $a\ge 1$, and $b=2$ (n=4).
					\item $(\sE,\sK)\cong (\sO_{\p^1}(a)^{\oplus 2}\oplus\sO_{\p^1}^{\oplus 2}\oplus \sO_{\p^1}(1),\sO_{\p^1}(a)^{\oplus 2})$
					for some integer $a\ge 1$, and $b=1$ (n=4).
					\item $(\sE,\sK)\cong (\sO_{\p^1}(a)^{\oplus 2}\oplus\sO_{\p^1}\oplus \sO_{\p^1}(1)^{\oplus 2},\sO_{\p^1}(a)^{\oplus 2})$ 
					for some integer $a\ge 1$, and $b=0$ (n=4).
					\item $\sK\cong \sO_{\p^1}(a)^{\oplus 2}$ for some integer $a$, and
					$\sE$ is an ample vector bundle of rank $5$ or $6$ with $\deg(\sE)=2+2a-b$ ($n\in\{4,5\}$).
					\item $\sK\cong \sO_{\p^1}(a)\oplus \sO_{\p^1}(c)$ for distinct integers $a$ and $c$, and
					$\sE$ is an ample vector bundle of rank $5$ or $6$ with $\deg(\sE)=1+a+c-b$ ($n\in\{4,5\}$).
				\end{enumerate}
		\item	 There exist 
			\begin{itemize}
			\item a smooth projective surface $S$, together with an exact sequence of $\sO_S$-modules
				$$
				0\ \to \ \sK \ \to \ \sE \ \to \ \sQ \ \to \ 0,
				$$
				where $\sK$, $\sE$, and $\sV:=\sQ^{**}$ are vector bundles on $S$, $\sE$ is ample of rank $n-1$, and
				$\rank(\sK)=2$;
			\item a codimension $1$ foliation $\sG$ on $\p_S(\sK)$, generically transverse to the natural projection  $q:\p_{S}(\sK)\to S$,
				satisfying $\det(\sG)\cong q^*\det(\sV)$ and $r_{\sG}^a\ge 1$;		
			\end{itemize}	
			such that  $X\cong \p_{S}(\sE)$, and
			$\sF$ is the pullback of $\sG$ via the induced relative linear projection $\p_C(\sE)\map \p_C(\sK)$.
			In this case,  $r_{\sF}^a\ge r_{\sF}-1$.	
			Moreover, one of the following holds.
				\begin{enumerate}
				\item $S\cong\p^2$, $\det(\sV)\cong\sO_{\p^2}(i)$ for some $i\in \{1,2,3\}$, and $4\le n\le 3+i$.
				\item $S$ is a del Pezzo surface $\not\cong\p^2$, $\det(\sV)\cong \sO_S(-K_S)$ , and $4\le n\le 5$.
				\item $S\cong \p^1\times \p^1$, $\det(\sV)$ is a line bundle of type $(1,1)$, $(2,1)$ or $(1,2)$, and $n=4$.
				\item $S\cong \mathbb{F}_e$ for some integer $e \ge 1$, $\det(\sV)\cong\sO_{\mathbb{F}_e}(C_0+(e+i)f)$,
				where  $i\in\{1,2\}$, $C_0$ is the minimal section of the natural morphism $\mathbb{F}_e\to\p^1$,
				$f$ is a general fiber, and $n=4$.
				\end{enumerate}
		\item $n=5$, $X$ is the blowup of one point $P\in \p^5$, and $\sF$ is induced by a pencil 
			of hyperplanes in $\p^5$ containing $P$ in its base locus.
		\item $n=4$, $X$ is the blowup of $\p^4$ at $m\le 8$ points in general position on a plane $\p^2 \cong S\subset \p^4$, 
			and $\sF$ is induced by the pencil of hyperplanes in $\p^4$ with base locus $S$.
		\item $n=4$, $X$ is the blowup of a smooth quadric $Q^4$ at $m\le 7$ points in general position on a 
			codimension $2$ linear section $Q^2 \cong S\subset Q^4$, 
			and $\sF$ is induced by the pencil of hyperplanes sections of $Q^4\subset \p^5$ with base locus $S$.

\end{enumerate} 
\end{thm}

\begin{rem} \ 
\begin{itemize}
	\item Foliations $\sG$ that appear in Theorem~\ref{main_thm_rho>1}(2) %when $r=3$ 
		are classified in Proposition~\ref{proposition:P-bdle_over_curve_(2)}.
	\item Foliations $\sG$ that appear in Theorem~\ref{main_thm_rho>1}(4) are classified in Remark~\ref{rem:flat_connection}, 
		Proposition~\ref{proposition:P-bdle_over_S_hirzebruch}, 
		Proposition~\ref{proposition:P-bdle_over_S_plane}, and 
		Proposition~\ref{proposition:P-bdle_over_S_plane_2}.
\end{itemize}
\end{rem}

This paper is organized as follows. 

In Section~\ref{section:foliations}, we review basic definitions and results 
about holomorphic foliations and Fano foliations. 

Section~\ref{section:rho=1} is devoted to Mukai foliations on manifolds with Picard number $1$.
First we show that if $X$ is a manifold with $\rho(X)=1$ admitting a codimension $1$ Mukai foliation $\sF$, then $X$
is a Fano manifold with index $\iota_X\ge \dim(X)-1$ (Lemma~\ref{lemma:bound_on_index2}).
The proof of Theorem~\ref{main_thm_rho=1} relies on the classification of such manifolds, which is reviewed 
in Section~\ref{Fano_manifolds_high_index}.
The proof distinguishes two cases, depending on whether or not $\sF$ is semi-stable.
When $\sF$ is not semi-stable, then it contains a codimension $2$ del Pezzo subfoliation $\sG$.
Such foliations are classified in Section~\ref{section:del_pezzo_codim2}
(Theorem~\ref{thm:codim2_delPezzo}).

Section~\ref{section:rho>1}  is devoted to Mukai foliations on manifolds with Picard number $>1$.
The existence of a codimension $1$ Mukai foliation $\sF$ on a manifold $X$ with $\rho(X)>1$ 
implies the existence of an extremal ray in $\NE(X)$ with large length. 
We use adjunction theory to classify all possible contractions of such extremal rays
(Theorem~\ref{tironi}). 
In order to prove Theorem~\ref{main_thm_rho>1}, 
we analyze the behavior of the foliation $\sF$ 
with respect to the contraction of a large extremal ray.
This is done separately for each type of contraction. 
Section~\ref{subsection:P-bdles/curves} deals with projective space bundles over curves. 
Section~\ref{subsection:Q-bdles/curves} deals with quadric bundles over curves. 
Section~\ref{subsection:P-bdles/surfaces} deals with projective space bundles over surfaces.
Birational contractions are treated in Section~\ref{subsection:proof_main_rho>0}.

\

\noindent {\bf Notation and conventions.}
We always work over the field ${\mathbb C}$ of complex numbers. 
Varieties are always assumed to be irreducible.
We denote by $\textup{Sing}(X)$ the singular locus of a variety $X$.

Given a sheaf $\sF$ of $\sO_X$-modules on a variety $X$, we denote by $\sF^{*}$ the sheaf $\sH\hspace{-0.1cm}\textit{om}_{\sO_X}(\sF,\sO_X)$.
If $r$ is the generic rank of $\sF$, then we denote by $\det (\sF)$ the sheaf $(\wedge^r \sF)^{**}$.
If $\sG$ is another sheaf of $\sO_X$-modules on $X$, then we denote by  $\sF[\otimes]\sG$ the sheaf $(\sF\otimes\sG)^{**}$.

If $\sE$ is a locally free sheaf of $\sO_X$-modules on a variety $X$, 
we denote by $\p_X(\sE)$ the Grothendieck projectivization $\textup{Proj}_X(\textup{Sym}(\sE))$,
and by $\sO_{\p}(1)$ its tautological line bundle.

If $X$ is a normal variety and $X\to Y$ is any morphism, we denote by
$T_{X/Y}$ the sheaf $(\Omega_{X/Y}^1)^*$. In particular, $T_X=(\Omega_{X}^1)^*$.

If $X$ is a smooth variety and $D$ is a reduced divisor on $X$ with simple normal crossings support, 
we denote by $\Omega_X^1(\textup{log }D)$ the sheaf of differential $1$-forms with logarithmic poles
along $D$, and by $T_X(-\textup{log }D)$ its dual sheal $\Omega_X^1(\textup{log }D)^*$. Notice that
$\det(\Omega_X^1(\textup{log }D))\cong\sO_X(K_X+D)$.

We denote by $Q^n$ a (possibly singular) quadric hypersurface in $\p^{n+1}$.

Given line bundles $\sL_1$ and $\sL_2$ on varieties $X$ and $Y$, we denote by $\sL_1\boxtimes\sL_2$ the line bundle
$\pi_1^*\sL_1\otimes \pi_2^* \sL_2$ on $X\times Y$, where $\pi_1$ and $\pi_2$ are the projections onto $X$ and $Y$,
respectively. 

Let $L$ be a Cartier divisor on a projective variety.
We denote by $\textup{Bs}(L)$ the base locus of the complete linear system $|L|$.

\

\noindent {\bf Acknowledgements.}
Much of this work was developed during the authors' visits to IMPA and Institut Fourier.
We would like to thank both institutions for their support and hospitality.

%%%%%%%%%%%%%%%%%%%%%%%%%%%%%%%%%%%%%%%%%%%%%
%
%			SECTION 2
%
%%%%%%%%%%%%%%%%%%%%%%%%%%%%%%%%%%%%%%%%%%%%%

\section{Preliminaries}\label{section:foliations} 

\subsection{Foliations}

\begin{defn}
A \emph{foliation} on  a normal  variety $X$ is a (possibly zero) coherent subsheaf $\sF\subsetneq T_X$ such that
\begin{itemize}
	\item $\sF$ is closed under the Lie bracket, and
	\item $\sF$ is saturated in $T_X$ (i.e., $T_X / \sF$ is torsion free).
\end{itemize}
The \emph{rank} $r_{\sF}$ of $\sF$ is the generic rank of $\sF$.
The \emph{codimension} of $\sF$ is defined as $q_{\sF}:=\dim(X)-r_{\sF}\ge 1$. 

The inclusion $\sF\into T_X$ induces a nonzero map
$$
\eta: \ \Omega_X^{r_{\sF}}=\wedge^{r_{\sF}}(\Omega_X^1) \to \wedge^{r_{\sF}}(T_X^*) 
\to \wedge^{r_{\sF}}(\sF^*) \to \det(\sF^*).
$$
The \emph{singular locus of $\sF$} is the singular scheme of this map.
I.e., it is the closed subscheme of $X$ whose ideal sheaf is the image of
the induced map $\Omega^{r_{\sF}}_X[\otimes] \det(\sF)\to \sO_X$.  

A closed subvariety $Y$ of $X$ is said to be \emph{invariant} by $\sF$ if 
it is not contained in the singular locus of $\sF$, and
the restriction $\eta_{|Y} : {\Omega^{r_{\sF}}_X}_{|Y}\to \det(\sF^*)_{|Y}$ factors through the natural map
${\Omega^{r_{\sF}}_X}_{|Y}\to\Omega^{r_{\sF}}_Y$.
\end{defn}

\begin{say}[Foliations defined by $q$-forms] \label{q-forms}
Let $\sF$ be a codimension $q$ foliation on an $n$-dimenional normal variety $X$.
The \emph{normal sheaf} of $\sF$ is $N_\sF:=(T_X/\sF)^{**}$.
The $q$-th wedge product of the inclusion
$N^*_\sF\into (\Omega^1_X)^{**}$ gives rise to a nonzero global section 
 $\omega\in H^0\big(X,\Omega^{q}_X[\otimes] \det(N_\sF)\big)$
 whose zero locus has codimension at least $2$ in $X$. 
Such $\omega$ is \emph{locally decomposable} and \emph{integrable}.
To say that $\omega$ is locally decomposable means that, 
in a neighborhood of a general point of $X$, $\omega$ decomposes as the wedge product of $q$ local $1$-forms 
$\omega=\omega_1\wedge\cdots\wedge\omega_q$.
To say that it is integrable means that for this local decomposition one has 
$d\omega_i\wedge \omega=0$ \ $\forall i\in\{1,\ldots,q\}$. 

Conversely, let $\sL$ be a reflexive sheaf of rank $1$ on $X$, $q\ge 1$, and 
$\omega\in H^0(X,\Omega^{q}_X[\otimes] \sL)$ a global section
whose zero locus has codimension at least $2$ in $X$.
Suppose that $\omega$  is locally decomposable and integrable.
Then  one defines 
a foliation of rank $r=n-q$ on $X$ as the kernel
of the morphism $T_X \to \Omega^{q-1}_X[\otimes] \sL$ given by the contraction with $\omega$. 
These constructions are inverse of each other. 
 \end{say}

\begin{say}[Foliations described as pullbacks] \label{pullback_foliations}

Let $X$ and $Y$ be normal varieties, and $\varphi:X\map Y$ a dominant rational map that restricts to a morphism $\varphi^\circ:X^\circ\to Y^\circ$,
where $X^\circ\subset X$ and  $Y^\circ\subset Y$ are smooth open subsets.

Let $\sG$ be a codimension $q$ foliation on $Y$ defined by a twisted $q$-form
$\omega\in H^0\big(Y,\Omega^{q}_Y[\otimes] \det(N_\sG)\big)$.
Then $\omega$ induces a nonzero twisted $q$-form 
$\omega_{X^\circ}\in 
H^0\Big(X^\circ,\Omega^{q}_{X^\circ}[\otimes] (\varphi^\circ)^*\big(\det(N_\sG)_{|Y^\circ}\big)\Big)$ which in turn defines a codimension $q$ foliation $\sF^\circ$ on $X^\circ$. We say that
the saturation $\sF$ of $\sF^\circ$ in $T_X$
\emph{is the pullback of $\sG$ via $\varphi$},
and write $\sF=\varphi^{-1}\sG$.

Suppose that $X^\circ$ can be taken so that $\varphi^\circ$ is an equidimensional morphism. Let 
$(B_i)_{i\in I}$
be the (possibly empty) set of hypersurfaces in $Y^\circ$ contained in the set of critical values of $\varphi^\circ$ and invariant by $\sG$. A straightforward computation shows that 
\begin{equation}\label{pullback_fol}
N_{\sF^\circ} \ \cong \ (\varphi^\circ)^*N_{\sG_{|Y^\circ}}
\otimes\sO_{X^\circ}\Big(\sum_{i\in I}\big((\varphi^\circ)^*B_i\big)_{red}-(\varphi^\circ)^*B_i\Big).
\end{equation}

Conversely, let $\sF$ be a foliation on $X$, and suppose that the general fiber of $\varphi$ is tangent to $\sF$. This means that, for a general point $x$ on a general fiber $F$ of $\varphi$,
the linear subspace $\sF_x\subset T_xX$ determined by the inclusion $\sF\subset T_X$ 
contains $T_xF$. Suppose moreover that $\varphi^\circ$ is smooth 
with connected fibers. Then, by \cite[Lemma 6.7]{fano_fols}, there is a  holomorphic foliation $\sG$ on $Y$
such that $\sF=\varphi^{-1}\sG$. 
Suppose that 
$X^\circ$
can be taken so that  $\codim_X(X\setminus X^\circ)\ge 2$.
Denote by $T_{X/Y}$ the saturation of $T_{X^\circ/Y^\circ}$ in $T_X$, and by $\varphi^*\sG$ an extension of
$(\varphi^\circ)^*\sG_{|Y^\circ}$ to $X$.
Then \eqref{pullback_fol} gives
\begin{equation} \label{K_pullback_fol}
\det(\sF)\cong \det(T_{X/Y})[\otimes] \det(\varphi^*\sG).
\end{equation}
\end{say}

\begin{defn}
Let $\sF$ be a foliation on a normal projective variety $X$.
The \textit{canonical class} $K_{\sF}$ of $\sF$ is any Weil divisor on $X$ such that  $\sO_X(-K_{\sF})\cong \det(\sF)$. 
\end{defn}

\begin{say}[Restricting foliations to subvarieties] \label{restricting_fols}
Let $X$ be a smooth projective variety, and  $\sF$ a 
codimension $q$ foliation on $X$ defined by a twisted $q$-form $\omega\in H^0\big(X,\Omega^{q}_X\otimes \det(N_\sF)\big)$.
Let $Z$ be a smooth subvariety with normal bundle $N_{Z/X}$.
Suppose that the restriction of $\omega$ to $Z$ is nonzero.
Then it induces a nonzero  twisted $q$-form
$\omega_Z\in H^0\big(Z,\Omega^{q}_Z\otimes \det(N_\sF)_{|Z}\big)$, 
and a codimension $q$ foliation $\sF_Z$ on $Z$.
There is a maximal effective divisor $B$ on $Z$ such that 
$\omega_Z\in H^0\big(Z,\Omega^{q}_Z\otimes \det(N_\sF)_{|Z}(-B)\big)$.
A straightforward computation shows that 
$$
\sO_Z\big(K_{\sF_Z}\big) \ \cong \ \det(N_{Z/X})({K_{\sF}}_{|Z}-B).
$$
\end{say}

\begin{defn}
Let $X$ be normal variety. A foliation $\sF$ on $X$ is said to be
\emph{algebraically integrable} if  the leaf of $\sF$ through a general point of $X$ is an algebraic variety. 
In this situation, by abuse of 
notation we often use the word \textit{leaf} to mean the closure in $X$ of a leaf of $\sF$. 
\end{defn}

\begin{say}[{\cite[Lemma 3.2]{fano_fols}}]\label{lemma:leaffoliation} \label{notation:family_leaves}
Let $X$ be normal projective variety, and $\sF$ an algebraically integrable foliation on $X$.
There is a unique irreducible closed subvariety $W$ of $\Chow(X)$ 
whose general point parametrizes the closure of a general leaf of $\sF$
(viewed as a reduced and irreducible cycle in $X$). In other words, if 
$U \subset W\times X$ is the universal cycle, with universal morphisms
$\pi:U\to W$ and $e:U\to X$,
then $e$ is birational, and, for a general point $w\in W$, 
$e\big(\pi^{-1}(w)\big) \subset X$ is the closure of a leaf of $\sF$.

We call the normalization $\tilde W$ of $W$ the \emph{space of leaves} of $\sF$,
and the induced rational map $X\dashrightarrow \tilde W$
a \emph{rational first integral for $\sF$}.
\end{say}

We end this subsection with a useful criterion of algebraic integrability for foliations.

\begin{thm}[{\cite[Theorem 0.1]{bogomolov_mcquillan01}, \cite[Theorem 1]{kebekus_solaconde_toma07}}] \label{thm:BM}
Let $X$ be a normal complex projective variety, and $\sF$ a  foliation on $X$.
Let $C \subset X$ be a complete curve disjoint from the singular loci of $X$ and $\sF$.
Suppose that the restriction $\sF_{|C}$ is an ample vector bundle on $C$.
Then the leaf of $\sF$ through any point of $C$ is an algebraic variety, and the 
leaf of $\sF$ through a general point of $C$ is rationally connected.
\end{thm}

\subsection{Fano foliations}

\begin{defn}
Let $\sF$ be a foliation on a normal projective variety $X$.
We say that $\sF$ is a \emph{Fano foliation} (respectively  \emph{$\bQ$-Fano foliation})
if $-K_{\sF}$ is an ample Cartier (respectively $\bQ$-Cartier) divisor on $X$.

The \emph{index} $\iota_{\sF}$ of a Fano foliation $\sF$ on $X$ 
is the largest integer dividing $-K_{\sF}$ in $\Pic(X)$. 
We say that a Fano foliation $\sF$ is a \emph{del Pezzo foliation} if $\iota_{\sF} = r_{\sF}-1$.
We say that it is a \emph{Mukai foliation} if $\iota_{\sF} = r_{\sF}-2$.
\end{defn}

The existence of a $\bQ$-Fano foliation on a variety $X$ imposes strong restrictions on $X$.

\begin{thm}[{\cite[Theorem 1.4]{codim_1_del_pezzo_fols}}]\label{Thm:KX-KF_not_nef}
Let $X$ be a klt projective variety, and $\sF\subsetneq T_X$ a $\bQ$-Fano foliation.
Then $K_X-K_\sF$ is not pseudo-effective.
\end{thm}

Suppose that a complex projective manifold $X$ admits a Fano foliation $\sF$.
By Theorem~\ref{Thm:KX-KF_not_nef},
$K_X$ is not pseudo-effective, and hence $X$ is uniruled by \cite{bdpp}. 
So we can consider a \emph{minimal dominating family of rational curves} on $X$.
This is an  irreducible component $H$ of $\rat(X)$ such that 
\begin{itemize}
	\item the curves parametrized by $H$ sweep out a dense subset of $X$, and
	\item for a general point $x\in X$, the subset of $H$ parametrizing curves through $x$ is proper.
\end{itemize}
To compute the intersection number $-K_{\sF}\cdot \ell$, where $\ell$ is a general curve from the family $H$,
we will use the following observations.

\begin{lemma}\label{lemma:bound_on_pseudo_index_bis}
Let $X$ be a complex projective manifold, and $\sF$ a codimension one foliation on $X$.
Let $C\subset X$ be a curve not contained in the singular locus of $\sF$, 
and denote by $g$ its geometric genus.
If $C$ is not tangent to $\sF$, then
$-K_\sF\cdot C \le -K_X\cdot C  + 2g-2$.
\end{lemma}

\begin{proof}
Set $n:=\dim(X)$, and let $\omega\in H^0\big(X,\Omega^1_{X}\otimes \det(N_\sF)\big)$
be a $1$-form 
defining $\sF$, as in \ref{q-forms}.
Consider the normalization morphism $f:\tilde C\to C\subset X$. The 
pullback of $\omega$ to $\tilde C$ yields a nonzero $1$-form 
$\tilde \omega \in H^0\big(\tilde{C},\Omega^1_{\tilde{C}}\otimes f^*\det(N_\sF)\big)$.
Thus 
$\deg\big(\Omega^1_{\tilde{C}}\otimes f^*\det(N_\sF)\big) \ge 0$,
proving the lemma.
\end{proof}

\begin{lemma}\label{lemma:bound_on_pseudo_index}
Let $X$ be a uniruled complex projective manifold, and $\sF$ a foliation on $X$.
Let $\ell\subset X$ be a general member of a  minimal dominating family of rational curves on $X$.
If $\ell$ is not tangent to $\sF$, then
$-K_\sF\cdot \ell \le -K_X\cdot \ell -2$.
\end{lemma}

\begin{proof}
Set $n:=\dim(X)$.
Consider the normalization morphism $f:\p^1\to \ell\subset X$.
By \cite[IV.2.9]{kollar96},
$f^*T_X\cong \sO_{\p^1}(2)\oplus \sO_{\p^1}(1)^{\oplus d}\oplus
\sO_{\p^1}^{\oplus (n-d-1)}$, where $0\le d=-K_X\cdot \ell-2\le n-1$.
Write $f^*\sF\cong \oplus_{i=1}^{r_{\sF}}\sO_{\p^1}(a_i)$.
Since $\ell$ is general, $f^*\sF$ is a subbundle of $f^*T_X$, and 
since $\ell$ is not tangent to $\sF$,
the inclusion $f^*\sF  \into  f^*T_X$ induces an inclusion
$$
\oplus_{i=1}^{r_{\sF}}\sO_{\p^1}(a_i)\cong f^*\sF  \into  f^*T_X/T_{\p^1}\cong \sO_{\p^1}(1)^{\oplus d}\oplus \sO_{\p^1}^{\oplus (n-d-1)}.
$$ 
Thus $a_i\le 1$ for $1\le i\le r_{\sF}$, and
$$
-K_\sF\cdot \ell \ = \\ \sum_{i=1}^{r_{\sF}}a_i \ \le \\ d \ = \\ -K_X\cdot \ell-2.
$$
This completes the proof of the lemma.
\end{proof}

\begin{defn}\label{log_leaf} 
Let $\sF$ be an algebraically integrable foliation on a complex projective manifold $X$.
Let $i:\tilde F\to  X$ be the normalization of the closure of a general leaf of $\sF$. 
There is an effective divisor $\tilde \Delta$ on $\tilde F$ such that
$K_{\tilde F}  +  \tilde \Delta \sim i^*K_{\sF} $ (\cite[Definition 3.4]{fano_fols}). 
The pair $( \tilde F,  \tilde \Delta)$ is called a \emph{general log leaf} of $\sF$.  
\end{defn}

In \cite{fano_fols}, we applied  the notions of singularities of pairs, 
developed in the context of the minimal model program, to the log leaf $( \tilde F,  \tilde \Delta)$ . 
The case when $( \tilde F,  \tilde \Delta)$ is \emph{log canonical} is specially interesting. 
We refer to  \cite[section 2.3]{kollar_mori} for the definition of log canonical pairs.
Here we only remark that if $\tilde F$ is smooth and $\tilde \Delta$ is a reduced simple normal crossing divisor, then 
$(\tilde F,\tilde \Delta)$ is log canonical.

\begin{prop}[{\cite[Proposition 5.3]{fano_fols} }]\label{prop:common_pt}
Let $\sF$ be an algebraically integrable Fano foliation on a complex projective manifold $X$.
Suppose that the general log leaf of $\sF$  is log canonical.
Then there is a common point contained in the closure of a general leaf of $\sF$.
\end{prop}

\subsection{Fano Foliations with large index on $\p^n$ and $Q^n$}\label{fols_in_p^n} \label{fols_in_Q^n}  \

\smallskip

Jouanolou's classification of codimension $1$ foliations on $\p^n$ of degree $0$ and $1$ 
has been generalized to arbitrary rank in \cite{cerveau_deserti} and \cite{lpt3fold}, respectively.
The \emph{degree} $\deg(\sF)$ of a foliation $\sF$ on $\p^n$ is defined as the degree 
of the locus of tangency of $\sF$ with a general linear subspace $\p^{n-r_{\sF}}\subset \p^n$. 
By \ref{q-forms}, a codimension $q$ foliation on $\p^n$ of degree $d$ is given by a twisted $q$-form
$\omega\in H^0\big(\p^n,\Omega^{q}_{\p^n}(q+d+1)\big)$. 
One easily checks that 
$$
\deg(\sF)=\deg(K_{\sF})+r_{\sF}.
$$

\begin{say}[{ \cite[Th\'eor\`eme 3.8]{cerveau_deserti}}] \label{cerveau_deserti}
A codimension $q$ foliation  of degree $0$ on $\p^n$
is induced by a  linear projection $\p^n \dashrightarrow \p^q$.
\end{say}

\begin{say}[{\cite[Theorem 6.2]{lpt3fold}}]  \label{lpt3fold}
A codimension $q$ foliation $\sF$ of degree $1$ on $\p^n$
satisfies one of the following conditions. 
\begin{itemize}
\item $\sF$ is induced by a dominant rational  map $\p^n\dashrightarrow \p(1^{q},2)$,
defined by $q$ linear forms $L_1,\ldots,L_q$
and one quadratic form $Q$; or
\item $\sF$ is the linear pullback of a foliation on $\p^{q+1}$ induced by a global 
holomorphic vector field.
\end{itemize}

In the first case, $\sF$ is induced by the $q$-form on $\mathbb{C}^{n+1}$
\begin{multline*}
\Omega  =  \sum_{i=1}^q(-1)^{i+1}L_idL_1\wedge\cdots\wedge \widehat{dL_i}\wedge\cdots \wedge dL_q\wedge dQ +(-1)^{q}2QdL_1\wedge\cdots\wedge dL_q\\
 =  (-1)^q \Big(\sum_{i=q+1}^{n+1}L_j\frac{\partial Q}{\partial L_i}\Big)dL_1\wedge\cdots\wedge dL_q\\
+\sum_{i=1}^{q}\sum_{j=q+1}^{n+1}(-1)^{i+1}L_i\frac{\partial Q}{\partial L_j}
dL_1\wedge\cdots\wedge \widehat{dL_i}\wedge\cdots \wedge dL_q\wedge dL_j,
\end{multline*}
where 
$L_{q+1},\ldots,L_{n+1}$ are linear forms such that 
$L_{1},\ldots,L_{n+1}$ are linearly independent.
The singular locus of $\sF$ is the union of the quadric 
$\{L_1=\cdots=L_q=Q=0\}\cong Q^{n-q-1}$ and the linear subspace 
$\{\frac{\partial Q}{\partial L_{q+1}}=\cdots =\frac{\partial Q}{\partial L_{n+1}}=0\}$. 

In the second, case the singular locus of $\sF$ is the union of linear
subspaces of codimension at least $2$ containing the center $\p^{n-q-2}$ of the projection.
\end{say}

\begin{say}[{\cite[Proposition 3.17]{fano_fols_2}}]\label{lemma:fols_in_Q^n} 
A codimension $q$ del Pezzo foliation on a smooth quadric hypersurface $Q^n \subset \p^{n+1}$ 
is induced by the restriction of a linear projection $\p^{n+1} \dashrightarrow \p^{q}$.
\end{say}

%%%%%%%%%%%%%%%%%%%%%%%%%%%%%%%%%%%%%%%%%%%%%
%
%			SECTION 3
%
%%%%%%%%%%%%%%%%%%%%%%%%%%%%%%%%%%%%%%%%%%%%%

\section{Codimension $1$ Mukai foliations on Fano manifolds with $\rho=1$}\label{section:rho=1}

\subsection{Fano manifolds of high index}\label{Fano_manifolds_high_index} \

\smallskip

A \emph{Fano manifold} $X$ is a complex projective manifold whose anti-canonical class $-K_X$
is ample. 
The index $\iota_{X}$ of $X$ is the largest integer dividing $-K_{X}$ in $\Pic(X)$. 
By Kobayachi-Ochiai's theorem (\cite{kobayashi_ochiai}), 
$\iota_X\le n+1$, equality holds if and only if $X\cong \p^n$, and  $\iota_X= n$ if and only if 
$X\cong Q^n\subset \p^{n+1}$.

Fano manifolds with $\iota_X = \dim X-1$ were classified by Fujita in \cite{fujita1}, \cite{fujita2} and  \cite{fujita3}. 
Those with Picard  number $1$ are isomorphic to one of the following.
\begin{enumerate}
	\item A cubic hypersurface in $\p^{n+1}$.
	\item An intersection of two hyperquadrics in $\p^{n+2}$.
	\item A linear section of the Grassmannian $G(2,5)\subset\p^9$ under the Pl\"ucker embedding.
	\item A  hypersurface of degree $4$ in the weighted projective space $\p(2,1,\ldots,1)$.
	\item A hypersurface of degree $6$ in the weighted projective space $\p(3,2,1,\ldots,1)$.
\end{enumerate}
A Fano manifold $X$ such that $(\dim X-1)$ divides $\iota_X$ is called a \emph{del Pezzo} manifold.
In this case, either $X\cong \p^3$, or $\iota_X = \dim X-1$.

Fano manifolds with $\iota_X = \dim X-2$ are called \emph{Mukai} manifolds.
Their classification was first announced in \cite{Mukai89}.
We do not include it here.
Instead, we refer to  \cite[Theorem 7]{AC_Harris} for the full list of Mukai manifolds with Picard  number $1$, and 
state below the property that we need.
This property can be checked directly for each Mukai manifold in the list.

\begin{rem}\label{Mukai_manifolds_covered_by_lines}
Let  $X$ be a Mukai manifold with $\rho(X)=1$ and $\dim X\ge 4$. Denote by 
$\sL$ the ample generator of $\Pic(X)$.
Then $X$ is covered by rational curves having degree $1$ with respect to $\sL$.
\end{rem}

If a complex projective manifold $X$ with $\rho(X)=1$ admits a Fano foliation $\sF$, then 
$X$ is a Fano manifold with index $\iota_X>\iota_{\sF}$ by Theorem~\ref{Thm:KX-KF_not_nef}.
When $\iota_{\sF}$ is high, we can improve this bound.

\begin{lemma}\label{lemma:bound_on_index2}
Let $X$ be an $n$-dimensional Fano manifold  with Picard  number $1$, $n\ge 4$,
and $\sF$ a Fano foliation on $X$.
Suppose that $\iota_{\sF}\ge n-3$.
Then $\iota_{X}\ge \iota_{\sF}+2$.
\end{lemma}

\begin{proof}
By Theorem \ref{Thm:ADK}, $\iota_{\sF}\le r_{\sF}$, and equality holds only if $X\cong \p^n$.
If $\iota_{\sF}= r_{\sF}-1=n-2$, then either $X\cong \p^n$ or $X\cong Q^n\subset \p^{n+1}$ by Theorem~\ref{Thm:codim1_dP}.
In all these cases we have $\iota_{X}\ge \iota_{\sF}+2$.
So we may assume from now on that $\iota_{\sF}= n-3$ and $r_{\sF}\in\{n-2, n-1\}$.

Let $\sL$ be the ample generator of $\Pic(X)$.
Let $H$ be a minimal dominating family of rational curves on $X$, and
$\ell$ a general curve parametrized by $H$.
By \cite{CMSB}, $-K_X\cdot \ell =n+1$ if and only if $X\cong \p^n$. 
So we may assume that $-K_X\cdot \ell \le n$.
Set $\lambda :={\sL}\cdot \ell$.

By \cite[Proposition 2]{hwang98}, $\ell$ is not tangent to $\sF$. Hence,
by Lemma \ref{lemma:bound_on_pseudo_index},
$$\lambda (n-3) = -K_\sF\cdot \ell \le -K_X\cdot \ell - 2 \le n-2.$$ 
If $\lambda >1$, then $\lambda=2$, $n=4$ and $\iota_{X}=2$,
contradicting Remark~\ref{Mukai_manifolds_covered_by_lines}.
So we conclude that $\lambda =1$, and thus 
$$\iota_{X} = -K_X\cdot \ell \ge  -K_\sF\cdot \ell +2 = \iota_{\sF}+2.$$
This completes the proof of the lemma.
\end{proof}

\subsection{Del Pezzo foliations of codimension $2$} \label{section:del_pezzo_codim2} \

\smallskip

When proving Theorem~\ref{main_thm_rho=1}, we distinguish the cases when the 
codimension $1$ Mukai foliation $\sF$ is semi-stable or not.
When $\sF$ is not semi-stable, we will show that it contains 
a codimension $2$ del Pezzo subfoliation.
The aim of this subsection is to provide a classification of these.

\begin{thm}\label{thm:codim2_delPezzo}
Let $X$ be an $n$-dimensional Fano manifold  with $\rho(X)=1$, $n\ge 4$,
and $\sG$ a codimension $2$ del Pezzo foliation on $X$.
Then the pair $(X,\sG)$ satisfies one of the following conditions.
\begin{enumerate}
	\item $X\cong \p^n$ and $\sG$ is the pullback  under a linear projection of a foliation on $\p^{3}$ induced by a global 
		vector field. 
	\item $X\cong \p^n$ and $\sG$ is induced by a rational map $\p^n\map \p(2,1,1)$ defined by one quadratic form and two
		linear forms. 
	\item $X\cong Q^n\subset \p^{n+1}$ and $\sG$ is induced by the restriction of a linear projection $\p^{n+1}\map \p^2$.
\end{enumerate}
\end{thm}

Throughout this subsection, we will use the following notation. 

\smallskip

\noindent {\bf Notation.} Let $X$ be an $n$-dimensional Fano manifold  with $\rho(X)=1$
and $\sL$ an ample line bundle on $X$ such that $\Pic(X)=\mathbb{Z}[\sL]$.
Given a sheaf of $\sO_X$-modules $\sE$ on $X$ and an integer $m$, we denote by 
$\sE(m)$ the twisted sheaf $\sE\otimes \sL^{\otimes m}$. 

\smallskip

Under the assumptions of Theorem~\ref{thm:codim2_delPezzo}, 
$\sG$ is defined by a nonzero section $\omega\in H^0\big(X, \Omega_X^2(\iota_X - n+3)\big)$
as in \ref{q-forms}.
In order to compute these cohomology groups, we will use the knowledge of 
several cohomology groups of special Fano manifolds, which we gather below.

\begin{say}[Bott's formulae]\label{bott}
Let $p,q$ and $k$ be integers, with $p$ and $q$ 
nonnegative.  Then
\begin{equation*}
h^p\big(\p^n,\Omega_{\p^n}^q(k)\big) =
\begin{cases}
\binom{k+n-q}{k}\binom{k-1}{q} & \text{for } p=0, 0\le q\le n \text{ and } k>q,\\
1 & \text{for } k=0 \text{ and } 0\le p=q\le n,\\
\binom{-k+q}{-k}\binom{-k-1}{n-q} & \text{for } p=n, 0\le q\le n \text{ and } k<q-n,\\
0 & \text{otherwise.}
\end{cases}
\end{equation*}
\end{say}

\begin{say}[{\cite[Satz 8.11]{flenner81}}]\label{flenner}
Let $X$ be a smooth $n$-dimensional complete intersection in a weighted projective space. Then 
\begin{enumerate}
	\item $h^q(X,\Omega_{X}^q) =1 $ for $0\le q\le n$, $q\neq \frac{n}{2}$.
	\item $h^p\big(X,\Omega_{X}^q(t)\big) = 0$ in the following cases
		\begin{itemize}
			\item $0<p<n$, $p+q\neq n$ and either $p\neq q$ or $t\neq 0$;
			\item  $p+q > n$ and $t>q-p$;
			\item $p+q < n$ and $t<q-p$.
		\end{itemize} 
\end{enumerate}
\end{say}

\begin{say} \label{vanishing_ADK} \label{vanishing_Q}
Let $X$ be an $n$-dimensional Fano manifold with $\rho(X)=1$. 
By \cite[Theorem 1.1]{adk08}, $h^0\big(X, \Omega_X^q(\iota_X - n+q)\big)=0$
unless $X\cong \p^n$, or $X\cong Q^n$ and $q=n$.
In particular for a smooth hyperquadric $Q=Q^n\subset \p^{n+1}$, $n\ge 3$,
$h^0\big(Q, \Omega_Q^2(2)\big)=0$.
\end{say}

\begin{say}[{\cite[Lemma 4.5]{fano_fols}}] \label{vanishing_AD}
Let $X\subset \p^{n+1}$ be a smooth hypersurface of degree $d\ge 3$. 
Suppose that $q\ge 1$ and $t\le q \le n-2$.
Then $h^0\big(X, \Omega_X^q(t)\big)=0$.
\end{say}

\begin{say}\label{restriction_exact_seqs} 
Let $Y$ be an $n$-dimensional  Fano manifold  with $\rho(Y)=1$, and $X\in \big|\sO_Y(d) \big|$ a smooth divisor. 
The following exact sequences will be used to relate foliations on $X$ with foliations on $Y$.
\begin{equation}\label{restriction2}
0 \ \to \ \Omega_Y^{q}(t-d) \ \to \ \Omega_Y^{q}(t) \ \to \ \Omega_Y^{q}(t)|_X \ \to \ 0, \ \text{ and }
\end{equation}
\begin{equation}\label{restriction1}
0 \ \to \ \Omega_X^{q-1}(t-d) \ \to \ \Omega_Y^{q}(t)|_X \ \to \ \Omega_X^{q}(t) \ \to \ 0.  
\end{equation}

By \cite[Lemma 1.2]{PW_Stability}, 
if $h^p(Y,\Omega_{Y}^{q-1}) \neq 0$ and $p+q-1 < n$, then 
the map in cohomology induced by the exact sequence \eqref{restriction1} (with $t=d$)
$$
H^p(X,\Omega_{X}^{q-1}) \ \to \ H^p\big(X,\Omega_{Y}^{q}(d)_{|X})
$$
is nonzero. 
\end{say}

\begin{say} \label{WPS}
Let $a_0,\ldots, a_n$ be positive integers such that $\gcd(a_0,\ldots, \hat a_i, \ldots a_n)=1$ for every $i\in\{0, \ldots, n\}$.
Denote by $S=S(a_0,\ldots, a_n)$ the polynomial ring $\c[x_0,\ldots, x_n]$ graded by $\deg x_i=a_i$, and
by $\p:=\p(a_0,\ldots, a_n)$ the weighted projective space $\Proj \big(S(a_0,\ldots, a_n)\big)$.
For each $t\in\mathbb{Z}$, let $\sO_{\p}(t)$ be the $\sO_{\p}$-module associated to the graded $S$-module $S(t)$.

Consider the sheaves of $\sO_{\p}$-modules
$\overline{\Omega}^{q}_{\p}(t)$ defined in \cite[2.1.5] {dolgachev} for $q, t\in\mathbb{Z}$, $q\ge 0$. 
If $U\subset \p$ denotes the smooth locus of $\p$, and $\sO_{U}(t)$ is the line bundle obtained by restricting 
$\sO_{\p}(t)$ to $U$, then ${\overline{\Omega}^{q}_{\p}(t)}_{|U}=\Omega^q_{U}\otimes \sO_{U}(t)$. 
The cohomology groups  $H^p\big(\p, \overline{\Omega}^{q}_{\p}(t)\big)$ are described in \cite[2.3.2] {dolgachev}.
We will need the following:
\begin{itemize}
	\item $h^0\big(\p, \overline{\Omega}^{q}_{\p}(t)\big)=\sum_{i=0}^q \Big((-1)^{i+q} \sum_{\#J=i}\dim_{\c}\big(S_{t-a_J}\big)\Big)$, 
	where $J\subset \{0, \ldots, n\}$ and $a_J:=\sum_{i\in J}a_i$.
	\item $h^p\big(\p, \overline{\Omega}^{q}_{\p}(t)\big)=0$ if $p\not\in \{0, q,n\}$.
\end{itemize}

Now suppose that $\p$ has only isolated singularities, let $d>0$ be such that  $\sO_{\p}(d)$ is a line bundle generated by 
global sections, and $X\in \big|\sO_{\p}(d)\big|$ a smooth hypersurface. 
We will use the cohomology groups $H^p\big(\p, \overline{\Omega}^{q}_{\p}(t)\big)$ to compute some cohomology groups 
$H^p\big(X,\Omega_{X}^q(t)\big)$.
Note that $X$ is contained in the smooth locus of $\p$, so we have an exact sequence as in \eqref{restriction1}:
\begin{equation}\label{restriction1_P}
0 \ \to \ \Omega_X^{q-1}(t-d) \ \to \ \overline{\Omega}^{q}_{\p}(t)|_X \ \to \ \Omega_X^{q}(t) \ \to \ 0.
\end{equation}
Tensoring the sequence 
$$
0 \ \to \ \sO_{\p}(-d)  \ \to \ \sO_{\p}  \ \to \ \sO_{X} \ \to \ 0.
$$
with the sheaf $\overline{\Omega}^{q}_{\p}(t)$, and noting that $\overline{\Omega}^{q}_{\p}(t)\otimes \sO_{\p}(-d) \cong \overline{\Omega}^{q}_{\p}(t-d)$,
we get an exact sequence  as in \eqref{restriction2}:
\begin{equation}\label{restriction2_P}
0 \ \to \ \overline{\Omega}^{q}_{\p}(t-d) \ \to \ \overline{\Omega}^{q}_{\p}(t) \ \to \ \overline{\Omega}^{q}_{\p}(t)|_X \ \to \ 0.
\end{equation}
\end{say}

\medskip

\begin{proof}[{Proof of Theorem~\ref{thm:codim2_delPezzo}}]
By Lemma~\ref{lemma:bound_on_index2}, $\iota_X \ge n - 1$. 
Recall the classification of Fano manifolds  of high index discussed in Subsection~\ref{Fano_manifolds_high_index}.
We will go through the manifolds in that list, and determine all codimension $2$ del Pezzo foliations on them.

\medskip

Suppose first that $X\cong \p^n$.
Then $\sG$ is a codimension $2$ foliation of degree $1$ on $\p^n$.
Such foliations are described in \ref{lpt3fold}.

\medskip

Suppose that $X$ is a smooth hyperquadric $Q=Q^n\subset \p^{n+1}$.
Codimension $2$ del Pezzo foliation on $Q^n$ 
are described in \ref{lemma:fols_in_Q^n}.

\medskip

From now on we suppose that $\iota_X = n - 1$. 
We consider the 5 possibilities for $X$ described  in Subsection~\ref{Fano_manifolds_high_index}.
If we show  that $h^0\big(X, \Omega_X^2(2)\big)=0$, then it follows from 
\ref{q-forms} that $X$ does not admit del Pezzo foliations of codimension $2$.

\

\noindent {\bf $\bullet$ $X$ is a  cubic hypersurface  in $\p^{n+1}$}. The vanishing 
 $h^0\big(X, \Omega_X^2(2)\big)=0$ follows from \ref{vanishing_AD} above.

\

\noindent {\bf $\bullet$  $X$ is the intersection of two hyperquadrics in $\p^{n+2}$}.

Let $Q$ and $Q'$ be smooth hyperquadrics in $\p^{n+2}$ such that $X=Q\cap Q'$.
Consider the exact sequences of \ref{restriction_exact_seqs}  for $Y=Q$, $d=2$, $q=2$ and $t=2$.  
They induce maps of cohomology groups: 
$$
H^0\big(Q, \Omega_Q^2(2)\big)\ \to \ H^0\big(X, \Omega_Q^2(2)_{|X}\big)\ \to  \ H^0\big(X, \Omega_X^2(2)\big).
$$
We will show that the composed map $H^0\big(Q, \Omega_Q^2(2)\big)\to H^0\big(X, \Omega_X^2(2)\big)$
is surjective. 
Since $h^0\big(Q, \Omega_Q^2(2)\big)=0$ by \ref{vanishing_Q}, the required vanishing 
$h^0\big(X, \Omega_X^2(2)\big)=0$ will follow. 

Surjectivity of the first map $H^0\big(Q, \Omega_Q^2(2)\big) \to  H^0\big(X, \Omega_Q^2(2)_{|X}\big)$ follows from 
the vanishing of $H^1\big(Q, \Omega_Q^2\big)$ granted by \ref{flenner}.
To prove surjectivity of the second map,
we consider the long exact sequence in cohomology associated to the sequence \eqref{restriction1}.
By  \ref{flenner}, $H^1\big(Q, \Omega_Q^1\big)\cong \mathbb{C}$. 
So, as we noted in \ref{restriction_exact_seqs} above, 
the map $H^1(X,\Omega_{X}^{1}) \to H^1\big(X,\Omega_{Q}^{2}(2)_{|X})$
is nonzero. 
Since  $H^1\big(X, \Omega_X^1\big)\cong \mathbb{C}$ by \ref{flenner}, we conclude that the map
$H^1(X,\Omega_{X}^{1}) \to H^1\big(X,\Omega_{Q}^{2}(2)_{|X})$ is injective, and thus the map 
$H^0\big(X, \Omega_Q^2(2)_{|X}\big) \to  H^0\big(X, \Omega_X^2(2)\big)$ is surjective.

\

\noindent {\bf $\bullet$  $X$ is a linear section of the Grassmannian $G(2,5)\subset\p^9$ of codimension $c\le 2$}.

We will show that $X$ does not admit del Pezzo foliations of codimension 2. 
By \ref{restricting_fols} and \ref{vanishing_ADK}, it is enough to prove this in the case $c=2$.

By \cite[Theorem 10.26]{fujita2}, $X$ can be described as follows. There is a plane $\p^2 \cong P \subset X$
such that the blow-up $f : Y \to X$ of $X$ along $P$ admits a morphism 
$g : Y \to \p^4$. Moreover, $g$ is the blow-up of $\p^4$ along a rational normal curve $C$ of degree 3 contained in an hyperplane $H \subset \p^4$. 
Denote by $E$ and $F$ the exceptional loci of $f$ and $g$, respectively.
Then $q(E)=H$, $f^*\sO_X(1)\cong g^*\sO_{\p^4}(2)\otimes \sO_Y(-F)$, and 
$g^*\sO_{\p^4}(1)\cong \sO_Y(E+F)$.

Suppose that $X$ admits a  codimension $2$ del Pezzo foliation 
$\sG$, which is defined by a twisted $2$-form
$\omega\in H^0\big(X, \Omega_X^2(2)\big)$.
Then $\omega$ induces a twisted $2$-form 
$\alpha \in H^0\big(Y, \Omega_Y^2\otimes f^*\sO_X(2)\big)
\cong H^0\big(Y, \Omega_Y^2\otimes g^*\sO_{\p^4}(4)\otimes \sO_Y(- 2 F)\big)$.
The restriction of $\alpha$ to $Y \setminus F$ induces a twisted $2$-form
$\tilde \alpha \in H^0\big(\p^4, \Omega_{\p^4}^2(4)\big)$ 
vanishing along $C$. 
Denote by $\tilde \sG$ the foliation on $\p^4$ induced by $\tilde \alpha$.
There are two possibilities:
\begin{itemize}
	\item Either $\tilde \alpha$ vanishes along $H$,
		and hence $\tilde \sG$ is a degree $0$ foliation on $\p^4$; or  
	\item $\tilde \sG$ is a degree $1$ foliation on $\p^4$
		containing $C$ in its singular locus.
\end{itemize}

In the first case, 
$\alpha$ vanishes along $E$, and thus 
$\alpha \in H^0\big(Y, \Omega_Y^2\otimes f^*\sO_X(2)\otimes\sO_Y(-E)\big)
\cong H^0\big(Y, \Omega_Y^2\otimes g^*\sO_{\p^4}(3)\otimes \sO_Y(- F)\big)$. 
Therefore $C$ must be tangent to $\tilde{\sG}$, which is impossible
since $\tilde \sG$ is induced by a linear projection $\p^4\map \p^2$.

To see that the second case cannot occur either, recall  the  
description of the two types of codimension $2$ degree $1$ foliations on $\p^4$
from \ref{lpt3fold}. 
In all these foliations,  any irreducible component
of the singular locus is either a linear subspace of dimension at most $2$, or a conic. This proves the claim.

\

\noindent {\bf $\bullet$  $X$ is a hypersurface of degree $4$ in the weighted projective space $\p(2,1,\ldots,1)$
or a hypersurface of degree $6$ in the weighted projective space $\p(3,2,1,\ldots,1)$
}.

Consider the exact sequences of \ref{WPS}  for $\p:= \p(2,1,\ldots,1)$, $d=4$, $q=2$ and $t=2$
(respectively $\p:= \p(3,2,1,\ldots,1)$, $d=6$, $q=2$ and $t=2$).
They induce maps of cohomology groups: 
$$
H^0\big(\p, \overline{\Omega}_{\p}^{2}(2)\big)\ \to \ H^0\big(X, \overline{\Omega}_{\p}^{2}(2)_{|X}\big)\ \to  \ H^0\big(X, \Omega_X^2(2)\big).
$$
We will show that the composed map $H^0\big(\p, \overline{\Omega}_{\p}^{2}(2)\big)\to H^0\big(X, \Omega_X^2(2)\big)$
is surjective. 
Since $h^0\big(\p, \overline{\Omega}_{\p}^{2}(2)\big)=0$ by Dolgachev's formulae described in \ref{WPS}, the required vanishing 
$h^0\big(X, \Omega_X^2(2)\big)=0$ follows. 

Surjectivity of the first map $H^0\big(\p, \overline{\Omega}_{\p}^{2}(2)\big) \to  H^0\big(X, \overline{\Omega}_{\p}^{2}(2)_{|X}\big)$ follows from 
the vanishing of  $H^1\big(\p, \overline{\Omega}_{\p}^{2}(2-d)\big)$, granted by Dolgachev's formulae.
Surjectivity of the second map $H^0\big(X, \overline{\Omega}_{\p}^{2}(2)_{|X}\big) \to  H^0\big(X, \Omega_X^2(2)\big)$ follows from 
the vanishing of  $H^1\big(X, \Omega_X^1(2-d)\big)$, granted by \ref{flenner}.
\end{proof}

\subsection{Proof of Theorem~\ref{main_thm_rho=1}} \label{proof_main_thm_rho=1} \

Let $\sF$ be a codimension $1$ Mukai foliation on an $n$-dimensional complex projective manifold  
$X\not\cong \p^n$
with $\rho(X)=1$, $n\ge 4$.
We will consider two cases, according to whether or not $\sF$ is  semi-stable.

\begin{lemma}[{\cite[Proposition 7.5]{fano_fols}}] \label{lemma:subfoliation}
Let $\sF$ be a Fano foliation on a Fano manifold  $X$ with $\rho(X)=1$.
Then one of the following holds.
\begin{enumerate}
	\item Either $\sF$ is semi-stable; or 
	\item there exits a Fano subfoliation $\sG\subsetneq \sF$
		such that  $\iota_{\sG}\ge  \iota_{\sF}$.
\end{enumerate}
\end{lemma}

Suppose first that $\sF$ is semi-stable, and denote by $\sL$ the ample generator of $\Pic(X)$.
In this case, we will use the following result to classify the possible pairs $(X,\sF)$.

\begin{lemma}[{\cite[Proposition 3.5]{lpt3fold}}] \label{lemma:pencil}
Let $X$ be a Fano manifold  with $\rho(X)=1$, and 
$\sF$ a semi-stable codimension $1$ Fano foliation on $X$.
Then $\sF$ is induced by a dominant rational map of the form  
$$
\varphi=(s_1^{\otimes m_1}:s_2^{\otimes m_2}):X\map \p^1,
$$
where $m_1$ and $m_2$ are relatively prime positive integers, and $s_1$ and $s_2$ are sections of line bundles 
$\sL_1$ and $\sL_2$ such that $\sL_1^{\otimes m_1}\cong \sL_2^{\otimes m_2}$, and 
$\sL_1\otimes \sL_2\cong \sO_X(-K_X+K_{\sF})$.
\end{lemma}

In our setting, Lemma~\ref{lemma:pencil} says that
there are positive integers $\lambda$ and $a\ge b$ such that  $a$ and $b$ relatively prime, $\lambda(a+b)=\iota_X-\iota_{\sF}$, 
and $\sF$ is induced by a pencil of  hypersurfaces generated by $b\cdot F$ and $a\cdot G$,
where $F\in \big|\sL^{\otimes \lambda a}\big|$ and $G\in \big|\sL^{\otimes \lambda b}\big|$.

If $X\cong Q^n\subset \p^{n+1}$, then $a=2$ and $b=1$. Thus $\sF$ 
is cut out by a pencil of  hyperquadrics of $\p^{n+1}$ containing a double hyperplane.

If $X$ is a del Pezzo manifold, then $a=b=1$, and $\sF$ 
is induced by a pencil in $|\sL|$. 

\medskip

From now on, we assume that $\sF$ is not semi-stable. 
By Lemma~\ref{lemma:subfoliation}, there exits a Fano subfoliation $\sG\subsetneq \sF$ 
such that  $\iota_{\sG}\ge  n-3$. 
By Theorem~\ref{Thm:ADK}, $\iota_{\sG}= n-3$, and $r_{\sG} = n-2$,
\textit{i.e.}, $\sG$ is a 
codimension $2$ del Pezzo foliation on $X$.
By Theorem~\ref{thm:codim2_delPezzo}, $X\cong Q^n\subset \p^{n+1}$,
and $\sG$ is induced by the restriction to $Q^n$ of a 
linear projection $\varphi: \p^{n+1}\map \p^{2}$.
By \eqref{K_pullback_fol}, $\sF$ is the pullback via $\varphi_{|Q^n}$ of
a foliation on $\p^2$ induced by a global vector field. 
This completes the proof of Theorem~\ref{main_thm_rho=1}.
\qed

%%%%%%%%%%%%%%%%%%%%%%%%%%%%%%%%%%%%%%%%%%%%%
%
%			SECTION 4
%
%%%%%%%%%%%%%%%%%%%%%%%%%%%%%%%%%%%%%%%%%%%%%

\section{Codimension $1$ Mukai foliations on manifolds with $\rho > 1$}\label{section:rho>1}

In this section we prove Theorem~\ref{main_thm_rho>1}. 
Our setup is the following.

\begin{assumptions}\label{assumptions}
Let $X$ be an $n$-dimensional complex projective manifold with  $\rho(X) > 1$,
and $\sF$ a codimension $1$ Mukai foliation on $X$ ($n\ge 4$).
Let $L$ be an ample divisor on $X$ such that $-K_{\sF} \sim (n-3)L$,
and set $\sL:=\O_{X}(L)$.
\end{assumptions}

Under Assumptions~\ref{assumptions}, 
Theorem~\ref{Thm:KX-KF_not_nef} implies  $K_X+(n-3)L$ is not nef.
Smooth polarized varieties $(X,L)$ satisfying this condition have been classified. 
We explain this classification in Subsection~\ref{subsection:adjunction},
and then use it in the following subsections to prove Theorem~\ref{main_thm_rho>1}.

% 4.1.  Adjunction Theory

\subsection{Adjunction Theory}\label{subsection:adjunction}

\

We will need the following classification of Fano manifolds with large index with respect to the dimension. 
For $n\ge 5$, the list follows from \cite{wis}. The classification for $n=4$ can be found in \cite[Table 12.7]{IP}.

\begin{thm}\label{Thm:Classification_Mukai}
Let $X$ be an
$n$-dimensional Fano manifold
with Picard number $\rho(X)>1$, $n\ge 4$.
Let $\sL$ be an ample line bundle such that $\sO_X(-K_X)\cong \sL^{\otimes \iota_X}$. 
\begin{itemize}

\item If $\iota_X = n-1$, then $n=4$ and $(X,\sL)\cong \big(\p^2\times \p^2, \O_{\p^2}(1)\boxtimes\O_{\p^2}(1)\big)$.

\item If $\iota_X = n-2$, then $n\in \{4,5,6\}$.

\begin{enumerate}
\item If $n=6$, then $(X,\sL)\cong \big(\p^3\times \p^3, \O_{\p^3}(1)\boxtimes\O_{\p^3}(1)\big)$.

\item If $n=5$, then one of the following holds. 
	\begin{enumerate}
		\item $(X,\sL)\cong \big(\p^2\times Q^3, \O_{\p^2}(1)\boxtimes\O_{Q^3}(1)\big)$.
		\item $(X,\sL)\cong \Big(\p_{\p^3}\big(T_{\p^3}\big), \O_{\p(T_{\p^3})}(1)\Big)$.
		\item $(X,\sL)\cong \Big(\p_{\p^3}\big(\O_{\p^3}(2)\oplus \O_{\p^3}(1)^{\oplus 2}\big), 
		\O_{\p(\O_{\p^3}(2)\oplus \O_{\p^3}(1)^{\oplus 2})}(1)\Big)$.
	\end{enumerate}

\item If $n=4$, then $X$ is isomorphic to one of the following. 
	\begin{enumerate}
		\item  $\p^1\times Y$, where $Y$ is a Fano $3$-fold with index $2$, or  $Y\cong \p^3$.
		\item  A double cover of $\p^2\times\p^2$ whose branch locus is a divisor of bidegree $(2,2)$. 
		\item  A divisor of $\p^2\times\p^3$  of bidegree $(1,2)$.
		\item  An intersection of two divisors of bidegree $(1,1)$ on  $\p^3\times\p^3$.
		\item  A divisor of $\p^2\times Q^3$  of bidegree $(1,1)$.
		\item  The blow-up of $Q^4$ along a conic $C$ which is not contained in a plane in $Q^4$.
		\item  $\p_{\p^3}(\sE)$, where $\sE$ is the null-correlation bundle on $\p^3$. 
		\item  The blow-up of $Q^4$ along a line $\ell$.
		\item  $\p_{Q^3}\big(\sO_{Q^3}(-1)\oplus\sO_{Q^3}\big)$.
		\item  $\p_{\p^3}\big(\sO_{\p^3}(-1)\oplus\sO_{\p^3}(1)\big)$.
	\end{enumerate}
\end{enumerate}
\end{itemize}
\end{thm}

\begin{say}[Nef values]\label{nef_value}
Let $X$ be  $\bQ$-factorial  terminal $n$-dimensional projective variety, and $L$ an 
ample $\bQ$-divisor on $X$.
The \emph{nef value} of $L$ is defined as 
$$
\tau(L)\ :=  \ \min\big\{t\ge 0\ \big| \ K_X+tL \text{ is nef } \big\}.
$$
It is a rational number by the rationality theorem  (\cite[Theorem 3.5]{kollar_mori}).
By the basepoint free theorem (\cite[Theorem 3.7.3]{kollar_mori}), for $m$ sufficiently large and divisible, the linear system $\big|m\big(K_X+\tau(L)L\big)\big|$ 
defines a morphism $\varphi_L:X\to X'$ with connected fibers onto a normal variety. 
We refer to $\varphi_L$ as the \emph{nef value morphism} of the polarized variety $(X,L)$.
\end{say}

The next theorem summarizes the classification of smooth polarized varieties $(X,L)$ such that 
$K_X+(n-3)L$ is not nef, i.e., $\tau(L)>n-3$.

\begin{thm}\label{tironi}
Let $(X,L)$ be an $n$-dimensional smooth polarized variety, with $\rho(X)>1$ and $n\ge 4$.
Set $\sL:=\O_{X}(L)$.
Suppose that $\tau(L)>n-3$.
Then $\tau(L)\in \{n-2, n-1, n\}$, unless $\big(n,\tau(L)\big)\in \big\{(5,\frac{5}{2}),(4,\frac{3}{2}),(4,\frac{4}{3}) \big\}$.
\begin{enumerate}
	\item Suppose that $\tau(L)=n$. 
		Then $\varphi_L$ makes $X$ a $\p^{n-1}$-bundle  over  a smooth curve $C$, and for a general fiber $F\cong \p^{n-1}$ of
		$\varphi_L$, $\sL_{|F}\cong \sO_{\p^{n-1}}(1)$.		
	
	\item Suppose that $\tau(L)=n-1$. Then $(X,\sL)$ satisfies one of the following conditions.
	\begin{enumerate}
		\item $(X,\sL)\cong \big(\p^2\times \p^2, \O_{\p^2}(1)\boxtimes\O_{\p^2}(1)\big)$.
		\item $\varphi_L$ makes $X$ a quadric bundle over a smooth curve $C$, and for a general fiber $F\cong Q^{n-1}$ of
			$\varphi_L$, $\sL_{|F}\cong \sO_{Q^{n-1}}(1)$.
		\item  $\varphi_L$ makes $X$ a $\p^{n-2}$-bundle  over  a smooth  surface $S$, and for a general fiber $F\cong \p^{n-2}$ of
			$\varphi_L$, $\sL_{|F}\cong \sO_{\p^{n-2}}(1)$.	
		\item $\varphi_L$ is the blowup of a  smooth projective variety at finitely many points, 
			and for any component $E\cong \p^{n-1}$ of the exceptional locus of $\phi_L$, $\sL_{|E}\cong \sO_{\p^{n-1}}(1)$. 
	\end{enumerate} 

	\item Suppose that $\tau(L)=n-2$. Then $(X,\sL)$ satisfies one of the following conditions.
	\begin{enumerate}
		\item $-K_X \sim (n-2)L$, and $(X,\sL)$ is as in Theorem~\ref{Thm:Classification_Mukai}.
		\item $\varphi_L$ makes $X$ a generic del Pezzo fibration over a smooth curve $C$, and, for a general fiber $F$ of
			$\varphi_L$,  either $\Pic(F)=\mathbb{Z}[\sL_{|F}]$, or 
			$(F,\sL_{|F})\cong \big(\p^2\times \p^2, \O_{\p^2}(1)\boxtimes\O_{\p^2}(1)\big)$.
		\item $\varphi_L$ makes $X$ a generic quadric bundle over a normal surface $S$, and, for a general fiber $F\cong Q^{n-2}$ of
			$\varphi_L$, $\sL_{|F}\cong \sO_{Q^{n-2}}(1)$.
		\item $\varphi_L$ makes $X$ a generic $\p^{n-3}$-bundle  over  a normal  $3$-fold $Y$, and, for a general fiber $F\cong \p^{n-3}$ of
			$\varphi_L$, $\sL_{|F}\cong \sO_{\p^{n-3}}(1)$.		
		\item $\varphi_L:X\to X'$  is the composition of finitely many disjoint divisorial contractions. In particular, $X'$ is $\bQ$-factorial and
			terminal. (See Proposition~\ref{prop:BS} below.)
	\end{enumerate} 
	\item Suppose that $n=5$ and $\tau(L)=\frac{5}{2}$. 
		Then $\varphi_L$ makes $X$ a $\p^{4}$-bundle  over  a smooth curve $C$ and, for a general fiber $F\cong \p^{4}$ of
		$\varphi_L$, $\sL_{|F}\cong \sO_{\p^{4}}(2)$.	
	\item Suppose that $n=4$ and $\tau(L)=\frac{3}{2}$. Then $(X,\sL)$ satisfies one of the following conditions.
	\begin{enumerate}
		\item $(X,\sL)\cong \big(\p^2\times \p^2, \O_{\p^2}(2 )\boxtimes\O_{\p^2}(2)\big)$.
		\item $\varphi_L$ makes $X$ a  generic  quadric bundle over a smooth curve $C$, and for a general fiber $F\cong Q^{3}$ of
			$\varphi_L$, $\sL_{|F}\cong \sO_{Q^{3}}(2)$.
		\item  $\varphi_L$ makes $X$ a generic $\p^{2}$-bundle  over  a normal  surface $S$, and for a general fiber $F\cong \p^{2}$ 			
			of $\varphi_L$, $\sL_{|F}\cong \sO_{\p^{2}}(2)$.	
	\end{enumerate} 
	\item Suppose that $n=4$ and $\tau(L)=\frac{4}{3}$. Then $\varphi_L$ makes $X$ a $\p^{3}$-bundle  over  a smooth curve $C$, 		
		and for a general fiber $F\cong \p^{3}$ of
		$\varphi_L$, $\sL_{|F}\cong \sO_{\p^{3}}(3)$.	
\end{enumerate} 
\end{thm}

\begin{proof}
The main references for the proof of Theorem~\ref{tironi} are \cite[Chapter 7]{beltrametti_sommese} and 
\cite{andreatta_wisniewski}.

By \cite[Proposition 7.2.2, Theorems 7.2.3 and 7.2.4]{beltrametti_sommese}, 
either $\tau(L)=n$ and $(X,L)$ is as in (1) above, or $\tau(L)\le n-1$.

Suppose that $\tau(L)\le n-1$.
By \cite[Theorems 7.3.2 and 7.3.4]{beltrametti_sommese}, one of the following holds.
\begin{itemize}
\item[(2a)] $-K_X \sim (n-1)L$, and hence $(X,\sL)\cong \big(\p^2\times \p^2, \O_{\p^2}(1)\boxtimes\O_{\p^2}(1)\big)$ by Theorem~\ref{Thm:Classification_Mukai}.
\item[(2b)]  $\varphi_L$ makes $X$ a generic quadric bundle over a smooth curve $C$, and for a general fiber $F\cong Q^{n-1}$ of
			$\varphi_L$, $\sL_{|F}\cong \sO_{Q^{n-1}}(1)$.
\item[(2c)]  $\varphi_L$ makes $X$ a generic $\p^{n-2}$-bundle  over  a normal  surface $S$, and for a general fiber $F\cong \p^{n-2}$ of
			$\varphi_L$, $\sL_{|F}\cong \sO_{\p^{n-2}}(1)$.	
\item[(2d)] $\varphi_L$ is the blowup of a  smooth projective variety at finitely many points, 
			and for any component $E\cong \p^{n-1}$ of the exceptional locus of $\phi_L$, $\sL_{|E}\cong \sO_{\p^{n-1}}(1)$. 
\item[(e)] $\tau(L)\le n-2$.
\end{itemize}
In case (2b), it follows from \cite[Theorem 5.1]{andreatta_wisniewski} that  $X$ is in fact a quadric bundle over $C$.
In case (2c), it follows from \cite[Theorem 5.1]{andreatta_wisniewski} that $S$ is smooth and $X$ is in fact a $\p^{n-2}$-bundle over $S$.

If $\tau(L)= n-2$, the classification under (3) follows from  \cite[Theorem 7.5.3]{beltrametti_sommese}. 

If $\tau(L)< n-2$, then, by \cite[Theorems 7.7.2, 7.7.3, 7.7.5 and 7.7.8]{beltrametti_sommese}
$\big(n,\tau(L)\big)\in \big\{(5,\frac{5}{2}),(4,\frac{3}{2}),(4,\frac{4}{3}) \big\}$
and $(X,L)$ is as in (4--6) above.
\end{proof}

We will also need the following result.

\begin{prop}[{\cite[Theorems 7.5.3, 7.5.6, 7.7.2, 7.7.3, 7.7.5 and 7.7.8]{beltrametti_sommese}}]\label{prop:BS}
Let $(X,L)$ be an $n$-dimensional smooth polarized variety, $n\ge 4$.
Suppose that $\tau(L)=n-2$, and the nef value morphism $\varphi_L:X\to X'$ is birational.
Then $\varphi_L$  is the composition of finitely many disjoint divisorial contractions $\varphi_i:X\to X_{i}$,
with exceptional divisor $E_i$, of the following types:
\begin{itemize}
	\item $\varphi_i:X\to X_{i}$ is the blowup of a smooth curve $C_i\subset X_{i}$.
		In this case $X_{i}$ is smooth and the restriction of $\sL$ to a fiber $F\cong \p^{n-2}$ of $(\varphi_i)_{|E_i}:E_i\to C_i$
		satisfies $\sL_{|F}\cong \sO_{\p^{n-2}}(1)$.
	\item $\big(E_i, \sN_{E_i/X}, \sL_{|E_i}\big)\cong \big(\p^{n-1},  \sO_{\p^{n-1}}(-2), \sO_{\p^{n-1}}(1)\big)$.
		In this case $X_i$ is $2$-factorial. In even dimension it is  Gorenstein.
	\item $\big(E_i, \sN_{E_i/X}, \sL_{|E_i}\big)\cong \big(Q^{n-1}, \sO_{Q^{n-1}}(-1), \sO_{Q^{n-1}}(1)\big)$.
		In this case $X_i$ is singular and  factorial. 
\end{itemize}

Set $L':=(\varphi_L)_*(L)$. 
Then $K_{X'}+(n-3)L'$ is nef except in the following cases.
\begin{enumerate}
	\item $n=6$ and $\big(X',\sO_{X'}(L')\big) \cong \big(\p^6, \sO_{\p^6}(2)\big)$.
	\item $n=5$ and one of the following holds.
		\begin{enumerate}
			\item $\big(X',\sO_{X'}(L')\big) \cong \big(Q^5, \sO_{Q^5}(2)\big)$.
			\item $X'$ is a $\p^4$-bundle over a smooth curve, and the restriction of $\sO_{X'}(L')$ to a general fiber is $\sO_{\p^4}(2)$.
			\item $\big(X,\sO_{X}(L)\big) \cong \Big(\p_{\p^4}\big(\sO_{\p^4}(3)\oplus \sO_{\p^4}(1)\big), \sO_{\p}(1)\Big)$.
		\end{enumerate}
	\item $n=4$ and one of the following holds.
		\begin{enumerate}
			\item $\big(X',\sO_{X'}(L')\big) \cong \big(\p^4, \sO_{\p^4}(3)\big)$.
			\item $X'$ is a  Gorenstein del Pezzo $4$-fold and $3L'\sim_{\bQ}-2K_{X'}$.
			\item $\varphi_{L'}$ makes $X'$ a generic quadric bundle over a smooth curve $C$, and for a general fiber $F\cong Q^{3}$ of
				$\varphi_{L'}$, $\sO_F\big(L'_{|F}\big)\cong \sO_{Q^{3}}(2)$.
			\item $\varphi_{L'}$ makes $X'$ a generic $\p^{2}$-bundle  over  a normal  surface $S$, and for a general fiber $F\cong \p^{2}$ 			
				of $\varphi_{L'}$, $\sO_F\big(L'_{|F}\big)\cong \sO_{\p^{2}}(2)$.	
			\item $\big(X',\sO_{X'}(L')\big) \cong \big(Q^4, \sO_{Q^4}(3)\big)$.
			\item $\varphi_L:X\to X'$ factors through $\tilde X$, the blowup of $\p^4$ along a cubic surface contained in 
				a hyperplane $E\subset \p^4$. Denote by $\tilde E\subset \tilde X$ the strict transform of $E$,
				and by $\tilde L$ the push-forward of $L$ to  $\tilde X$.
				Then $N_{\tilde E/ \tilde X}\cong \sO_{\p^{3}}(-2)$, $\sO_{\tilde E}\big(\tilde L_{|\tilde E}\big)\cong \sO_{\p^{3}}(1)$,
				and only $\tilde E$ is contracted by $\tilde X\to X'$.
			\item $\varphi_L:X\to X'$ factors through $\tilde X$, a conic bundle over $\p^3$. 
				 Denote by $\tilde L$ the push-forward of $L$ to  $\tilde X$.
				 The morphism  $\tilde X\to X'$ only contracts a subvariety $\tilde E\cong \p^3$ such that 
				 $N_{\tilde E/ \tilde X}\cong \sO_{\p^{3}}(-2)$ and $\sO_{\tilde E}\big(\tilde L_{|\tilde E}\big)\cong \sO_{\p^{3}}(1)$.
			\item $\varphi_{L'}$ makes $X'$ a $\p^{3}$-bundle  over  a smooth curve $C$, 		
				and for a general fiber $F\cong \p^{3}$ of
				$\varphi_{L'}$, $\sO_F\big(L'_{|F}\big)\cong \sO_{\p^{3}}(3)$.	
			\item $\big(X',\sO_{X'}(L')\big) \cong \big(\p^4, \sO_{\p^4}(4)\big)$.
			\item $X'\subset \p^{10}$ is a cone over $\big(\p^3, \sO_{\p^3}(2)\big)$ and $L'\sim_{\bQ} 2 H$, where $H$ 
				denotes a hyperplane section in $\p^{10}$.
		\end{enumerate}
\end{enumerate}
\end{prop}

\begin{rem}
In \cite[Theorems 7.5.3]{beltrametti_sommese}, the description of the first type of divisorial contraction is as follows:  
$X_{i}$ is smooth, and $\varphi_i:X\to X_{i}$ contracts a smooth divisor $E_i\subset X$ onto a smooth curve $C_i\subset X_{i}$.
By \cite[Theorem 2]{luo}, $\varphi_i$ is a smooth blowup.
\end{rem}

% 4.2.  Codimension 1 Mukai foliations on projective space bundles over curves

\subsection{Codimension $1$ Mukai foliations on projective space bundles over curves}\label{subsection:P-bdles/curves}

\

In this subsection, we work under Assumptions~\ref{assumptions}, supposing moreover that $\tau(L)=n$, and thus 
$\varphi_L$ makes $X$ a $\p^{n-1}$-bundle  over  a smooth curve $C$.
We start with the following  observation, which is a special case of Proposition \ref{prop:common_pt}.
 
\begin{prop}\label{prop:Fano_fol_not_fibration}
Let $\sF$ be a codimension 1 Fano foliation on a smooth projective variety $X$.
Then $\sF$ is not the relative tangent sheaf of any surjective morphism $\pi:X\to C$ onto a smooth curve.
\end{prop}

As a consequence of Proposition~\ref{prop:Fano_fol_not_fibration}, $\sH:=T_{X/C}\cap \sF$
is a codimension $2$ foliation on $X$.
The restriction of $\sH$ to the general fiber $F\cong \p^{n-1}$ of  $\varphi_L$ inherits positivity of $\sF$, 
which allows us to describe it explicitly. 
In order to do so, we recall the description of families of degree $0$ foliations on
projective spaces from  \cite[7.8]{fano_fols}.

\begin{say}[{Families of degree $0$ foliations on $\p^m$}]\label{V_in_E}
Let $Y$ be a positive dimensional smooth projective variety, and
$\sE$ a locally free sheaf of rank $m+1\ge 2$ on $Y$. Set $X:=\p_Y(\sE)$,
denote by $\sO_X(1)$ the tautological line bundle on $X$,
by $\pi:X\to Y$ the natural projection, and by $F\cong \p^m$ a general fiber of $\pi$.
Let $\sH\subsetneq T_{X/Y}$ be a foliation of rank $r\le m-1$ on $X$, and suppose that 
$\sH|_{F}\cong \sO_{\p^m}(1)^{\oplus r}\subset T_{\p^m}$.

Let $\sV^*$ be the saturation of $\pi_*\big(\sH\otimes\sO_X(-1)\big)$ in  
$\sE^*\cong \pi_*\big(T_{X/Y}\otimes\sO_X(-1)\big)$,
and set $\sV:=(\sV^*)^*$.
Then $\sH\cong(\pi^*\sV^*)\otimes \sO_X(1)$. 
In particular, $\det (\sH)\cong \pi^*\det( \sV^*)\otimes  \sO_X(r)$.

The description of $\sH|_{F}$ as the relative tangent sheaf of a linear projection $\p^m\map \p^{m-r}$
globalizes as follows. 
Let $\sK$ be the (rank $m+1-r$) kernel of the dual map $\sE\to \sV$.
Then there exists an open subset $Y^{\circ}\subset Y$, 
with $\codim_Y(Y\setminus Y^{\circ})\ge 2$, over which we have an exact sequence
of vector bundles
$$
0\ \to \ \sK|_{Y^{\circ}} \ \to \ \sE|_{Y^{\circ}} \ \to \ \sV|_{Y^{\circ}} \ \to \ 0.
$$
This induces a relative linear projection $\varphi: \p_{Y^{\circ}}(\sE|_{Y^{\circ}}) \map  \p_{Y^{\circ}}(\sK|_{Y^{\circ}})=:Z$,
which restricts to a smooth morphism $\varphi^{\circ}:X^{\circ}\to Z$, where $X^{\circ}\subset X$ is an open subset 
with $\codim_X(X\setminus X^{\circ})\ge 2$.
The restriction of $\sH$ to $X^{\circ}$ is precisely $T_{X^{\circ}/Z}$. 
\end{say}

\begin{prop}\label{lemma:P-bdle_over_C}
Let $X$, $\sF$, $L$ and $\sL$ be as in Assumptions~\ref{assumptions}.
Suppose that $\tau(L)=n$, and thus $\varphi_L$ makes $X$ a $\p^{n-1}$-bundle  over  a smooth curve $C$.
Set $\sE:=(\varphi_L)_*\sL$, so that $X\cong \p_C(\sE)$.
Then one of the following holds.
\begin{enumerate}
	\item $C\cong \p^1$, $\sF$ is algebraically integrable, and its
	restriction to a general fiber is induced by a pencil of hyperquadrics in $\p^{n-1}$ containing a double hyperplane.
	\item There exist 
		\begin{itemize}
			\item an exact sequence
				$
				0\ \to \ \sK \ \to \ \sE \ \to \ \sV \ \to \ 0
				$
				 of vector bundles on $C$, with $\rank(\sK)=3$; 
			\item a rank 2 foliation $\sG$ on $\p_C(\sK)$, generically transverse to the natural projection $p:\p_C(\sK)\to C$,
				satisfying $\det(\sG)  \cong  p^*\big(\det(\sV)\big)$ and $r_{\sG}^a\ge 1$;			
		\end{itemize}		
		such that  $\sF$ is the pullback of $\sG$ via the induced relative linear projection $\p_C(\sE)\map \p_C(\sK)$.
		In this case,  $r_{\sF}^a\ge r_{\sF}-1$.	
	\item There exist 
		\begin{itemize}
			\item an exact sequence
				$
				0\ \to \ \sK \ \to \ \sE \ \to \ \sV \ \to \ 0
				$
				 of vector bundles on $C$, with $\rank(\sK)=2$; 
			\item a foliation by curves $\sG$ on $\p_C(\sK)$, generically transverse to the natural projection $p:\p_C(\sK)\to C$,
				and satisfying $\sG  \cong  p^*\big(\det(\sV)\big)\otimes \sO_{\p(\sK)}(-1)$;			
		\end{itemize}		
		such that  $\sF$ is the pullback of $\sG$ via the induced relative linear projection $\p_C(\sE)\map \p_C(\sK)$.
		In this case,  $r_{\sF}^a\ge r_{\sF}-1$.	
\end{enumerate}
\end{prop}

\begin{proof}
 
By Proposition~\ref{prop:Fano_fol_not_fibration}, $\sF\neq T_{X/C}$.
So $\sH:=\sF\cap T_{X/C}$ is a codimension $2$ foliation on $X$.
Set  $\sQ:=(\sF/\sH)^{**}$. It is an invertible subsheaf of $(\varphi_L)^*T_{C}$, 
and we have
\begin{equation}\label{eq1_P-bdle_over_C}
\det(\sH)\ \cong \ \det(\sF)\otimes \sQ^*.
\end{equation}

We want to describe the codimension $1$ foliation $\sH_F$ 
obtained by restricting $\sH$ to a general fiber $F\cong \p^{n-1}$ of $\varphi_L$. 
By \ref{restricting_fols}, there exists a non-negative integer $b$ such that 
$$
-K_{\sH_F} \ = \ (n-3+b) H,
$$
where $H$ denotes a hyperplane in $F\cong \p^{n-1}$.
By Theorem~\ref{Thm:ADK}, we must have $b\in\{0,1\}$.

\medskip

First we suppose that $b=0$, i.e., $\sH_F$ is a degree $1$ foliation  on $\p^{n-1}$.

Then $\sQ|_F\cong  \sO_{\p^{n-1}}$, and thus $\sQ\cong (\varphi_L)^*\sC$ for some line bundle 
$\sC\subset T_C$ on $C$.
Recall that there are two types codimension $1$ degree $1$ foliations on $\p^{n-1}$:
\begin{itemize}
	\item[(i)] either $\sH_{F}$ is induced by pencil of hyperquadrics containing a double hyperplane, or 
	\item[(ii)] $\sH_{F}$  is the linear pullback of a foliation on $\p^{2}$ induced by a global 
			holomorphic vector field. 
\end{itemize}

\medskip 

Suppose that we are in case (i). 
Then $\sH$ is algebraically integrable, and its general log leaf is $(Q^{n-2}, H)$, where $Q^{n-2}\subset F\cong \p^{n-1}$
is an irreducible (possibly singular) hyperquadric, and $H$ is a hyperplane section.
Note that $(Q^{n-2}, H)$ is log canonical, unless $Q^{n-2}$ is a cone over a conic curve 
and $H$ is a tangent hyperplane through the $(n-4)$-dimensional vertex.
The latter situation falls under case (ii), treated below. So we may assume that the
general log leaf of $\sH$ is log canonical.
By Proposition~\ref{prop:common_pt}, $\det(\sH)$ cannot be ample.
By \eqref{eq1_P-bdle_over_C}, we must have $\deg(\sC)>0$, and hence $C\cong \p^1$.

Next we show that $\sF$ is algebraically integrable. 
It then follows that we are in case (1) in the statement of Proposition~\ref{lemma:P-bdle_over_C}.
Since $\sH$ is algebraically integrable, 
there is a smooth surface $S$ with a generic $\p^1$-bundle structure $p:S\to \p^1$, and a rational 
map $\psi:X\map S$ over $\p^1$ inducing $\sH$.
By  \ref{pullback_foliations}, $\sF$ is the pullback via $\psi$ of a rank $1$ foliation $\sG$ on $S$.
Moreover, there is an inclusion $p^*\sC\subset  \sG$. 
It follows from Theorem~\ref{thm:BM} that the leaves of $\sG$ are algebraic, and so are the leaves of $\sF$.

\medskip

Suppose that we are in case (ii). Then there exists a codimension $3$ foliation $\sW\subset \sH$
whose restriction to $F\cong \p^{n-1}$ is a degree $0$ foliation  on $\p^{n-1}$. 
By \ref{V_in_E}, there exists an exact sequence of vector bundles on $C$
$$
0\ \to \ \sK \ \to \ \sE \ \to \ \sV \ \to \ 0
$$
with $\rank(\sK)=3$, such that $\sW\cong \big((\varphi_L)^*\sV^*\big)\otimes \sL$ is the tangent sheaf to the relative linear projection 
$\varphi: X\cong \p_C(\sE)\map \p_C(\sK)$.
Denote by $p: \p_C(\sK)\to C$ the natural projection.
By \eqref{K_pullback_fol}, 
there is a codimension $1$ foliation $\sG$ on $\p_C(\sK)$ such that $\sF$ is the pullback of $\sG$ via $\varphi$, and
$\det(\sG) \ \cong \ p^*\big(\det(\sV)\big)$.
Note that $\det(\sV)$ is an ample line bundle on $C$.
Thus, applying Theorem~\ref{thm:BM} to a suitable destabilizing subsheaf of $\sG$, we conclude that $r_{\sG}^a\ge 1$.
We are in case (2) in the statement of Proposition~\ref{lemma:P-bdle_over_C}.

\medskip

From now on we assume that $b=1$, i.e., $\sH_F$ is a degree $0$ foliation  on $\p^{n-1}$. 
By \ref{V_in_E}, there exists an exact sequence of vector bundles on $C$
$$
0\ \to \ \sK \ \to \ \sE \ \to \ \sV \ \to \ 0
$$
with $\rank(\sK)=2$, such that $\sH\cong \big((\varphi_L)^*\sV^*\big)\otimes \sL$ is the tangent sheaf to the relative linear projection 
$\varphi: X\cong \p_C(\sE)\map \p_C(\sK)$.
Denote by $p: \p_C(\sK)\to C$ the natural projection.
By \eqref{K_pullback_fol}, 
there is a foliation by curves on $\p_C(\sK)$
$$
\sG \ \cong \ p^*\big(\det(\sV)\big)\otimes \sO_{\p(\sK)}(-1)\ \into \ T_{\p(\sK)}
$$
such that $\sF$ is the pullback of $\sG$ via $\varphi$.
We are in case (3) in the statement of Proposition~\ref{lemma:P-bdle_over_C}.
\end{proof}

Next we describe the codimension $1$ foliations on $\p_C(\sK)$ that appear in Proposition~\ref{lemma:P-bdle_over_C}(2).

\begin{prop}\label{proposition:P-bdle_over_curve_(2)}
Let
$\sK$ be a rank $3$ vector bundle on a smooth complete curve $C$, and set $Y:= \p_C(\sK)$,
with  natural projection $p:Y\to C$.
Let $\sG$ be a rank 2 foliation on $Y$, generically transverse to
$p:Y\to C$, and satisfying $\det(\sG)  \cong  p^*\sA$ for some ample line bundle $\sA$ on $C$. 
Then one of the following holds.

\begin{enumerate}
\item There exist 
		\begin{itemize}
			\item an exact sequence
				$
				0\ \to \ \sK_1 \ \to \ \sK \ \to \ \sB \ \to \ 0
				$
				 of vector bundles on $C$, with $\rank(\sK_1)=2$; 
			\item a rank 1 foliation  $\sN$ on $\p_C(\sK_1)$, generically transverse to the natural projection $p_1:\p_C(\sK_1)\to C$,
				and satisfying $\sN  \cong  p_1^*(\sB)\otimes \sO_{\p(\sK_1)}(-1)$;			
		\end{itemize}		
		such that  $\sG$ is the pullback of $\sN$ via the induced relative linear projection $Y=\p_C(\sK)\map \p_C(\sK_1)$.

\item There exist 
		\begin{itemize}
			\item  a $\p^1$-bundle $q:S\to C$;
			\item a rational map $\psi: Y\map S$ over $C$ whose restriction to a general fiber $F\cong \p^2$ of $p$ is given 
				by  a pencil of conics containing a double line;
			\item a rank 1 foliation $\sN$ on $S$, generically transverse to $q:S\to C$,
				and satisfying $\sN  \cong  q^*\big(T_C(-B)\big)$ for some effective divisor $B$ on $C$;			
		\end{itemize}		
		such that  $\sG$ is the pullback of $\sN$ via $\psi$. Moreover, the set of critical values
		of $\psi$ is invariant by $\sN$.

\item $C\cong \p^1$, and there exist
		\begin{itemize}
			\item a rational map $\psi: Y\map \p^2$, which restricts to an isomorphism on all but possibly one
			fiber of $p:Y\to C$;
			\item a rank 1 foliation $\sN$ on $\p^2$ induced by a global vector field;
		\end{itemize}		
		such that  $\sG$ is the pullback of $\sN$ via $\psi$. 	
\end{enumerate}
\end{prop}

\begin{proof}
Consider the rank 1 foliation $\sC:=\sG \cap T_{Y/C}\subsetneq T_{\p_C(\sK)}$.
It  induces a rank 1 foliation $\sC_F $ on a general fiber $F\cong \p^2$ of $p$.
By \ref{restricting_fols}, there exists a non-negative integer $b$ such that 
$$
-K_{\sC_F} \ = \ b H,
$$
where $H$ denotes a hyperplane in $F\cong \p^{2}$.
By Thereom~\ref{Thm:ADK}, we must have $b\in\{0,1\}$.

\medskip

First we suppose that $b=1$, i.e., $\sC_F$ is a degree $0$ foliation  on $\p^2$. 
The same argument used in the last paragraph of the proof of Proposition \ref{lemma:P-bdle_over_C} shows that
$(\sK,\sG)$ satisfies condition (1) in the statement of Proposition \ref{proposition:P-bdle_over_curve_(2)}.

\medskip

From now on we assume that $b=0$, i.e., $\sC_F$ is a degree $1$ foliation  on $\p^2$. 
It follows that $\sC_{|F} \cong \sO_{\p^2}$  and 
there exists an effective divisor $B$ on $C$ such that $\sC\cong p^*\big(\sA\otimes \sO_C(K_C+B)\big)$.
We will distinguish two cases, depending on whether or not $\sC$ is algebraically integrable.

\medskip

Suppose first that $\sC$ is algebraically integrable.
Then $\sC_F$ is induced by a pencil of conics containing a double line.
We will show that $\sG$  satisfies condition (2) in the statement of Proposition \ref{proposition:P-bdle_over_curve_(2)}.

Let $S$ be the space of leaves of $\sC$.
Then $S$ comes with a natural morphism onto $C$, whose general fiber 
parametrizes a pencil of conics in $F\cong \p^2$.
We conclude that  
$S\to C$ is a $\p^1$-bundle.

So $\sC$ is induced by
a rational map $\psi: Y\map S$ over $C$,
and the restriction of $\psi$ to a general fiber $F\cong \p^2$ of $p$, $\psi_{|F}:F\cong\p^2\to \p^1$, is 
given by a pencil of conics containing a double line $2\ell_F$. 
Let $R\subset Y$ be the closure of the union of the lines $\ell_F$ when $F$ runs through general fibers of $p:Y\to C$. 
Then $p_{|R}:R\to C$ is a $\p^1$-bundle.

Next we show that $R$ is the singular locus of $\psi$.
Let $F'\cong\p^2$ be a special fiber of $p$ such that $\sC_{|F'}\hookrightarrow T_{F'}$ vanishes in codimension one.
Then the foliation $\sC_{F'}$ on $F'$ induced by $\sC$ is a degree $0$ foliation on $\p^2$, and 
the cycle in $S$ corresponding to the leaf $\ell$ of $\sC_{F'}$ is $R\cap F'+\ell$. 
We conclude that $R$ is the singular locus of $\psi$.

By \ref{pullback_foliations}, 
there is a foliation by curves $\sN$ on $S$ such that $\sG$ is the pullback of $\sN$ via $\psi$.
If $\psi(R)$ is not invariant by $\sN$, then 
$\psi^*\sN\cong \sO_{\p_C(\sK)}(R)\otimes p^*\sO_C(-K_C-B)=
\sO_{\p_C(\sK)}(R)\otimes \psi^*\big(q^*\sO_C(-K_C-B)\big)$
by \eqref{pullback_fol}. Thus
$\sO_{\p_C(\sK)}(R)\cong \psi^*\big(\sN\otimes q^*\sO_C(-K_C-B)\big)$, yielding a contradiction.
Therefore $\psi(R)$ is invariant by $\sN$. Moreover,
$\psi^*\sN\cong p^*\sO_C(-K_C-B)$. 

\medskip

Suppose from now on that $\sC$ is not algebraically integrable, and hence neither is $\sG$.
We will show that $\sG$  satisfies condition (3) in the statement of Proposition \ref{proposition:P-bdle_over_curve_(2)}.

Let $\sL$ be a very ample line bundle on $Y$. 
By \cite[Proposition 7.5]{fano_fols}, there exists an algebraically integrable subfoliation by curves 
$\sM \subset \sG$,  $\sM\not\subset T_{Y/C}$, such that 
$\sM \cdot \sL^2 \ge \det(\sG) \cdot \sL^2 \ge 1$. 
Moreover the general leaf of $\sM$ is a rational curve. 
Since $\sM\not\subset T_{Y/C}$, the general leaf of $\sM$ dominates $C$, and
we conclude that $C\cong \p^1$.

Next we show that $\sM\cong p^*\sO_{\p^1}(c)$, with $c\in\{1,2\}$. 
Write $\sM\cong \sO_{\p_C(\sK)}(a)\otimes p^*\sO_{\p^1}(c)$ for some integers $a$ and $c$.
First note that, for a general line  $\ell \subset F$, 
$\sC_{|\ell}\subset \sG_{|\ell}$ is a subbundle. 
Since $\sC_{|\ell}\cong \det(\sG)_{|\ell}\cong \sO_\ell$, we must have 
$\sG_{\ell}\cong \sO_\ell\oplus \sO_\ell$. This implies that $a \le 0$.
Now observe that, since $\sG$ is not algebraically integrable, $\sM$ does not depend on the choice
of $\sL$.  
Therefore 
$\sM \cdot (\sO_{Y}(kF)\otimes \sL)^2 
\ge \det(\sG)\cdot (\sO_{Y}(kF)\otimes \sL)^2 >0$ for all $k\ge 1$.
Thus $\sM \cdot F \cdot \sL \ge 0$, and hence $a \ge 0$.
We conclude that $a=0$ and $\sM\cong p^*\sO_{\p^1}(c)$.
Since $\sM \cdot \sL^2 \ge 1$, we have $c \ge 1$. 
Since $\sM\subset p^*T_{\p^1}$, we conclude that $c\in\{1,2\}$.

If $c=2$, then $\sM$ yields a flat connection on $p$. Hence, 
$\sK\cong \sO_{\p^1}(d)\oplus\sO_{\p^1}(d)\oplus \sO_{\p^1}(d)$ for some integer $d$, and 
$\sM$ is induced by the projection $\psi: Y\cong \p^1\times \p^2 \to \p^2$.

Now suppose that $c=1$.
We may assume that $\sK$ is of the form 
$\sK\cong \sO_{\p^1}\oplus\sO_{\p^1}(-a_1)\oplus \sO_{\p^1}(-a_2)$ for integers $a_2\ge a_1 \ge 0$.
Let $\tilde C \subset Y$ be the closure of a general leaf of $\sM$.
We will show that $\tilde C$ is a section of $p$. 
Suppose to the contrary that 
$\tilde C$ has degree $\ge 2$ over $\p^1$. By \cite[Lemme 1.2 and Corollaire 1.3]{druel04}, 
$\tilde C$ has degree $2$ over $\p^1$, and $\sM$ is regular in a neighbourhood of $\tilde C$. 
In particular, we have
$\sN_{\tilde C/Y}\cong \sO_{\tilde C}\oplus\sO_{\tilde C}$.
Write $\tilde C\sim 2 \sigma +kf$, where $\sigma$ is the section of $p$ corresponding to 
the surjection $\sK\twoheadrightarrow  \sO_{\p^1}(-a_2)$, and $f$ is a line on a fiber of $p$.
Let $E\subset Y$ be the divisor corresponding to the surjection 
$\sK\twoheadrightarrow \sO_{\p^1}(-a_1)\oplus \sO_{\p^1}(-a_2)$, so that 
$\sO_Y(E)\cong \sO_{\p_C(\sK)}(1)$. 
Since the deformations of $\tilde C$ sweep out a dense open subset of $Y$, we must have 
\begin{equation} \label{E.tildeC}
E \ \cdot \ \tilde C \ = \ -2a_2 \ + \ k \ \ge \ 0. 
\end{equation}
On the other hand, since $\sN_{\tilde C/Y}\cong \sO_{\tilde C}\oplus\sO_{\tilde C}$, we have 
\begin{equation} \label{deg_NC}
0 \ = \ \deg(\sN_{\tilde C/Y}) \ = \ \deg\big((T_Y)_{|\tilde C}\big) \ - \ 2 \ = \ 2 \ + \  2a_1\ - \  4a_2 \ + \  3k. 
\end{equation}
Equations \eqref{E.tildeC} and \eqref{deg_NC} together yield a contradiction, 
proving that $\tilde C$ is a section of $p$. 
The map $\sM\cong p^*\sO_{\p^1}(1) \to p^*T_{\p^1}$ vanishes exactly along one fiber $F_0$ of
$p$. This implies that $\sM \subset T_Y$ restricts to a regular foliation (with algebraic leaves) over 
$Y\setminus F_0$. This foliation is induced by a smooth morphism 
$\psi: Y\setminus F_0 \to \p^2$, which restricts to an isomorphism on all fibers $F\neq F_0$ of $p:Y\to \p^1$. 

In either case,  by \eqref{pullback_foliations}, $\sG$ is the pullback via $\psi$ of a rank $1$ foliation 
$\sN$ on $\p^2$.
A straightforward computation shows that $\sN\cong \sO_{\p^2}$. This completes the proof of the proposition.
\end{proof}

We construct examples of foliations described in Proposition~\ref{proposition:P-bdle_over_curve_(2)}(2).

\begin{exmp}Let $C$ be a smooth complete curve, $\sA$ an ample line bundle on $C$, and $P\in C$.
Set $\sL:=\sA\otimes\sO_C(K_C+P)$,
$\sK:=\sL^{\otimes 2}\oplus\sL\oplus\sO_C$, and $\sW:=\sL^{\otimes 4}\oplus\sL^{\otimes 2}$.
Suppose that $\deg(\sL)\neq 0$.

Let $s$ be a local frame for $\sL$.
It induces local frames $(k_1,k_2,k_3)$ and $(w_1,w_2)$ for $\sK$ and $\sW$, respectively.
We view $\sW$ as a subbundle of $Sym^2\sK$ by mapping 
$w_1$ to $k_1\otimes k_1$, and $w_2$ to $k_2\otimes k_2-k_1\otimes k_3$. This gives rise to a rational map
$\psi : \p_C(\sK) \dashrightarrow \p_C(\sW)$ such that $\psi^*\sO_{\p_C(\sW)}(1)\cong \sO_{\p_C(\sK)}(2)$. 
Note that the set $\sigma$ of critical values of $\psi$ is the section of $q$ corresponding to 
$\sW=\sL^{\otimes 4}\oplus\sL^{\otimes 2} \twoheadrightarrow \sL^{\otimes 2}$.

Denote by $q : \p_C(\sW)\to C$ the natural morphism.
By \cite[Lemma 9.5]{fano_fols}, the inclusion $\sN:=q^*\big(T_C(-P)\big)\hookrightarrow p^*T_C$
lifts to an inclusion $\iota : \sN\hookrightarrow T_{\p_C(\sW)}$.
We claim that the  cokernel of $\iota$ is torsion-free, and thus it defines a foliation on $\p_C(\sW)$.
Indeed, if $T_{\p_C(\sW)}/\sN$ is not  torsion-free, then we get an inclusion 
$\sN\subset T_{\p_C(\sW)}\otimes q^*\sO_C(-P)$ (see \cite[Lemma 9.7]{fano_fols}). 
Thus $p^*T_C\cong \sN\otimes q^*\sO_C(P)  \subset T_{\p_C(\sW)}$, and 
the natural exact sequence 
$$
0\ \to \ T_{\p_C(\sW)/C} \ \to \ T_{\p_C(\sW)} \ \to \ p^*T_C \ \to \ 0
$$
splits. This implies that $\sK$ admits a flat projective connection, which is absurd.
This proves the claim.
An easy compuation shows that 
$\deg(\omega_\sigma\otimes\sN_{|\sigma})=-\deg(\sA)-\deg(\sL)<0$, and thus $\sigma$ is invariant under
$\sN$. 

Now set $\sG:=\psi^{-1}(\sN)$. Then $\sG$ is a rank 2 foliation on $\p_C(\sK)$, generically transverse to the natural projection 
$p:\p_C(\sK)\to C$, and satisfies $\det(\sG)  \cong  p^*\sA$.
\end{exmp}

Next we construct examples of foliations described in Proposition~\ref{lemma:P-bdle_over_C}(3).

\begin{exmp}
Let $C$ be a smooth complete curve and $\sV$ an ample vector bundle of rank $n-2$ on $C$.
Let $\sK_0$ be a vector bundle of rank $2$ on $C$, and suppose that $\sK_0$
does not admit a flat projective connection.
Choose a sufficiently ample line bundle $\sA$ on $C$ such that the following conditions hold:
\begin{enumerate}
	\item $\sK:=\sK_0\otimes \sA$ is an ample vector bundle;
	\item there is a nowhere vanishing section $\alpha\in H^0\big(C, T_C\otimes \det(\sV^*)\otimes \sK_0\otimes \sA\big)$; and
	\item $h^1\big(C, \det(\sV^*)\otimes Sym^3(\sK_0)\otimes \det(\sK_0^*)\otimes \sA\big)=0$.
\end{enumerate}

Set  $S:= \p_C(\sK)$, denote by $p:S\to C$ the natural projection, and by $\sO_S(1)$ the tautological line bundle.
The section $\alpha$ from condition (2) yields an inclusion 
$$
\sG \ := \ p^*\big(\det(\sV)\big)\otimes \sO_S(-1) \ \into \ p^*T_C,
$$
which does not vanish identically on any fiber of $p$. 
Notice that $\sG\otimes \sO_S(B)  \cong  p^*T_C$  
for some section $B$ of $p$.
By Lemma~\ref{lemma:lifting} below, condition (3)  implies that the inclusion 
$\sG\into p^*T_C$ can be lifted to an inclusion 
$$
\iota: \ \sG \ \into \ T_S.
$$

We claim that the  cokernel of $\iota$ is torsion-free, and thus it defines a foliation on $S$.
Indeed, if $T_S/\sG$ is not  torsion-free, then we get an inclusion 
$\sG\subset T_S(-B)$ (see \cite[Lemma 9.7]{fano_fols}). 
Thus $p^*T_C \cong \sG\otimes \sO_S(B) \subset T_S$, and 
the natural exact sequence 
$$
0\ \to \ T_{S/C} \ \to \ T_S \ \to \ p^*T_C \ \to \ 0
$$
splits. This implies that $\sK$ admits a flat projective connection, contradicting our assumption.
This proves the claim.

Now set $\sE:= \sV\oplus \sK$, $X:=\p_C(\sE)$, 
denote by $\pi:X\to C$ the natural projection, and by $\sO_X(1)$ the tautological line bundle.
Condition (1) above implies that $\sO_X(1)$ is an ample line bundle on $X$.
The natural quotient $\sE\to \sV$ defines a relative linear projection 
$\varphi:X\map S$. 
Let $\sF$ be the codimension $1$ foliation on $X$ obtained as pullback of $\sG$ via $\varphi$.
Recall that $T_{X/S}\cong \pi^*\det(\sV^*)\otimes \sO_X(1)$, and thus, by \eqref{K_pullback_fol}, 
$$
\det(\sF)\ \cong \ \sO_X(n-3),
$$
i.e., $\sF$ is a codimension $1$ Mukai foliation on $X$.
\end{exmp}

\begin{lemma}\label{lemma:lifting}
Let $\sK$ be a vector bundle of rank $2$ on a smooth projective curve $C$,
$p:S=\p(\sK)\to C$ the corresponding ruled surface, and  $\sO_S(1)$ the tautological line bundle.
Let $\sB$ be a line bundle on $C$ such that there is an inclusion 
$j : \ p^*\sB \otimes \sO_S(-1)\into p^*T_C$.
If $h^1\big(C, \sB\otimes Sym^3(\sK)\otimes \det(\sK^*)\big)=0$, then 
$j$ can be lifted to an inclusion $p^*\sB \otimes \sO_S(-1)\into T_S$.
\end{lemma}

\begin{proof}
Let $e$ be the class in $H^1(S,T_{S/C}\otimes p^*\omega_C)$
corresponding to  the exact sequence 
$$
0 \ \to \  T_{S/C} \ \to \  T_S \ \to \ p^*T_C\ \to \ 0 \ .
$$
An inclusion of line bundles $j: \sG\into p^*T_C$ extends to 
an inclusion $\sG \into T_S$ if and only if the induced section 
$j^*e\in H^1(S,T_{S/C}\otimes \sG^{-1})$ vanishes identically.

Setting $\sG:=p^*\sB \otimes \sO_S(-1)$, we get
$$
H^1(S,T_{S/C}\otimes \sG^{-1}) \ =\ H^1\Big(C, p_*\big(T_{S/C} \otimes \sO_S(1)\big)\otimes \sB\Big).
$$
Since $T_{S/C}\cong p^*\big(\det(\sK^*))\otimes  \sO_S(2)$, this gives
$$
H^1(S,T_{S/C}\otimes \sG^{-1}) \ =\ H^1\big(C, p_*\sO_S(3)\otimes \det(\sK^*) \otimes \sB\big)  \ =\ 
H^1\big(C, Sym^3(\sK)\otimes \det(\sK^*) \otimes \sB\big).
$$
The latter vanishes by assumption, and thus $j: p^*\sB \otimes \sO_S(-1)\into p^*T_C$ extends to 
an inclusion $p^*\sB \otimes \sO_S(-1) \into T_S$.
\end{proof}

% 4.3.  Codimension 1 Mukai foliations on  quadric bundles over curves

\subsection{Codimension $1$ Mukai foliations on quadric bundles over curves}\label{subsection:Q-bdles/curves}

\

In this subsection, we work under Assumptions~\ref{assumptions}, supposing moreover that $\tau(L)=n-1$ and 
$\varphi_L$ makes $X$ a quadric bundle  over  a smooth curve $C$. This is case  (2b) of Theorem~\ref{tironi}.

We start with two useful observations.

\begin{rem} \label{rem:sing_Q-bdle_over_C}
Let $\phi:X\to C$ be a quadric bundle over a smooth curve, with $X$ smooth. 
An easy computation shows that 
the (finitely many) singular fibers of $\pi$ have only isolated singularities. 
\end{rem} 

\begin{lemma}\label{lemma:line_bundle_base}
Let $T$ be a complex variety, and  $\phi : X \to T$  a flat projective morphism whose fibers are all irreducible and reduced. Let $\sQ$ be a line bundle on $X$ such that $\sQ_{|F}\cong \sO_F$ for a general fiber $F$ of $\phi$. Then there exists a line bundle $\sM$ on $T$ such that $\sQ\cong \phi^*\sM$.
\end{lemma}

\begin{proof}
Let $t\in T$ be any point, and denote by $X_t$ the corresponding fiber of $\phi$. By the semicontinuity theorem, 
$h^0(X_t,\sQ_{|X_t}) \ge 1$, and $h^0(X_t,\sQ^*_{|X_t}) \ge 1$. 
It follows that $\sQ_{|X_t}\cong \sO_{X_t}$,
since $X_t$ is irreducible and reduced. 
 By \cite[Corollary III.12.9]{hartshorne77},
$\sM:=\phi_*\sQ$ is a line bundle on $T$, and the evaluation map $\phi^*\sM=\phi^*\phi_*\sQ\to \sQ$ is an isomorphism. 
\end{proof}

\begin{prop}\label{proposition:Q-bdle_over_C}
Let $X$, $\sF$ and $L$ be as in Assumptions~\ref{assumptions}.
Suppose that $\tau(L)=n-1$, and  $\varphi_L$ makes $X$ a quadric bundle over a smooth curve $C$.
Then $C \cong \p^1$, and there exist
\begin{itemize}
	\item an exact sequence of vector bundles on $\p^1$
		$$
		0\ \to \ \sK \ \to \ \sE \ \to \ \sV \ \to \ 0,
		$$
		with $\rank(\sE)=n+1$, $\rank(\sK)=2$, and natural projections $\pi:\p_{\p^1}(\sE)\to \p^1$
		and $q:\p_{\p^1}(\sK)\to \p^1$;
	\item  an integer $b$ and a foliation by rational curves $\sG\cong q^*\big(\det(\sV)\otimes \sO_{\p^1}(b)\big)$ on $\p_{\p^1}(\sK)$;
\end{itemize}
such that $X\in \big|\sO_{\p_{\p^1}(\sE)}(2)\otimes \pi^*\sO(b)\big|$, and
$\sF$ is the pullback of $\sG$ via the restriction to $X$ of the relative linear projection $\p_{\p^1}(\sE) \map  \p_{\p^1}(\sK)$.
Moreover, one of the following holds.

\begin{enumerate}

\item $(\sE,\sK)\cong (\sO_{\p^1}(a)^{\oplus 2}\oplus\sO_{\p^1}^{\oplus 3},\sO_{\p^1}(a)^{\oplus 2})$
for some integer $a\ge 1$, and $b=2$.

\item $(\sE,\sK)\cong (\sO_{\p^1}(a)^{\oplus 2}\oplus\sO_{\p^1}^{\oplus 2}\oplus \sO_{\p^1}(1),\sO_{\p^1}(a)^{\oplus 2})$
for some integer $a\ge 1$, and $b=1$.

\item $(\sE,\sK)\cong (\sO_{\p^1}(a)^{\oplus 2}\oplus\sO_{\p^1}\oplus \sO_{\p^1}(1)^{\oplus 2},\sO_{\p^1}(a)^{\oplus 2})$ for some integer $a\ge 1$, and $b=0$.

\item $\sK\cong \sO_{\p^1}(a)^{\oplus 2}$ for some integer $a$, and
$\sE$ is an ample vector bundle of rank $5$ or $6$ with $\deg(\sE)=2+2a-b$.

\item $\sK\cong \sO_{\p^1}(a)\oplus \sO_{\p^1}(c)$ for distinct integers $a$ and $c$, and
$\sE$ is an ample vector bundle of rank $5$ or $6$ with $\deg(\sE)=1+a+c-b$.

\end{enumerate}
In particular, $n\in\{4,5\}$  and $\sF$ is algebraically integrable. 

Conversely, given $\sK$, $\sE$ and $b$ satisfying any of the conditions (1-5), 
and a smooth member $X\in \big|\sO_{\p_{\p^1}(\sE)}(2)\otimes \pi^*\sO_{\p^1}(b)\big|$, 
there exists a codimension one Mukai foliation on $X$ as described above.
\end{prop} 

\begin{proof}
Denote by $F\cong Q^{n-1}\subset \p^n$ a general (smooth) fiber of $\varphi_L$, and recall 
from Theorem~\ref{tironi}(2b) that $\sL_{|F}\cong \sO_{Q^{n-1}}(1)$.
Set  $\sE:=(\varphi_L)_*\sL$, and 
denote by $\pi:\p_{C}(\sE)\to C$ the natural projection. 
Then $X$ is a divisor of relative degree $2$ on $\p_C(\sE)$, i.e., 
$X\in \big|\sO_{\p_C(\sE)}(2)\otimes \pi^*\sB\big|$ for some line bundle $\sB$ on $C$.

By Proposition~\ref{prop:Fano_fol_not_fibration}, $\sF\neq T_{X/C}$.
So $\sH:=\sF\cap T_{X/C}$ is a codimension $2$ foliation on $X$.
Set  $\sQ:=(\sF/\sH)^{**}$. It is an invertible subsheaf of $(\varphi_L)^*T_{C}$, 
and $\det(\sH)\ \cong \ \det(\sF)\otimes \sQ^*$.
Denote by  $\sH_F$ the codimension $1$ foliation on $F\cong Q^{n-1}$
obtained by restriction of $\sH$.
By \ref{restricting_fols}, there exists a non-negative integer $d$ such that 
$-K_{\sH_F} \ = \ (n-3+d) H$,
where $H$ denotes a hyperplane section of $Q^{n-1}\subset \p^n$.
By Thereom~\ref{Thm:ADK}, we must have $d=0$. Hence,
\begin{itemize}
	\item $\sH_{F}$ is induced by a pencil of hyperplane sections
	on $Q^{n-1}\subset \p^n$ by Theorem \ref{Thm:codim1_dP}. In paticular 
	the general log leaf of $\sH$ is log canonical. 
	\item $\det(\sH)_{|F}\ \cong \ \det(\sF)_{|F}$, and thus $\sQ\cong (\varphi_L)^*\sM$ for some line bundle $\sM\subset T_C$
	by Lemma~\ref{lemma:line_bundle_base}.
\end{itemize}

By  Proposition~\ref{prop:common_pt},  $\det(\sH)$ is not ample.
Since  $\det(\sH)\ \cong \ \det(\sF)\otimes (\varphi_L)^*(\sM^*)$ and $\det(\sF)$ is ample,
 the line bundle $\sM$ has positive degree. 
Hence $C\cong \p^1$, $\sB\cong \sO_{\p^1}(b)$ for some $b\in \mathbb{Z}$,
and  $\deg(\sM)\in \{1,2\}$.

The linear span of $\Sing(\sH_{F})$ in $\p^n$ is the base locus of the pencil of hyperplanes in
$\p^n$ inducing $\sH_{F}$ on $F\cong Q^{n-1}\subset \p^n$.
So  $\sH$ is the restriction to $X$ of a foliation $\tilde \sH$ on $\p_{\p^1}(\sE)$ whose restriction to a general fiber of $\pi$ is a 
degree zero foliation on $\p^n$.
By \eqref{V_in_E}, there is a sequence of vector bundles on $\p^1$,
\begin{equation}\label{exact_sequence}
0\ \to \ \sK \ \to \ \sE \ \to \ \sV \ \to \ 0,
\end{equation}
with $\rank(\sK)=2$ and natural projection $q:\p_{\p^1}(\sK)\to \p^1$,
such that $\tilde \sH$ is induced by the relative linear projection  $\tilde \psi: \p_{\p^1}(\sE) \map  \p_{\p^1}(\sK)$.
So $\sH$ is induced by the restriction $\psi= \tilde \psi_{|X}:X\map \p_{\p^1}(\sK)$.
By Remark~\ref{rem:sing_Q-bdle_over_C}, there is an open subset $X^{\circ}\subset X$ with $\codim_X(X\setminus X^{\circ})\ge 2$
such that $\psi^{\circ}=\psi_{|X^{\circ}}:X^{\circ}\to \p_{\p^1}(\sK)$ is a smooth morphism with connected fibers. 
In particular, $\sH\cong T_{X/\p(\sK)}$, where $T_{X/\p(\sK)}$ denotes the saturation of $T_{X^\circ/\p(\sK)}$ in $T_X$.
By \ref{pullback_foliations}, $\sF$ is the pullback via $\psi$ of a rank $1$ foliation $\sG$ on $\p_{\p^1}(\sK)$.
By \eqref{K_pullback_fol}, $\sG \cong q^*\sM$ and
$$
\sL^{\otimes n-3}\cong \det(\sF)\cong \det(T_{X/\p(\sK)})\otimes (\varphi_L)^*\sM.
$$
Since $\deg(\sM)>0$, the leaves of $\sG$ are rational curves by Theorem~\ref{thm:BM}.
A straightforward computation gives $\sM\cong \det (\sV)\otimes \sO_{\p^1}(b)$, and so
\begin{equation}\label{degreeV}
\deg(\sM)=\deg(\sV)+b\in \{1,2\}.
\end{equation}

If $\deg(\sM)=2$, i.e., if $\sM\cong T_C$, then $q^*\sM \subset T_{\p_{\p^1}(\sK)}$
yields a flat connection on $q:\p_{\p^1}(\sK)\to \p^1$. Hence,
$\p_{\p^1}(\sK)\cong \p^1\times \p^1$, and $\sG$
is induced by the projection to $\p^1$ transversal to $q$.
In this case, $\sK\cong\sO_{\p^1}(a)^{\oplus 2}$ for some integer $a$.
If $\deg(\sM)=1$, then $\sK\cong \sO_{\p^1}(a)\oplus \sO_{\p^1}(c)$ for distinct integers $a$ and $c$.
This can be seen from the explicit description of the Atiyah classes in the proof of 
\cite[Theorem 9.6]{fano_fols}.

By \cite[Theorem 4.13]{campana04}, the vector bundle 
$(\varphi_L)_*\sO_{X}(K_{X/\p^1}+m L)\cong S^{m-n+1}\sE\otimes \det(\sE)\otimes \sO_{\p^1}(b)$
is nef for all $m\ge n-1$. Therefore $\sE$
is a nef vector bundle on $\p^1$, and we write $\sE\cong\sO_{\p^1}(a_1)\oplus\cdots\oplus\sO_{\p^1}(a_{n+1})$,
with $0 \le a_1\le \cdots \le a_{n+1}$.

\medskip

First suppose that $\sE$ is not ample.
Let $r\in \{1, \ldots,  n+1\}$ be the largest positive integer
such that $a_1=\cdots =a_{r}=0$.
Let $\varpi : \p_{\p^1}(\sE) \to \p\Big(H^0\big(\p_{\p^1}(\sE),\sO_{\p_{\p^1}(\sE)}(1)\big)^*\Big)$ be the morphism induced by the complete
linear system $\big|\sO_{\p_{\p^1}(\sE)}(1)\big|$. 
If $r=n+1$, then $\p_{\p^1}(\sE)\cong \p^1 \times \p^{n}$, and $\varpi$ is induced by the projection morphism
$\p^1 \times \p^{n} \to \p^{n}$.
If $r \le n$, then $\varpi$ is a birational morphism, and its  restriction to the exceptional 
locus $\textup{Exc}(\varpi)=\p_{\p^1}\big(\sO_{\p^1}(a_1)\oplus\cdots\oplus\sO_{\p^1}(a_{r})\big)\cong \p^1 \times \p^{r-1}$
corresponds to the projection $\p^1 \times \p^{r-1} \to \p^{r-1}$.
Thus, if $Z\subset \p_{\p^1}(\sE)$ is any closed subset, then 
${\sO_{\p_{\p^1}(\sE)}(1)}_{|Z}$ is ample if and only if
$Z$ does not contain any fiber of 
$\textup{Exc}(\varpi)\cong\p^1 \times \p^{r-1} \to \p^{r-1}$.
Since ${\sO_{\p_{\p^1}(\sE)}(1)}_{|X}\cong \sL$ is ample,
$X$ does not contain any fiber of 
$\textup{Exc}(\varpi)\cong\p^1 \times \p^{r-1} \to \p^{r-1}$.
Since $\big(\sO_{\p_{\p^1}(\sE)}(2)\otimes \pi^*\sO_{\p^1}(b)\big)_{|\textup{Exc}(\varpi)}
\cong \sO_{\p^1}(b) \boxtimes \sO_{\p^{r-1}}(2)$ and
$X\in \big|\sO_{\p_{\p^1}(\sE)}(2)\otimes \pi^*\sO_{\p^1}(b)\big|$, we must have $b \ge 0$ and 
\begin{equation}\label{rnb}
h^0\big(\p^1,\sO_{\p^1}(b)\big)= b+1 \ge r.
\end{equation}

If $\deg(\sV)=0$, then $\sV\cong \sO_{\p^1}^{\oplus (n-1)}$ and the exact sequence \eqref{exact_sequence} splits.
This implies that $r \ge n-1$. On the other hand, by \eqref{degreeV} and \eqref{rnb}, $r\le b+1\le 3$. 
Thus $n=4$, $r=3$, $b=2$, and $\deg(\sM)=2$.
This is case (1) described in the statement of Proposition~\ref{proposition:Q-bdle_over_C}.

If $\deg(\sV) \ge 1$, then $b\in\{0,1\}$ by \eqref{degreeV}. 

Suppose that $b=1$. Then $\deg(\sV)=1$ by \eqref{degreeV}.
Thus $\sV\cong \sO_{\p^1}^{\oplus (n-2)}\oplus \sO_{\p^1}(1)$, and 
the exact sequence \eqref{exact_sequence} splits. By \eqref{rnb}, we have $n-2 \le 2$.
This implies $n=4$, and $\deg(\sM)=2$.
This is case (2) described in the statement of Proposition~\ref{proposition:Q-bdle_over_C}.

Suppose that $b=0$. Then we must have $r=1$ by \eqref{rnb}. 
By \eqref{degreeV}, we have $\deg(\sV)\le 2$.
On the other hand, $\deg(\sV)\ge a_1 + \cdots + a_{n-1} \ge n-2$. Thus
$n=4$,  $\sV\cong \sO_{\p^1}\oplus \sO_{\p^1}(1)^{2}$, the exact sequence \eqref{exact_sequence} splits,
and $\deg(\sM)=2$.
This is case (3) described in the statement of Proposition~\ref{proposition:Q-bdle_over_C}.

\medskip

Suppose from now on that $\sE$ is an ample vector bundle on $\p^1$. 
Since 
$\deg(\sV)\ge a_1 + \cdots + a_{n-1}$, we have
$b \le 2 - (a_1 + \cdots + a_{n-1})$ by \eqref{degreeV}.

We claim that $n\in\{4,5\}$. 
Suppose to the contrary that $n\ge 6$.
Then 
$h^0\big(\p^1,\sO_{\p^1}(a_i+a_j+b)\big)=0$ 
if either $1 \le i,j \le n-2$, or
$1 \le i\le n-2$ and $j=n-1$. This implies that 
in suitable homogeneous coordinates $(x_1:\cdots:x_{n+1})$,
$F\cong Q^{n-1}\subset \p^n$ is given by equation
$$
c_{n-1}x_{n-1}^2 + x_{n}l_{n}(x_1,\ldots,x_{n+1}) + x_{n+1}l_{n+1}(x_1,\ldots,x_{n+1})=0,
$$ 
where $c_{n-1}\in\mathbb{C}$, and $l_n$ and $l_{n+1}$ are linear forms.
This contradicts the fact that $F$ is smooth, and proves the claim.
So we are in one of cases (4) and (5) of  Proposition~\ref{proposition:Q-bdle_over_C},
depending on whether $\deg(\sM)$ is $2$ or $1$, respectively.

\medskip

Now we proceed to prove the converse statement. 
Let $\sK$, $\sE$ and $b$ satisfy one of the conditions (1-5) in the statement of Proposition~\ref{proposition:Q-bdle_over_C}, and 
$X\in \big|\sO_{\p_{\p^1}(\sE)}(2)\otimes \pi^*\sO_{\p^1}(b)\big|$ a smooth member. 
Then, one  easily checks that $\sO_{\p_{\p^1}(\sE)}(1)_{|X}$
is an ample line bundle on $X$.
We shall construct a codimension 1 Mukai foliation $\sF$ on $X$
such that $\sO_X(-K_{\sF})\cong \sO_{\p_{\p^1}(\sE)}(1)_{|X}$.
First, let $\sV$ be a vector bundle of rank $n-1$ on $\p^1$ 
fitting into an exact sequence of vector bundles 
$$
0 \ \to \ \sK \ \to \ \sE \ \to \ \sV\ \to 0,
$$
and consider the induced  rational map 
$\tilde{\psi}:\p_{\p^1}(\sE) \dashrightarrow \p_{\p^1}(\sK)$. 
Let 
$\psi:=\psi_{|X} : X \dashrightarrow \p_{\p^1}(\sK)$ be the restriction of $\tilde{\psi}$ to $X$, and let
$q:\p_{\p^1}(\sK)\to \p^1$ be the 
natural projection.
By Remark~\ref{rem:sing_Q-bdle_over_C}, there is an open subset $X^{\circ}\subset X$ with $\codim_X(X\setminus X^{\circ})\ge 2$
such that $\psi^{\circ}=\psi_{|X^{\circ}}:X^{\circ}\to \p_{\p^1}(\sK)$ is a smooth morphism with connected fibers. 
Set $\sM:=\det(\sV)\otimes\sO_{\p^1}(b)\subset T_{\p^1}$.
This inclusion lifts 
to an inclusion of vector bundles $q^*\sM\subset T_{\p(\sK)}$.
Let $\sF$ be the pullback via $\psi$ of the foliation defined by $q^*\sM$ in $\p_{\p^1}(\sK)$.
One  computes that  $\sO_X(-K_{\sF})\cong \sO_{\p_{\p^1}(\sE)}(1)_{|X}$.
\end{proof}

\begin{exmp}
Set $\sE=\sO_{\p^1}(a_1)\oplus\cdots\oplus\sO_{\p^1}(a_6)$,
with $1=a_1=a_2\le a_3\le a_4 \le a_5=a_6=a$, and $a_3+a_4 \le a+1$.
Set $\sK=\sO_{\p^1}(a)^{\oplus 2}$, and
$b=-(a_3+a_4)$. Then $\sE$ and $b$ satisfy condition (4) in Proposition \ref{proposition:Q-bdle_over_C}.
Let $\lambda_{4,4} \in H^0\big(\p^1,\sO_{\p^1}(2a_4+b)\big)$, 
$\lambda_{2,5}\in H^0\big(\p^1,\sO_{\p^1}(a_2+a_5+b)\big)$
and $\lambda_{1,6} \in H^0\big(\p^1,\sO_{\p^1}(a_1+a_6+b)\big)$
be general sections. Then
$\lambda_{4,4}+\lambda_{2,5}+\lambda_{1,6}
\in 
H^0\big(\p^1,S^2\sE\otimes \sO_{\p^1}(b)\big)
\cong H^0\big(\p_{\p^1}(\sE),\sO_{\p_{\p^1}(\sE)}(2)\otimes \pi^*\sO(b)\big)$
defines a smooth hypersurface $X\subset \p_{\p^1}(\sE)$.  
\end{exmp}

% 4.4.  Codimension 1 Mukai foliations on   projective space bundles over surfaces

\subsection{Codimension $1$ Mukai foliations on projective space bundles over surfaces}\label{subsection:P-bdles/surfaces}

\

In this subsection, we work under Assumptions~\ref{assumptions}, supposing moreover that $\tau(L)=n-1$ and 
$\varphi_L$ makes $X$ a $\p^{n-2}$-bundle over a smooth surface $S$. This is case  (2c) of Theorem~\ref{tironi}. 

We start with some easy observations.

\begin{lemma}\label{lemma:hirzebruch}
Let $C$ be a smooth proper curve of genus $g\ge 0$,
and $p: S \to C$ a ruled surface.
Suppose that $-K_S\sim A+B$ where $A$ is an ample divisor, and $B$ is effective.
Then $g=0$. If moreover $B$ is nonzero and supported on fibers of $p$, then 
$S\cong\p^1\times\p^1$.
\end{lemma}

\begin{proof}
Let $\sE$ be a normalized vector bundle
on $C$ such that $S\cong \p_C(\sE)$, and set $e:=-\deg(\sE)$,
as in \cite[Notation V.2.8.1]{hartshorne77}.
Denote by $f$ a general fiber, and by 
$C_0$ a minimal section of $p: S \to C$. 
Let $a$ and $b$ be integers such that 
$A\sim aC_0+bf$. We have $-K_S\sim 2C_0+(2-2g+e)f$, and hence
$a\in\{1,2\}$, and $B\sim (2-a)C_0+(e-2g+2-b)f$. 
Since $B$ is effective, $e-2g+2-b \ge 0$.

We claim that $e\ge 0$. Suppose to the contrary that $e<0$.
Then $b>\frac{1}{2}ae$ by \cite[Proposition V.2.21]{hartshorne77}. Thus
$2g-2\le e-b <e-\frac{1}{2}ae<0$,
and $g=0$. But this contradicts \cite[Theorem V.2.12]{hartshorne77}, proving the claim.

By \cite[Proposition V.2.20]{hartshorne77}, 
we must have $b\ge ae+1$, and thus $-2 \le 2g-2\le e-b \le (1-a)e-1<0$.
This implies $g=0$, and $b-e\in\{1,2\}$. If moreover $a=2$ and $B\neq 0$, 
then $e=0$, completing the proof of the lemma.
\end{proof}

\begin{lemma}\label{lemma:big_anticanonical}
Let $S$ be a smooth projective surface such that 
$-K_S\sim A+B$ where $A$ is ample, and $B\neq 0$ is effective. 
Then 
either $S\cong \p^2$, or $S$ is a Hirzebruch surface.
\end{lemma}

\begin{proof}
It is enough to show that either $S$ is minimal, or $S \cong \mathbb{F}_1$. 
Suppose otherwise, and let $c : S \to T$ be a proper birational morphism onto a ruled surface $q:T\to C$. 
Set $A_T:=c_*A$, and $B_T:=c_*B$. Then $A_T$ is ample, $B_T$ is effective, and $-K_{T}\sim A_T+B_T$.
By Lemma \ref{lemma:hirzebruch}, $C \cong \p^1$ and $T \cong \mathbb{F}_e$ for some $e \ge 0$. 

Let $p : S \to C$ be the induced morphism, and denote by $f$  a fiber of $p$ or $q$.
Since $c$ is not an isomorphism by assumption, $A \cdot f  \ge 2$. 
On the other hand, $-K_S\cdot f=2$, and thus $A \cdot f = -K_S\cdot f-B\cdot f\le 2$. Hence
$A \cdot f = A_T\cdot f =2$, and $B\cdot f= B_T\cdot f = 0$. 
These equalities, together with the fact that $A_T$ is ample and  $-K_{T}\sim A_T+B_T$, 
imply that one of the following holds:
\begin{enumerate} 
	\item  $T \cong \mathbb{F}_1$ and $B_T=0$; or
	\item $T\cong \p^1\times \p^1$ and either $B_T=0$ or $B=f$.
\end{enumerate}

Suppose that $B_T=0$. Intersecting $-K_S$ with the (disjoint) curves $E_i$'s contracted by $c$ gives that 
$A=c^*A_T - \sum E_i=-K_S$. But this forces $B=0$, contrary to our assumptions. 

So we must have $T\cong \p^1\times \p^1$ and $B=f$.
Intersecting $-K_S$ with the (disjoint) curves $E_i$'s contracted by $c$ gives that 
$A=c^*A_T - \sum E_i$ and $B=f$. 
Let $\ell\subset T$ be a  fiber of the projection $T \cong \p^1\times \p^1\to \p^1$
transversal to $q:T\to \p^1$. Then $A_T\cdot \ell= 1$.
We choose $\ell$ to contain the image of some $E_i$, and
let $\tilde \ell\subset S$ be its strict transform.
Then  $A\cdot \tilde \ell\le 0$, contrary to our assumptions. 

We conclude  that either $S$ is minimal, or $S \cong \mathbb{F}_1$. 
\end{proof}

\begin{prop}\label{proposition:P-bdle_over_S}
Let $X$, $\sF$ and $L$ be as in Assumptions~\ref{assumptions}.
Suppose that $\tau(L)=n-1$, and  $\varphi_L$ makes $X$ a $\p^{n-2}$-bundle over a smooth surface $S$.
Then there exist
\begin{itemize}
	\item an exact sequence of sheaves of $\sO_S$-modules
		$$
		0\ \to \ \sK \ \to \ \sE \ \to \ \sQ \ \to \ 0,
		$$
		where $\sK$, $\sE$, and $\sV:=\sQ^{**}$ are vector bundles on $S$, $\sE$ is ample of rank $n-1$, and
		$\rank(\sK)=2$;
		
	\item  a codimension $1$ foliation $\sG$ on $\p_S(\sK)$, generically transverse to the natural projection  $q:\p_{S}(\sK)\to S$,
		satisfying $\det(\sG)\cong q^*\det(\sV)$ and $r_{\sG}^a\ge 1$;			
\end{itemize}		
such that  $X\cong \p_{S}(\sE)$, and
$\sF$ is the pullback of $\sG$ via the induced relative linear projection $\p_C(\sE)\map \p_C(\sK)$.
In this case,  $r_{\sF}^a\ge r_{\sF}-1$.	
Moreover, one of the following holds.

\begin{enumerate}
\item $S\cong\p^2$, $\det(\sV)\cong\sO_{\p^2}(i)$ for some $i\in \{1,2,3\}$, and $4\le n\le 3+i$.
\item $S$ is a del Pezzo surface $\not\cong\p^2$, $\det(\sV)\cong \sO_S(-K_S)$ , and $4\le n\le 5$.
\item $S\cong \p^1\times \p^1$, $\det(\sV)$ is a line bundle of type $(1,1)$, $(2,1)$ or $(1,2)$, and $n=4$.
\item $S\cong \mathbb{F}_e$ for some integer $e \ge 1$, $\det(\sV)\cong\sO_{\mathbb{F}_e}(C_0+(e+i)f)$,
	where  $i\in\{1,2\}$, $C_0$ is the minimal section of the natural morphism $\mathbb{F}_e\to\p^1$,
	$f$ is a general fiber, and $n=4$.
\end{enumerate}

Conversely, given $S$, $\sK$, $\sE$, $\sV$ as above, and a codimension one foliation $\sG\subset T_{\p_{S}(\sK)}$ 
satisfying $\det(\sG)\cong q^*\det(\sV)$, 
the pullback of $\sG$ via the relative linear projection $X\cong \p_{S}(\sE) \map \p_{S}(\sK)$ is a codimension one Mukai foliation on $X$.
\end{prop}

\begin{proof}

Denote by $F\cong \p^{n-2}$ a general fiber of $\varphi_L$, and recall from Theorem~\ref{tironi}(2c) that $\sL_{|F}\cong \sO_{\p^{n-2}}(1)$.
Set  $\sE:=(\varphi_L)_*\sL$. Then $\sE$ is an ample vector bundle of rank $n-1$, and $X\cong \p_S(\sE)$.

We claim that $T_{X/S}\not \subseteq \sF$. Indeed, if $T_{X/S}\subseteq \sF$,
then $\sF$ would be the pullback via $\varphi_L$ of a foliation on $S$, and hence
${\sL^{\otimes n-3}}_{|F}\cong \det(\sF)_{|F}\cong \det(T_{X/S})_{|F}\cong {\sL^{\otimes n-1}}_{|F}$
by \ref{pullback_foliations}, which is absurd. 
So $\sH:=\sF\cap T_{X/S}$ is a codimension $3$ foliation on $X$. Denote by  $\sH_F$ the codimension $1$ foliation on $F\cong \p^{n-2}$
obtained by restriction of $\sH$.
By \ref{restricting_fols}, there exists a non-negative integer $c$ such that 
$-K_{\sH_F} \ = \ (n-3+c) H$,
where $H$ denotes a hyperplane in $\p^{n-2}$.
By Thereom~\ref{Thm:ADK},  $c=0$, and
$\sH_{F}$ is a degree zero foliation on $F\cong \p^{n-2}$.

Let $\sK$ and $\sV$ be as defined  in \ref{V_in_E}. 
By \cite[Corollary 1.4]{hartshorne80} and \cite[Remark 2.3]{fano_fols}, $\sK$ and $\sV$ are vector bundles on $S$, and there is an exact sequence
$$
0\ \to \ \sK \ \to \ \sE \ \to \ \sQ \ \to \ 0
$$
with $\sQ^{**}\cong\sV$.
The foliation $\sH$ is induced by the relative linear projection $\psi: X\cong\p_{S}(\sE) \map  \p_{S}(\sK)$.
So, by \ref{pullback_foliations}, $\sF$ is the pullback via $\psi$ of a rank $2$ foliation $\sG$ on $\p_{S}(\sK)$.
There is an open subset $X^{\circ}\subset X$ with $\codim_X(X\setminus X^{\circ})\ge 2$
such that $\psi^{\circ}=\psi_{|X^{\circ}}:X^{\circ}\to \p_{S}(\sK)$ is a smooth morphism with connected fibers.
So, by \eqref{K_pullback_fol},
$$
\sL^{\otimes n-3}\cong \det(\sF)\cong \det(T_{X/\p_{S}(\sK)})\otimes \varphi^*\det(\sG),
$$
where $T_{X/\p_{S}(\sK)}$ denotes the saturation of $T_{X^\circ/\p_{S}(\sK)}$ in $T_X$.
A straightforward computation gives $\det(\sG)\cong q^*\det (\sV)$, where $q:\p_{S}(\sK)\to S$ 
denotes the natural projection. 
By Lemma \ref{lemma:almost_quotient_of_ample_is_ample}, $\det(\sV)$ is an ample line bundle on $S$.
Thus, applying Theorem~\ref{thm:BM} to a suitable destabilizing subsheaf of $\sG$, we conclude that $r_{\sG}^a\ge 1$.

The natural morphism $\sG \to q^*T_{S}$ is injective since 
$T_{\p_{S}(\sK)/S}\not\subseteq \sG$. Let $q^*B$ be the divisor of zeroes of 
the induced map $q^*\det(\sV)\cong\det(\sG) \to q^*\det(T_{S})$.

\medskip

Suppose first that $B=0$. Then
$\det(\sV)\cong \sO_S(-K_S)$, and hence $S$ is a del Pezzo surface. 
If $S\cong\p^2$, then $\det(\sV)\cong\sO_{\p^2}(3)$.
Since the restriction of $\sV$ to a general line on $\p^2$ is an ample vector bundle, 
$\rank(\sV) \le 3$, and hence $4\le n\le 6$.
Suppose that $S\not \cong\p^2$, and let $\ell \subset S$ be a general free rational curve of minimal anticanonical degree.
Then $\det(\sV)\cdot \ell = -K_S \cdot \ell = 2$. 
Since $\sV_{|\ell}$ is an ample vector bundle, $\rank(\sV) \le 2$, and hence $4\le n\le 5$.

\medskip

Suppose now that $B\neq 0$. By Lemma \ref{lemma:big_anticanonical}, 
either $S\cong \p^2$, or $S$ is a Hirzebruch surface.
If $S\cong\p^2$, then $\det(\sV)\cong\sO_{\p^2}(i)$, with $i\in \{1,2\}$.
As above, we see that $\rank(\sV) \le i$, and hence $4\le n\le 3+i$.
If $S\cong \mathbb{F}_e$ for some $e\ge 0$, then 
a straightforward computation gives that  
either $\det(\sV)\cong\sO_{\mathbb{F}_e}(C_0+(e+i)f)$, with $i\in\{1,2\}$, 
or $e=0$ and $\det(\sV)\cong\sO_{\mathbb{F}_e}(2C_0+f)$. 
In any case, $\det(\sV)\cdot \ell=1$ for a suitable free rational curve $\ell\subset \mathbb{F}_e$.
Since $\sV_{|\ell}$ is an ample vector bundle, $\rank(\sV) =1$, and hence $n=4$.

\medskip

Conversely, given $S$,  $\sK$, $\sE$, $\sV$ satisfying one the conditions in the statement of Proposition
\ref{proposition:P-bdle_over_S}, and a codimension one foliation $\sG\subset T_{\p_{S}(\sK)}$ 
satisfying $\det(\sG)\cong q^*\det(\sV)$, a straightforward computation shows that
the pullback of $\sG$ via the relative linear projection $X\cong \p_{S}(\sE) \map \p_{S}(\sK)$ is a codimension one foliation on $X$
with determinant $\sO_{\p_S(\sE)}(n-3)$.
\end{proof}

\begin{lemma}\label{lemma:almost_quotient_of_ample_is_ample}
Let $S$ be a smooth projective surface, $W\subset S$ a closed subscheme with $\textup{codim}_S W \ge 2$, 
$\sE$ an ample vector bundle on $S$,
and $\sV$ a vector bundle on $S$ such that there exists a surjective morphism of sheaves of
$\sO_S$-modules
$\sE \twoheadrightarrow \sI_W\sV$. Then $\det(\sV)$ is an ample line bundle.
\end{lemma}

\begin{proof}
Let $r$ be the rank of $\sV$.
The $r$-th wedge product of the morphism
$\sE \twoheadrightarrow \sI_W\sV$ gives rise to 
a surjective morphism $\wedge^r\sE_{|S \setminus \textup{Supp}(W)}  \twoheadrightarrow 
\det(\sV)_{|S \setminus \textup{Supp}(W)}$. It follows that $\det(\sV) \cdot C \ge 1$ for any curve $C \subset S$. 
To conclude that $\det(\sV)$ is ample, it is enough to show that $h^0(S,\det(\sV)^{\otimes m})\ge 1$
for some integer $m\ge 1$
by the Nakai-Moishezon criterion.

Set $Y:=\p_S(\wedge^r\sE)$. Denote by $\sO_Y(1)$ the tautological line bundle on $Y$, and 
by $q : Y \to S$ the natural projection.
Let $T\subset Y$ be the closure of the section of
$q_{|S\setminus \textup{Supp}(W)}$ corresponding to 
$\wedge^r\sE_{|S \setminus \textup{Supp}(W)}  \twoheadrightarrow \det(\sV)_{|S \setminus \textup{Supp}(W)}$. 
Then $\sO_Y(1)_{|T}\cong (q_{|T})^*\det(\sV)\otimes\sO_T(E)$ for some divisor $E \subset T$
with $\textup{Supp}(E)\subset T \cap q^{-1}(\textup{Supp}(W))$. 
Since $\sO_Y(1)_{|T}$ is an ample line bundle, we must have 
$$h^0(S\setminus \textup{Supp}(W),{\det(\sV)^{\otimes m}}_{|S\setminus \textup{Supp}(W)}) \ge 
h^0(T\setminus \textup{Supp}(E),{\sO_Y(m)}_{|T\setminus \textup{Supp}(E)})
\ge h^0(T,{\sO_Y(m)}_{|T})
\ge 1$$ for some $m\ge 1$,
and hence 
$h^0(S,\det(\sV)^{\otimes m})\ge 1$. 
\end{proof}

Our next goal is to classify pairs $(\sK,\sG)$ that appear in Proposition~\ref{proposition:P-bdle_over_S}.
When $\det(\sG)\cong q^*\sO_S(-K_S)$, the situation is easily described as follows. 
This includes the cases described in Proposition~\ref{proposition:P-bdle_over_S}(1, $i=3$) and (2). 

\begin{rem}\label{rem:flat_connection}
Let $Z$ be a simply connected smooth projective variety, and let $\sK$ be a rank $2$ vector bundle on $Z$.
Set $Y:=\p_Z(\sK)$, with natural projection $q : Y \to Z$. Denote by $\sO_{Y}(1)$ the tautological line bundle on $Y$. Let $\sG \subset T_Y$ be a codimension one foliation on $Y$ such that
$\sG\cong q^*\sO_Z(-K_Z)$.
Then $\sG \subset T_Y$ induces a flat connection on $q$. Thus $\sK\cong\sM\oplus \sM$ for some line bundle $\sM$ on $Z$,  
and $\sG$ is induced by the natural morphism $\p_Z(\sK)\cong Z\times\p^1\to \p^1$.
\end{rem}

Suppose now that  $S\cong \p^2$ or $\mathbb{F}_e$, and $\det(\sG)\not \cong q^*\sO_S(-K_S)$.
We will describe $\sK$ and $\sG$ that appear in Proposition~\ref{proposition:P-bdle_over_S}
by restricting them to special rational curves on $S$.
Our analysis  will rely on the following result.

\begin{lemma}\label{lemma:foliation_surface}
Let $m \ge 0$ be an integer, and consider the ruled surface $q : \mathbb{F}_m \to \p^1$. 
Let $\sC\cong q^*\sO_{\p^1}(a)$ be a foliation by curves on $\mathbb{F}_m$ with $a> 0$.
Then $a\in\{1,2\}$, and one of the following holds. 
\begin{enumerate}
\item If $a=2$, then $m=0$, and $\sC$ is induced by the projection $\mathbb{F}_0 \cong \p^1\times\p^1 \to \p^1$ transversal to $q$.
\item If $a=1$, then $m\ge 1$, and $\sC$ 
is induced by a pencil  containing $C_0+m f_0$, 
where $C_0$ denotes the minimal section and $f_0$ a fiber of $q : \mathbb{F}_m \to \p^1$. 
\end{enumerate}
\end{lemma}

\begin{proof}
Notice that $\sC\neq T_{\mathbb{F}_m/\p^1}$, thus the natural map 
$q^*\sO_{\p^1}(a)\cong \sC \to q^*T_{\p^1}\cong q^*\sO_{\p^1}(2)$ is nonzero,
and hence $a\in\{1,2\}$.
If $a=2$, then, as in Remark~\ref{rem:flat_connection},  
$\sC$ yields a flat connexion on $q$, $m=0$, and $\sC$ is induced by the projection 
$\mathbb{F}_0 \cong \p^1\times\p^1 \to \p^1$ transversal to $q$, proving (1).

From now on we assume that $a=1$. Then we must have $m\ge 1$,
since the map $\sC \to T_{\mathbb{F}_m}$ does not vanish in codimension one.
Denote by $C_0$ be a minimal section, and by $f$ a fiber of $q : \mathbb{F}_m \to \p^1$.

By Theorem \ref{thm:BM}, $\sC$ is algebraically integrable and its leaves are rational curves. 
So it is induced by a rational map with irreducible general fibers $\pi : \mathbb{F}_m \dashrightarrow \p^1$.
Let $C$ be the closure of a general leaf of $\sC$.
As in the proof of Proposition~\ref{proposition:P-bdle_over_curve_(2)},
one shows that $C$ is a section of  $q$, and $\sC$ is not regular along $C$.
Write $C \sim C_0+b f$ with $b \ge m$ (see \cite[Proposition V.2.20]{hartshorne77}). 

The foliation $\sC$ is induced by a pencil $\Pi$ of members of $|\sO_{\mathbb{F}_m}(C)|$.
Observe that the space of reducible members of $|\sO_{\mathbb{F}_m}(C)|$ is a codimension one linear subspace.
Therefore, $\Pi$ has a unique reducible member. 

Let $f_0$ be the divisor of zeroes of 
$q^*\sO_{\p^1}(1)\cong \sC \to q^*T_{\p^1}\cong q^*\sO_{\p^1}(2)$. It is a fiber of $q$.
Note that $\sC$ induces a flat connection on $q$ over $\mathbb{F}_m \setminus f_0$. In particular, 
$\sC$ is regular over $\mathbb{F}_m \setminus f_0$.
Let $R(\pi)$ be the ramification divisor of $\pi$, and notice that $R(\pi)$ is supported on $f_0$. 
A straightforward computation gives $R(\pi)\equiv (2b-(m+1))f$. 
Let $C_1+kf_0$ the reducible member of $\Pi$ ($k\ge 1$), where $C_1$ is irreducible. 
Write $C_1 \sim C_0+b_1f$.
Then 
$k=2b-(m+1)+1$, and $b_1+k=b$. 
Thus $b-k=m-b \le 0$, and $b-k=b_1 \ge 0$. 
Hence $b_1=0$, and $k=b=m$. This proves (2).
\end{proof}

\begin{prop}\label{proposition:P-bdle_over_S_hirzebruch} 
Let $\sK$ be a rank $2$ vector bundle on a ruled surface $p : \mathbb{F}_e \to \p^1$, $e \ge 0$.
Set $Y:=\p_{\mathbb{F}_e}(\sK)$, 
with natural projection $q : Y \to \mathbb{F}_e$, and tautological line bundle $\sO_{Y}(1)$.
Let $\sG\subset T_Y$ be a codimension one foliation on $Y$ with $\det(\sG)\cong q^*\sA$ 
for some ample line bundle $\sA$ on $\mathbb{F}_e$. Then one of the following holds.

\begin{enumerate}
\item $e\in\{0,1\}$ and there exists a line bundle $\sB$ on $\mathbb{F}_e$ such that 
	\begin{itemize}
		\item $\sK\cong\sB\oplus \sB$, and
		\item $\sG$ is induced by the natural morphism $Y\cong \mathbb{F}_e\times\p^1\to \p^1$
			and thus  $\det(\sG)\cong q^*\sO_{\mathbb{F}_e}(-K_{\mathbb{F}_e})$.
	\end{itemize}
\item There exist a line bundle $\sB$ on $\mathbb{F}_e$, integers $s \ge 1$ and $t\ge 0$, 
	a minimal section $C_0$ and a fiber $f$ of $p : \mathbb{F}_e \to \p^1$ such that
	\begin{itemize}
		\item $\sK\cong\sB\otimes \big(\sO_{\mathbb{F}_e}\oplus\sO_{\mathbb{F}_e}(sC_0+tf)\big)$,
		\item $\sG$ is induced by a pencil in  $|\sO_{Y}(1)\otimes q^*\sB^*|$
			containing $\Sigma+q^*(sC_0+tf)$, where  $\Sigma$ is the section of $q : Y \to \mathbb{F}_e$
			corresponding to the surjection 
			$\sO_{\mathbb{F}_e}\oplus\sO_{\mathbb{F}_e}(sC_0+tf)\twoheadrightarrow \sO_{\mathbb{F}_e}$,
		\item  $\det(\sG)\cong q^*\sO_{\mathbb{F}_e}(C_0+(e+2)f)$ if $t=0$, and 
			$\det(\sG)\cong q^*\sO_{\mathbb{F}_e}(C_0+(e+1)f)$ if $t>0$.
	\end{itemize}
\item $e\in\{0,1\}$, there exist a line bundle $\sB$ on $\mathbb{F}_e$, an integer $s \ge 1$, and 
	an irreducible divisor $B\sim C_0+f$ on $\mathbb{F}_e$, where $C_0$ is a minimal section and
	$f$ a  fiber  of $p : \mathbb{F}_e \to \p^1$, such that 
	\begin{itemize}
		\item $\sK\cong\sB\otimes \big(\sO_{\mathbb{F}_e}\oplus\sO_{\mathbb{F}_e}(s(C_0+f))\big)$,
		\item $\sG$ is induced by a pencil in  $|\sO_{Y}(1)\otimes q^*\sB^*|$
			containing $\Sigma+sq^*B$, where $\Sigma$ is the section of $q : Y \to \mathbb{F}_e$
			corresponding to the surjection 
			$\sO_{\mathbb{F}_e}\oplus\sO_{\mathbb{F}_e}(s(C_0+f))\twoheadrightarrow \sO_{\mathbb{F}_e}$,
		\item $\det(\sG)\cong q^*\sO_{\mathbb{F}_e}(C_0+(e+1)f)$. 
	\end{itemize}
\item There exist a line bundle $\sB$ on $\mathbb{F}_e$, integers $s,t \ge 1$, 
	a minimal section $C_0$ and a fiber $f$ of $p : \mathbb{F}_e \to \p^1$ such that
	\begin{itemize}
		\item $\sK\cong\sB\otimes \big(\sO_{\mathbb{F}_e}(sC_0) \oplus \sO_{\mathbb{F}_e}(t f)\big)$,
		\item $\sG$ is induced by a pencil in  $|\sO_{Y}(1)\otimes q^*\sB^*|$ generated by 
			$\Sigma+sq^*C_0$ and $\Sigma'+tq^*f$, where 
			$\Sigma$ and $\Sigma'$ are sections of $q : Y \to \mathbb{F}_e$ corresponding to the surjections 
			$\sO_{\mathbb{F}_e}(sC_0) \oplus \sO_{\mathbb{F}_e}(t f) \twoheadrightarrow \sO_{\mathbb{F}_e}(t f)$, and
			$\sO_{\mathbb{F}_e}(sC_0) \oplus \sO_{\mathbb{F}_e}(t f) \twoheadrightarrow \sO_{\mathbb{F}_e}(s C_0)$, respectively,
		\item  $\det(\sG)\cong q^*\sO_{\mathbb{F}_e}(C_0+(e+1)f)$.
	\end{itemize}
\item There exist a line bundle $\sB$ on $\mathbb{F}_e$, integers $s \ge 1$ and $\lambda\ge 0$, 
	a minimal section $C_0$ and a  fiber $f$ of $p : \mathbb{F}_e \to \p^1$ such that 
	\begin{itemize}
		\item $\sK$ fits into an exact sequence 
			$$
			0\ \to \ \sO_{\mathbb{F}_e}\oplus \sO_{\mathbb{F}_e}(sC_0) \ \to \ 
			\sK\otimes\sB^*
			\ \to \ 
			\sO_{f}(-\lambda)
			\to \ 0,
			$$
		\item $\sG$ is induced by a pencil in  $|\sO_{Y}(1)\otimes q^*\sB^*|$ generated by 
			$\Sigma+sq^*C_0$ and $\Sigma'$, where 
			$\Sigma$ is the zero locus of the section of 
			$\sO_{Y}(1)\otimes q^*(\sB^*\otimes \sO_{\mathbb{F}_e}(-sC_0))$
			corresponding to 
			$\sO_{\mathbb{F}_e}(sC_0) \to \sO_{\mathbb{F}_e}\oplus \sO_{\mathbb{F}_e}(sC_0) \to \sK\otimes\sB^*$, and 
			$\Sigma'$ corresponds to  
			$\sO_{\mathbb{F}_e} \to \sO_{\mathbb{F}_e}\oplus \sO_{\mathbb{F}_e}(sC_0) \to \sK\otimes\sB^*$,
		\item $\det(\sG)\cong q^*\sO_{\mathbb{F}_e}(C_0+(e+1)f)$.
	\end{itemize}
\item There exist a line bundle $\sB$ on $\mathbb{F}_e$, an integer $t\ge 0$,
 	a minimal section $C_0$ and a  fiber $f$ of $p : \mathbb{F}_e \to \p^1$,
	and a local complete intersection subscheme $\Lambda \subset \mathbb{F}_e$ of codimension $2$, with
	$h^0(\mathbb{F}_e,\sI_\Lambda\sO_{\mathbb{F}_e}(C_0))\ge 1$,  such that 
	\begin{itemize}
		\item $\sK$ fits into an exact sequence 
			$$
			0 \ \to \ \sO_{\mathbb{F}_e}(tf) \ \to \
			\sK\otimes\sB^* \ \to \ \sI_\Lambda\sO_{\mathbb{F}_e}(C_0) \ \to \ 0,
			$$
		\item $\sG$ is induced by a pencil in  $|\sO_{Y}(1)\otimes q^*\sB^*|$ 
			containing $\Sigma+tq^*f$, where  $\Sigma$ is the zero locus of the section of 
			$\sO_{Y}(1)\otimes q^*(\sB^*\otimes\sO_{\mathbb{F}_e}(-tf))$
			corresponding to 
			$\sO_{\mathbb{F}_e}(tf) \to \sK\otimes\sB^*$,
		\item $\det(\sG)\cong q^*\sO_{\mathbb{F}_e}(C_0+(e+1)f)$.
	\end{itemize}
\item There exist a line bundle $\sB$ on $\mathbb{F}_e$, 
 	a minimal section $C_0$ and a fiber $f$ of $p : \mathbb{F}_e \to \p^1$,
	and a local complete intersection subscheme $\Lambda \subset \mathbb{F}_e$ of codimension $2$
	such that 
	\begin{itemize}
		\item there exists a curve $C \sim C_0+f$ with $\Lambda \subset C$ and, 
			for any proper subcurve $C' \subsetneq C$, $\Lambda\not\subset C'$,
		\item $\sK$ fits into an exact sequence 
			$$
			0 \ \to \ \sO_{\mathbb{F}_e} \ \to \
			\sK\otimes\sB^* \ \to \ \sI_\Lambda\sO_{\mathbb{F}_e}(C_0+ f) \ \to \ 0,
			$$
		\item $\sG$ is induced by a pencil of irreducible members of  $|\sO_{Y}(1)\otimes q^*\sB^*|$ 
			containing $\Sigma$,  the zero locus of the section of 
			$\sO_{Y}(1)\otimes q^*\sB^*$ corresponding to 
			$\sO_{\mathbb{F}_e} \to \sK\otimes\sB^*$,
		\item $\det(\sG)\cong q^*\sO_{\mathbb{F}_e}(C_0+(e+1)f)$.
	\end{itemize}
\end{enumerate}
\end{prop}

\begin{proof}
To ease notation, set $S:=\mathbb{F}_e$.
Denote by $C_0$ a minimal section, and by $f$ a general fiber of $p : S \to \p^1$.
Denote by $F\cong \p^1$ a general fiber of $q:Y\to S$.
Given any curve $C\subset S$, we set $Y_C:=q^{-1}(C)$, and denote by 
$q_C : Y_C \to C$ the restriction of $q$ to $Y_C$.

We claim that $T_{Y/S}\not \subseteq \sG$. 
Indeed, if $T_{Y/S}\subseteq \sG$,
then $\sG$ would be the pullback via $q$ of a foliation on $S$, and hence
$\sO_{\p^1}\cong \det(\sG)_{|F}\cong (T_{Y/S})_{|F}\cong \sO_{\p^1}(2)$
by \ref{pullback_foliations}, which is absurd. 
Therefore, the natural map $T_Y \to q^*T_{S}$
induces an injective morphism of sheaves $\sG \to q^*T_{S}$.
Let $q^*B$ be the divisor of zeroes of 
the induced map $q^*\det(\sA)\cong \det(\sG) \to q^*\det(T_{S})$.

Suppose that $B=0$. Then
$\sA\cong \sO_S(-K_S)$ is ample, and hence $e\in\{0,1\}$. Moreover,  
$\sG \subset T_Y$ induces a flat connection on $q$. Thus $\sK\cong\sB\oplus \sB$ for some line bundle $\sB$ on $S$,  
and $\sG$ is induced by the natural morphism $Y\cong S\times\p^1\to \p^1$.
This is case (1) in the statement of Proposition~\ref{proposition:P-bdle_over_S_hirzebruch}.

Suppose from now on that $B\neq 0$.
A straightforward computation shows that one of the following holds
(up to possibly exchanging $C_0$ and $f$ when $e=0$). 
\begin{itemize}
	\item $\sA\cong\sO_{S}(C_0+(e+1)f)$ and $B\sim C_0+f$, or
	\item $\sA\cong\sO_{S}(C_0+(e+2)f)$ and $B\sim C_0$.
\end{itemize}
In either case, $B$ contains a unique irreducible component dominating $\p^1$. 
We denote this  irreducible component by $B_1$, and set $B_2:=B-B_1$.

Let $C\cong\p^1$  be a general member of $|\sO_S(C_0+ef)|$. 
Since $T_{Y/S}\not \subseteq \sG$, $\sG$ induces foliations by curves 
$\sC_f\subset {T_{Y_f}}$ and $\sC_C\subset {T_{Y_C}}$ on $Y_f$ and $Y_C$, respectively.
By \ref{restricting_fols}, there exist effective divisors
$D_f$ on $Y_f$ and $D_C$ on $Y_C$ such that
\begin{equation}\label{equ:canonical_bundle}
\sO_{Y_f}(-K_{\sC_f})\cong (q^*\sA)_{|Y_f}
\otimes\sO_{Y_f}(D_f)  \text{, and }
\sO_{Y_C}(-K_{\sC_C})\cong q^*\big(\sA\otimes\sO_{S}(-C_0-ef)\big)_{|Y_C}
\otimes\sO_{Y_C}(D_C). 
\end{equation}

\noindent {\bf Claim.}
\begin{itemize}
	\item[(a)] $\sC_f\cong q_f^*\sO_{\p^1}(1)$, $Y_f\cong \mathbb{F}_m$ with $m\ge 1$, 
		       and $\sC_f$ is induced by a pencil  containing $\sigma_0+m \ell_0$, 
		       where $\sigma_0$ denotes the minimal section and $\ell_0$ a fiber of $q_f : Y_f \to \p^1$. 
	\item[(b)] If $B\sim C_0+f$ (and so $C \cap \textup{Supp}(B) \neq \emptyset$), 
		        then $\sC_C\cong q_C^*\sO_{\p^1}(1)$, $Y_C\cong \mathbb{F}_m$ with $m\ge 1$, 
		       and $\sC_C$ is induced by a pencil  containing $\sigma_0+m \ell_0$, 
		       where $\sigma_0$ denotes the minimal section and $\ell_0$ a fiber of $q_C : Y_C \to \p^1$.	
	\item[(c)] If $B\sim C_0$ (and so $C \cap \textup{Supp}(B) = \emptyset$), then $\sC_C\cong q_C^*\sO_{\p^1}(2)$, 
			$Y_C\cong C\times \p^1$, and $\sC_C$ is induced by the projection morphism $C\times \p^1 \to \p^1$.
\end{itemize}
On the open subset $Y\setminus q^{-1}(\textup{Supp}(B))$, $\sG$ is regular and induces a flat connection.
Therefore $\sG_{|Y_f}$ intersects $T_{Y_f/f}$ transversely over $f\setminus \{f\cap \textup{Supp}(B)\}$, 
and thus the support of $D_f$ is contained in $q^{-1}(f\cap \textup{Supp}(B))$. 
It follows from  \eqref{equ:canonical_bundle} that $\sC_f\cong q_f^*\sO_{\p^1}(k)$ for some 
positive integer $k$. 
Since the natural map $\sC_f \to q_f^*T_f$ is injective, we must have $k\in\{1,2\}$. 
The same argument shows that $\sC_C\cong q_C^*\sO_{\p^1}(l)$, with $l\in\{1,2\}$.
If $\sC_f\cong q_f^*\sO_{\p^1}(2)$, then $\sC_f$ is a regular foliation, $Y_f\cong f\times \p^1$, and $\sC_f$
is induced by the projection morphism $f\times \p^1 \to \p^1$. 
On the other hand, if $b$ is a general point in $\textup{Supp}(B)$, 
then $q^{-1}(b)$ is tangent to $\sG$, 
while $\sG$ is regular at a general point of $q^{-1}(b)$.
Since $f$ is assumed to be general, $f\cap \textup{Supp}(B)$ is a general point $b\in \textup{Supp}(B)$, and we 
conclude from this observation that $Y_f$ must be a leaf of $\sG$, which is absurd.
Therefore we must have $\sC_f\cong q_f^*\sO_{\p^1}(1)$ and $D_f=0$.
Analogously, we prove that if  $C \cap \textup{Supp}(B) \neq \emptyset$, then $\sC_C\cong q_C^*\sO_{\p^1}(1)$
and $D_C=0$.
The description of $(Y_f,\sC_f)$ and $(Y_C,\sC_C)$ in this case follow from Lemma~\ref{lemma:foliation_surface}.
Finally, if $B\sim C_0$, then 
$D_C=0$, and $\sC_C$ induces a flat connection on $q_C$.
Therefore $\sC_C\cong q_C^*\sO_{\p^1}(2)$,
$Y_C\cong C\times \p^1$, and $\sC_C$
is induced by the projection morphism $C\times \p^1 \to \p^1$.
This proves the claim.

\smallskip

Next we show that $\sG$ has algebraic leaves, 
and that a general leaf has relative degree $1$ over $S$. 
From the claim, we know that  the general leaves of $\sC_f$ and $\sC_C$ are sections of 
$q_f:Y_f\to f$ and $q_C:Y_C\to C$, respectively.
Let $F_C$ be a general leaf of $\sC_C$ mapping onto $C$. 
For a general  fiber $f$ of $p:S\to \p^1$, $Y_f$ meets $F_C$ 
in a single point, and there is a unique leaf $F_f$ of $\sC_f$ through this point. 
We let $\Sigma$ be the closure of the union of the $F_f$'s obtained in this way, as 
$f$ varies through general fibers of $p:S\to \p^1$.
It is a general leaf of $\sG$, and has relative degree $1$ over $S$.

Since $\sG$ is algebraically integrable, we can consider the rational first integral 
for $\sG$, $\pi : Y \dashrightarrow \tilde W$, as described in \ref{notation:family_leaves}.
Since $Y$ is a rational variety, $\tilde W\cong\p^1$.  So
$\sG$ is induced by a pencil $\Pi$ in the linear system $|\sO_Y(1)\otimes q^*\sM|$
for some line bundle $\sM$ on $S$. 
Notice that 
$\pi_f:=\pi_{|Y_f} : Y_f \dashrightarrow \p^1$ and 
$\pi_C:=\pi_{|Y_C} : Y_C \dashrightarrow \p^1$ are rational first integrals  for 
$\sC_f$ and $\sC_C$, respectively, and $\sC_f$ and $\sC_C$ are 
induced by the restricted pencils  $\Pi_{|Y_f}$ and $\Pi_{|Y_C}$, respectively.

Our next task is to determine the line bundle $\sM$.
From claim (a--c), there are integers $a,b,s,t$, with $s\ge 1$ and $t\ge 0$
such that $\sK_{|f}\cong\sO_{\p^1}(a)\oplus\sO_{\p^1}(a+s)$, and 
$\sK_{|C}\cong\sO_{\p^1}(b)\oplus\sO_{\p^1}(b+t)$. 
Moreover, $\sM_{|f}\cong \sO_{\p^1}(-a)$ and $\sM_{|C}\cong \sO_{\p^1}(-b)$. 
This implies that $\sM\cong\sO_S(-aC_0-b f)$.

Any member of $\Pi$ can be written as $\Sigma+uq^*B_1+vq^*B_2$, where 
$\Sigma$ is irreducible and has relative degree $1$ over $S$, 
and  $u,v\ge 0$ are  integers.
In particular, the ramification divisor $R(\pi)$ of $\pi$ must be of the form 
$R(\pi)=cq^*B_1+d q^*B_2$, with $c,d \ge 0$  integers.
We have 
$$N_\sG\cong \big(\pi^*\Omega^1_{\p^1}\otimes \sO_Y(R(\pi))\big)^*\cong \sO_Y(2)\otimes q^*\sM^{\otimes 2}\otimes \sO_Y(-R(\pi)).$$
On the other hand, 
\begin{multline*}
$$
N_\sG\cong\sO_{Y}(-K_{Y})\otimes \sO_{Y}(K_\sG)\cong \sO_{Y}(2)\otimes q^*(\det(\sK^*)\otimes 
\sO_S(-K_S)\otimes\sA^*) \\
\cong \sO_{Y}(2)\otimes q^*(\det(\sK^*)\otimes 
\sO_S(B)),
$$
\end{multline*}
and hence
$$\sO_{Y}(R(\pi)) \cong  q^*(\sM^{\otimes 2}\otimes \det(\sK)\otimes 
\sO_S(-B))\cong q^*(\sO_S(s C_0+ t f)\otimes 
\sO_S(-B)).$$
It follows from claim (b--c) that, if $B\sim C_0$ then  
$t=d=0$, and if $B\not\sim C_0$, then $c=s-1$ and $d=t-1$. 
Notice also that, if $s \ge 2$, then the pencil $\Pi$ contains a member of the form 
$\Sigma+sq^*B_1+vq^*B_2$, where $\Sigma$ is irreducible and has relative degree $1$ over $S$, 
and  $v\ge 0$.
Similarly, if $t \ge 2$, then the pencil $\Pi$ contains a member of the form 
 $\Sigma+uq^*B_1+tq^*B_2$, where $\Sigma$ is irreducible and has relative degree $1$ over $S$, 
and   $u \ge 0$.

\medskip

\noindent {\bf Case 1:} Suppose that  $\Pi$ contains a member of the form 
$\Sigma+uq^*B_1+vq^*B_2$, where $\Sigma$ is irreducible and has relative degree $1$ over $S$, $B_2\neq 0$,
and  $u,v >0$.

Up to replacing $C_0$ and $f$ with linearly equivalent  curves on $S$, we may write
$B=C_0+f$.
It follows from claim (a--b) that  $u=s$ and $v=t$.
Moreover $\Sigma\cap Y_f$ and $\Sigma \cap Y_C$ 
are the minimal sections of $q_f:Y_f\to f$ and $q_C:Y_C\to C$, respectively.
If $\Sigma'$ is the closure of a general leaf of $\sG$, then 
$\Sigma\cap \Sigma'\cap Y_f = \emptyset=\Sigma\cap \Sigma'\cap Y_C$.
This implies that $\Sigma\cap \Sigma' \cap q^{-1}(b)=\emptyset$ for a general point $b \in\textup{Supp}(B)$.
One can find an open subset $V \subset S$, 
with $\textup{codim}_S(S\setminus V)\ge 2$, such that $\Sigma\cap q^{-1}(V)$
and $\Sigma'\cap q^{-1}(V)$ are sections of $q_{|q^{-1}(V)}$, and
$\Sigma\cap\Sigma'\cap q^{-1}(V)=\emptyset$. 
Therefore, there are line bundles $\sB_1$ and $\sB_2$ on $S$ such that 
$\sK \cong \sB_1\oplus \sB_2$, and $\Sigma$ corresponds to the surjection 
$\sB_1\oplus\sB_2\twoheadrightarrow \sB_1$. 
From the description of $\sK_{|f}$ and $\sK_{|C}$ above, we see that 
${\sB_1}_{|f}\cong\sO_{\p^1}(a)$, 
${\sB_1}_{|C}\cong\sO_{\p^1}(b)$,
${\sB_2}_{|f}\cong\sO_{\p^1}(a+s)$, and 
${\sB_2}_{|C}\cong\sO_{\p^1}(b+t)$.
Thus $\sB_1\cong\sO_S(aC_0+b f)$, and $\sB_2\cong\sO_S((a+s)C_0+(b+t) f)$. 
We are in case (2)  in the statement of Proposition~\ref{proposition:P-bdle_over_S_hirzebruch}, with $t\ge 1$.

\medskip

\noindent {\bf Case 2:} Suppose that  $\Pi$ contains a member of the form 
$\Sigma+uq^*B_1$, where $\Sigma$ is irreducible and has relative degree $1$ over $S$, 
and  $u>0$.

It follows from claim (a) that  $u=s$.
Next we prove that any other reducible divisor of $\Pi$ must be of the form 
$\Sigma'+tq^*B_2$, where $\Sigma'$ is irreducible and has relative degree $1$ over $S$. 
In particular, if there exists such divisor in $\Pi$, we must have $B_2\neq 0$ and $t>0$.
Indeed, let $D=\Sigma'+iq^*B_1+jq^*B_2$ be a reducible member of $\Pi$,
where $\Sigma'$ is irreducible and has relative degree $1$ over $S$. 
If $i>0$, then it follows from the claim  that  $i=s$ and $\Sigma = \Sigma'$.
Since $D\sim \Sigma+sq^*B_1$, we must have  $D= \Sigma+sq^*B_1$.
If $i=0$, then we must have $B_2\neq 0$ and $j>0$.
It follows from claim (b) that  $j=t$, and so $D=\Sigma'+tq^*B_2$.

We consider three cases.

\noindent {\bf Case 2.1:} Suppose that $B_2=0$, and $B=B_1\sim C_0+f$. 
Then we must have $e\in\{0,1\}$.

Let $\Sigma'$ is the closure of a general leaf of $\sG$.
It follows from claim (a) that  $\Sigma\cap Y_f$ is the minimal section of $q_f$, and
$\Sigma\cap \Sigma'\cap Y_f = \emptyset$.
Proceeding as in case 1, we show that we must be in case (3)  
in the statement of Proposition~\ref{proposition:P-bdle_over_S_hirzebruch}.

\noindent {\bf Case 2.2:} Suppose $B_2=0$, and $B=B_1\sim C_0$.
 
Up to replacing $C_0$ with a linearly equivalent  curve on $S$, we may write
$B=C_0$.

Since $\Sigma$ is irreducible and has relative degree $1$ over $S$, 
it contains only finitely many fibers of $q$, and corresponds to
a surjective morphism of sheaves $\sK \twoheadrightarrow \sI_{\Lambda}\sS$, where $\sS$ is a line bundle on $S$, and
$\Lambda \subset S$ is a closed subscheme with $\textup{codim}_S\Lambda \ge 2$. 
Denote by $\sT$ the kernel of this morphism.
Then $\sT$ is a line bundle on $S$, and  $\sT \otimes \sS\cong \det(\sK)\cong \sO_S((2a+s)C_0+2bf)$.
Since $\Sigma+sq^*C_0 \in \Pi\subset  |\sO_Y(1)\otimes q^*\sO_S(-aC_0-b f)|$, 
we must have $\sT \otimes \sO_S(-aC_0-b f)\cong \sO_S(s C_0)$, and 
$\sS \otimes \sO_S(-aC_0-b f) \cong \sO_S$.

Let $\Sigma'$ is the closure of a general leaf of $\sG$, and
$\sigma' \in H^0(Y,\sO_Y(1)\otimes q^*\sO_S(-aC_0-b f))$  a nonzero section vanishing along 
$\Sigma'$. 
We claim that 
$\sigma'$ is mapped to a nonzero element in $H^0(S,\sI_{\Lambda}\sS\otimes\sO_S(-aC_0-b f))\cong H^0(S,\sI_{\Lambda})$
under the natural morphism
$H^0(S,\sK\otimes \sO_S(-aC_0-b f)) \to 
H^0(S,\sI_{\Lambda}\sS\otimes\sO_S(-aC_0-b f))$.
Indeed, if $0\neq \sigma\in H^0\big(Y,\sO_Y(1)\otimes q^*\sO_S(-aC_0-b f))\big)$ comes from 
$H^0\big(S,\sT\otimes \sO_S(-aC_0-b f))\big) 
\subset H^0\big(S,\sK\otimes \sO_S(-aC_0-b f))\big) \cong 
H^0\big(Y,\sO_Y(1)\otimes q^*\sO_S(-aC_0-b f))\big)$, then its zero locus on $Y$ must be reducible, yielding a contradiction. 
We conclude that $\Lambda=\emptyset$, and
the exact sequence
$0 \to \sT \to \sK \to \sS \to 0$ splits.
So we are in case (2)  in the statement of Proposition~\ref{proposition:P-bdle_over_S_hirzebruch}, with
$s>0$ and $t=0$.

\noindent {\bf Case 2.3:} Suppose that $B_2\neq 0$.

Up to replacing $C_0$ and $f$ with linearly equivalent  curves on $S$, we may write
$B=C_0+f$.

As in case 2.2,  $\Sigma$ corresponds to
a surjective morphism of sheaves $\sK \twoheadrightarrow \sI_{\Lambda}\sS$, where $\sS$ is a line bundle on $S$, and
$\Lambda \subset S$ is a closed subscheme with $\textup{codim}_S\Lambda \ge 2$. 
Denote by $\sT$ the kernel of this morphism.
Then $\sT$ is a line bundle on $S$, and $\sT \otimes \sS\cong \det(\sK)\cong \sO_S((2a+s)C_0+(2b+t)f)
\cong \sO_S(sC_0+tf)\otimes\sO_S(2aC_0+2b f)$.
Since $\Sigma+sq^*C_0 \in |\sO_Y(1)\otimes q^*\sO_S(-aC_0-b f)|$, we must have 
$\sT \otimes \sO_S(-aC_0-b f)\cong \sO_S(s C_0)$, and 
$\sS \otimes \sO_S(-aC_0-b f) \cong \sO_S(t f)$.
Thus, we have an exact sequence
$$
0 \ \to \ \sO_S(sC_0) \ \to \
\sK\otimes\sO_S(-aC_0-b f) \ \to \ \sI_{\Lambda}\sO_S(t f) \ \to \ 0,
$$
where $\sO_S(sC_0) \to \sK\otimes\sO_S(-aC_0-b f)$ is the map corresponding to
$\Sigma+sq^*C_0\in \Pi$.

\smallskip

Suppose that $\Sigma+sq^*C_0$ is the only reducible member of
$\Pi$. Then we must have 
$t=1$. 
Let $\Sigma'$ is the closure of a general leaf of $\sG$, and
$\sigma' \in H^0(Y,\sO_Y(1)\otimes q^*\sO_S(-aC_0-b f))$  a nonzero section vanishing along 
$\Sigma'$. 
As in  case 2.2, we see that 
$\sigma'$ is mapped to a nonzero element  $\bar \sigma'\in H^0(S,\sI_{\Lambda}\sS\otimes\sO_S(-aC_0-b f))\cong H^0(S,\sI_{\Lambda}\sO_S( f))$
under the natural morphism
$H^0(S,\sK\otimes \sO_S(-aC_0-b f)) \to H^0(S,\sI_{\Lambda}\sS\otimes\sO_S(-aC_0-b f))$.
Let  $f'\sim f$ be the divisor of zeros of $\bar \sigma'$.
Then $\Lambda \subset f'$. 
Since
$\Sigma \cap \Sigma' \cap q^{-1}(b)=\emptyset$ for any point 
$b \in C_0\setminus f$, we must have 
$f=f'$.
We obtain an exact sequence
$$
0 \ \to \ \sO_S\oplus \sO_S(sC_0) \ \to \
\sK\otimes\sO_S(-aC_0-b f) \ \to \ \sO_{f}(-\Lambda) \ \to \ 0,
$$
where $\sO_S \to \sK\otimes\sO_S(-aC_0-b f)$  is the map given by $\sigma'$.
We are in case (5)  in the statement of Proposition~\ref{proposition:P-bdle_over_S_hirzebruch}

\smallskip

Now suppose that $\Pi$ contains a second reducible divisor.
We have seen above that it must be of the form 
$\Sigma'+tq^*f$, where $\Sigma'$ is irreducible and has relative degree $1$ over $S$. 
As before, it gives rise to an exact sequence 
$$
0 \ \to \ \sO_S(t f) \ \to \
\sK\otimes\sO_S(-aC_0-b f) \ \to \ \sI_{\Lambda'}\sO_S(s C_0) \ \to \ 0,
$$
where 
$\Lambda' \subset S$ is a closed subscheme with $\textup{codim}_S\Lambda' \ge 2$.
Notice that  $\Sigma\neq \Sigma'$, and so the induced morphism 
$\sO_S(sC_0) \oplus \sO_S(t f) \to \sK\otimes\sO_S(-aC_0-b f)$ is injective.
Since $\det(\sO_S(sC_0) \oplus \sO_S(t f))\cong\det(\sK\otimes\sO_S(-aC_0-b f))$,
it is in fact an  isomorphism.
We are in case (4)  in the statement of Proposition~\ref{proposition:P-bdle_over_S_hirzebruch}.

\medskip

\noindent {\bf Case 3:} Suppose that  $\Pi$ contains a member of the form 
$\Sigma+vq^*B_2$, where $\Sigma$ is irreducible and has relative degree $1$ over $S$, $B_2\neq 0$,
and  $v >0$.

Up to replacing $C_0$ and $f$ with linearly equivalent  curves on $S$, we may write
$B=C_0+f$.
It follows from claim (b) that  $v=t$.

As in case (2), we see that any other reducible divisor of $\Pi$ must be of the form 
$\Sigma'+sq^*C_0$, where $\Sigma'$ is irreducible and has relative degree $1$ over $S$. 
If there exists such divisor, we are in case 2.3 above.
So we may assume that $\Sigma+tq^*f$ is the only reducible member of $\Pi$.
This implies that $s=1$, and $\Sigma$ gives rise to an exact sequence
$$
0 \ \to \ \sO_S(tf) \ \to \
\sK\otimes\sO_S(-aC_0-b f) \ \to \ \sI_{\Lambda}\sO_S(C_0) \ \to \ 0,
$$
where $\Lambda \subset S$ is a closed subscheme with $\textup{codim}_S\Lambda \ge 2$.
If $\Lambda=\emptyset$, then the sequence splits since 
$h^1(S,\sO_S(-C_0+tf))=0$, and 
we are in case (4)  in the statement of Proposition~\ref{proposition:P-bdle_over_S_hirzebruch}, with $s=1$.
If $\Lambda\neq\emptyset$, then $\Lambda$ is a local complete intersection subscheme, and we are in case
(6)   in the statement of Proposition~\ref{proposition:P-bdle_over_S_hirzebruch}, with $t \ge 1$.

\medskip

\noindent {\bf Case 4:} Suppose that all members of $\Pi$ are irreducible. 

Then $s=1$ and $t\le 1$.
Let $\Sigma'$ be the closure of a general leaf of $\sG$.
It gives rise to an exact sequence
$$
0 \ \to \ \sO_S \ \to \
\sK\otimes\sO_S(-aC_0-b f) \ \to \ \sI_{\Lambda}\sO_S(C_0+tf) \ \to \ 0,
$$
where 
$\Lambda \subset S$ is a closed subscheme with $\textup{codim}_S\Lambda \ge 2$.
We claim that $\Lambda\neq\emptyset$.
Indeed, if $\Lambda=\emptyset$, then the sequence splits since $h^1(S,\sO_S(-C_0-f))=0$. 
But this implies that $\Pi$ contains a reducible member, contrary to our assumptions. 
Hence $\textup{codim}_S\Lambda = 2$, and $\Lambda$ is a local complete intersection subscheme.

If $t=0$, then we are in case
(6)   in the statement of Proposition~\ref{proposition:P-bdle_over_S_hirzebruch}.

Suppose from now on that $t=1$. 
Then we must have $h^0(S,\sI_{\Lambda}\sO_{S}(C_0+f))\ge 1$.
We will show that, there exists a curve $C \sim C_0+f$ with $\Lambda \subset C$
such that, for any proper subcurve $C' \subsetneq C$, $\Lambda\not\subset C'$.
Suppose to the contrary that any curve $C \sim C_0+f$ with $\Lambda \subset C$ 
can be written as $C=C_1\cup f_1$ with $C_1\sim C_0$,  $f_1\sim f$,
and either $\Lambda \subset C_1$, or $\Lambda \subset f_1$.
This implies that the set of reducible members of $|\sO_Y(1)\otimes q^*\sO_S(-aC_0-b f)|$ has
codimension 1 since
$h^0(S,\sO_{S}(-f))=0=h^0(S,\sO_{S}(-C_0))$. This contradicts the fact that all members of $\Pi$ are irreducible.
We are in case (7)   in the statement of Proposition~\ref{proposition:P-bdle_over_S_hirzebruch}.
\end{proof}

Next we investigate whether all the 7 cases described in Proposition~\ref{proposition:P-bdle_over_S_hirzebruch} in fact occur.

\begin{say}[{Proposition~\ref{proposition:P-bdle_over_S_hirzebruch}(1--4)}]
Let $\sK$ be a rank $2$ vector bundle on a ruled surface $p : \mathbb{F}_e\to\p^1$, $e \ge 0$.
Set $Y:=\p_{\mathbb{F}_e}(\sK)$, 
with natural projection $q : Y \to \mathbb{F}_e$, and denote 
by $\sO_{Y}(1)$ the tautological line bundle on $Y$.
Suppose that $\sK$ satisfies one of  the conditions (1-4) in the statement of 
Proposition~\ref{proposition:P-bdle_over_S_hirzebruch}. 
Then the pencil $\Pi$ described in the statement yields a codimension one foliation
$\sG$ on $Y$ with $\det(\sG)\cong q^*\sA$ where
$\sA$ is an ample line bundle.
\end{say}

\begin{say}[{Proposition~\ref{proposition:P-bdle_over_S_hirzebruch}(5)}]
Consider the ruled surface $p : \mathbb{F}_e\to\p^1$, with $e \ge 0$, 
denote by $C_0$ a minimal section, and by $f$ a fiber of $p$. 
Let $s, \lambda \ge 0$ be integers, and suppose that $\sK$ is a coherent sheaf on $\mathbb{F}_e$ 
fitting into an exact sequence
\begin{equation}\label{eq:sequence1}
0 \ \to \ \sK \ \to \ \sO_{\mathbb{F}_e}(f)\oplus \sO_{\mathbb{F}_e}(sC_0+f)
 \ \to \ \sO_{f_0}(s+\lambda+1) \ \to \ 0.
\end{equation}

By \cite[Proposition 5.2.2]{HuyLehn}, $\sK$ is a rank $2$ vector bundle on $\mathbb{F}_e$.
Since $\det(\sK)\cong \sO_{\mathbb{F}_e}(sC_0+f)$, we have $\sK^*\cong \sK\otimes \sO_{\mathbb{F}_e}(-sC_0-f)$.
Dualizing sequence~\eqref{eq:sequence1}, and twisting it  with $\sO_{\mathbb{F}_e}(sC_0+f)$
yields
\begin{equation}\label{eq:sequence2}
0\ \to \ \sO_{\mathbb{F}_e}\oplus \sO_{\mathbb{F}_e}(sC_0) \ \to \ 
\sK
\ \to \ 
\sO_{f_0}(-\lambda)
\to \ 0
\end{equation}
Conversely, dualizing sequence~\eqref{eq:sequence2}, and twisting it with $\sO_{\mathbb{F}_e}(sC_0+f)$
yields sequence~\eqref{eq:sequence1}.

Set $Y:=\p_{\mathbb{F}_e}(\sK)$, with natural projection $q : Y \to \mathbb{F}_e$
and tautological line bundle $\sO_{Y}(1)$.
Let $\Sigma$ be the zero locus of the section of 
$\sO_{Y}(1)\otimes q^*\sO_{\mathbb{F}_e}(-sC_0)$
corresponding to the map 
$\sO_{\mathbb{F}_e}(sC_0) \to \sO_{\mathbb{F}_e}\oplus \sO_{\mathbb{F}_e}(sC_0) \to \sK$
induced by \eqref{eq:sequence2}. 
Similarly, let 
$\Sigma'$ be the 
zero locus of the section of 
$\sO_{Y}(1)$
corresponding to the map 
$\sO_{\mathbb{F}_e} \to \sO_{\mathbb{F}_e}\oplus \sO_{\mathbb{F}_e}(sC_0) \to \sK$
induced by \eqref{eq:sequence2}. 
Let $\Pi$ be the pencil in $|\sO_Y(1)|$ generated by $\Sigma+sq^*C_0$ and $\Sigma'$.

If $\Sigma+sq^*C_0$ is the only reducible member of $\Pi$, then this pencil 
induces a foliation $\sG$ on $Y$ as in Proposition \ref{proposition:P-bdle_over_S_hirzebruch}(5).
So we investigate this condition. 
Suppose that there exists another reducible divisor in $\Pi$, and write it as $\Sigma''+q^*D\neq \Sigma+sq^*C_0$,
where $D\sim uC_0+vf$ is a nonzero effective divisor on $\mathbb{F}_e$.
By restricting $\Pi$ to $Y_f=q^{-1}(f)$, we see that $u=0$ and $v>0$. 
Thus
$h^0(\mathbb{F}_e,\sK\otimes\sO_{\mathbb{F}_e}(-vf_0))\ge 1$. On the other hand, 
$h^0(\mathbb{F}_e,\sO_{\mathbb{F}_e}(-vf_0)\oplus \sO_{\mathbb{F}_e}(sC_0-vf_0))=0$.
This implies $\lambda = 0$, and thus 
$\sK \cong \sO_{\mathbb{F}_e}(f)\oplus \sO_{\mathbb{F}_e}(sC_0)$.  
\end{say}

\begin{say}[{Proposition~\ref{proposition:P-bdle_over_S_hirzebruch}(6)}]
\label{rem:existence}
Consider the ruled surface $p : \mathbb{F}_e\to\p^1$, with $e \ge 0$, 
denote by $C_0$ a minimal section, and by $f$ a fiber of $p$. 
Let $\Lambda \subset \mathbb{F}_e$ be a local complete intersection subscheme of codimension $2$, and $t \ge 0$ an integer.
By \cite[Theorem 5.1.1]{HuyLehn}, there exists a vector bundle $\sK$ on $\mathbb{F}_e$ fitting into an exact sequence
$$
0 \ \to \ \sO_{\mathbb{F}_e}(tf) \ \to \
\sK\ \to \ \sI_\Lambda\sO_{\mathbb{F}_e}(C_0) \ \to \ 0.
$$
Set $Y:=\p_{\mathbb{F}_e}(\sK)$, with natural projection $q : Y \to \mathbb{F}_e$ and tautological line bundle $\sO_{Y}(1)$.

Note that the map
$H^0(\mathbb{F}_e,\sK)
\to H^0(\mathbb{F}_e,\sI_\Lambda\sO_{\mathbb{F}_e}(C_0))$
is surjective since 
$H^1(\mathbb{F}_e,\sO_{\mathbb{F}_e}(tf))$ vanishes.
Suppose moreover that 
$h^0(\mathbb{F}_e,\sI_\Lambda\sO_{\mathbb{F}_e}(C_0))\ge 1$.
Let $s'\in H^0(Y,\sO_Y(1))$
be a section mapping to a nonzero section in 
$H^0(\mathbb{F}_e,\sI_\Lambda\sO_{\mathbb{F}_e}(C_0))$, and denote by $\Sigma'$ its zero locus.
We claim that $\Sigma'$ is irreducible.
Indeed, if $\Sigma'$ is reducible, then, up to replacing $C_0$ by a linearly equivalent curve, we see that 
$s'$ must vanish along $q^{-1}(C_0)$. 
Therefore, 
$h^0(\mathbb{F}_e,\sK\otimes\sO_{\mathbb{F}_e}(-C_0))\ge 1$, 
and hence 
$h^0(\mathbb{F}_e,\sO_{\mathbb{F}_e}(tf-C_0))\ge 1$, which is absurd. 
This shows that $\Sigma'$ is irreducible.
Let $\Sigma$ be the zero locus of the section of 
$\sO_{Y}(1)\otimes q^*\sO_{\mathbb{F}_e}(-tf)$
corresponding to $\sO_{\mathbb{F}_e}(tf) \to \sK$.  
Then the pencil $\Pi\subset |\sO_Y(1)|$ generated by $\Sigma+tq^*f_0$ and $\Sigma'$
induces a foliation $\sG$ on $Y$ as in Proposition \ref{proposition:P-bdle_over_S_hirzebruch}(6).
\end{say}

\begin{say}[{Proposition~\ref{proposition:P-bdle_over_S_hirzebruch}(7)}]

Consider the ruled surface $p : \mathbb{F}_e\to\p^1$, with $e \ge 0$, 
denote by $C_0$ a minimal section, and by $f$ a fiber of $p$. 
Let $\Lambda \subset \mathbb{F}_e$ be a local complete intersection subscheme of codimension $2$.
By \cite[Theorem 5.1.1]{HuyLehn}, there exists a vector bundle $\sK$ on $\mathbb{F}_e$ fitting into an exact sequence
$$
0 \ \to \ \sO_{\mathbb{F}_e} \ \to \
\sK\ \to \ \sI_\Lambda\sO_{\mathbb{F}_e}(C_0+f) \ \to \ 0.
$$
Set $Y:=\p_{\mathbb{F}_e}(\sK)$, with natural projection $q : Y \to \mathbb{F}_e$ and tautological line bundle $\sO_{Y}(1)$.

Note that the map
$H^0(\mathbb{F}_e,\sK)
\to H^0(\mathbb{F}_e,\sI_\Lambda\sO_{\mathbb{F}_e}(C_0+f))$
is surjective since 
$H^1(\mathbb{F}_e,\sO_{\mathbb{F}_e})$ vanishes.
We assume moreover that $h^0(\mathbb{F}_e,\sI_\Lambda\sO_{\mathbb{F}_e}(C_0+f))\ge 1$, 
and that there exists a curve $C \sim C_0+f$ with $\Lambda \subset C$ and
such that, for any proper subcurve $C' \subsetneq C$, $\Lambda\not\subset C'$.

Let $\Sigma$ be the zero locus of the section of $\sO_{Y}(1)$
corresponding to the map
$\sO_{\mathbb{F}_e} \to \sK$.
Then $\Sigma$ is a section of $q$ over 
$q^{-1}(\textup{Supp}(\Lambda))$, and hence it is irreducible.
Let $\Sigma'$ be the zero locus of the section $s'\in H^0(\mathbb{F}_e,\sO_Y(1))$
that lifts the section of 
$H^0(\mathbb{F}_e,\sI_\Lambda\sO_{\mathbb{F}_e}(C_0+f))$ whose divisor of zeroes
is $C$.
We claim that $\Sigma'$ is irreducible. Suppose otherwise. Then $s'$ must vanish along a subcurve 
$C'$ of $C$. This implies that 
$h^0(S,\sI_\Lambda\sO_{\mathbb{F}_e}(C_0+f-C'))\neq 0$.
Therefore,
$C'$ is a proper subcurve of $C$, 
and $\Lambda\subset C''$ where $C''$ is such that $C=C'\cup C''$, 
contrary to our assumptions. 
This shows that $\Sigma'$ is irreducible.
The pencil $\Pi\subset |\sO_Y(1)|$
generated by $\Sigma$ and $\Sigma'$ 
has only  irreducible members, and induces 
a foliation $\sG$ on $Y$ as in Proposition \ref{proposition:P-bdle_over_S_hirzebruch}(7).
\end{say}

Now we go back to the problem of describing 
$\sK$ and $\sG$ that appear in Proposition~\ref{proposition:P-bdle_over_S}. 
It remains to consider the case when $S\cong \p^2$, and 
$\det(\sG)\cong q^*\sO_{\p^2}(a)$ with $a\in\{1,2\}$.
Proposition \ref{proposition:P-bdle_over_S_plane} below addresses the case $a=2$,
while Proposition \ref{proposition:P-bdle_over_S_plane_2} addresses the case $a=1$.

\begin{prop}\label{proposition:P-bdle_over_S_plane}
Let
$\sK$ be a rank $2$ vector bundle on $\p^2$. Set $Y:=\p_{\p^2}(\sK)$, with natural projection $q : Y \to \p^2$
and tautological line bundle $\sO_{Y}(1)$. 
Let $\sG\subset T_Y$ be a codimension one foliation on $Y$ with $\det(\sG)\cong q^*\sO_{\p^2}(2)$.
Then one of the following holds.

\begin{enumerate}
\item There exist integers
$a$ and $s$, with  $s \ge 1$, such that $\sK\cong \sO_{\p^2}(a)\oplus\sO_{\p^2}(a+s)$, and 
$\sG$ is induced by a pencil in $|\sO_{Y}(1)\otimes q^*\sO_{\p^2}(-a)|$
containing a divisor of the form $\Sigma+sq^*\ell_0$, where $\Sigma$ is the section of 
$q$ corresponding to the map $\sO_{\p^2}(a)\oplus\sO_{\p^2}(a+s)\twoheadrightarrow \sO_{\p^2}(a)$,
and $\ell_0\subset \p^2$ is a line.

\item There exist an integer $a$ and a local complete intersection subscheme
$\Lambda \subset \p^2$ of codimension $2$ such that
$h^0(\p^2,\sI_\Lambda\sO_{\p^2}(1))\ge 1$,  $\sK$ fits into 
an exact sequence
$$
0 \ \to \ \sO_{\p^2} \ \to \
\sK\otimes\sO_{\p^2}(-a) \ \to \ \sI_\Lambda\sO_{\p^2}(1) \ \to \ 0,
$$
and $\sG$ is induced by a pencil of irreducible members of $|\sO_{Y}(1)\otimes q^*\sO_{\p^2}(-a)|$
containing the zero locus of the section of 
$\sO_{Y}(1)\otimes q^*\sO_{\p^2}(-a)$
corresponding to the map 
$\sO_{\p^2} \to \sK\otimes\sO_{\p^2}(-a)$.
\end{enumerate}
\end{prop}

\begin{proof}
The proof is very similar to that of Proposition \ref{proposition:P-bdle_over_S_hirzebruch}, 
and so we leave some easy details to the reader. 
To ease notation, set $S:=\p^2$, and write $\sO_S(1)$ for $\sO_{\p^2}(1)$.

Since $T_{Y/S}\not \subseteq \sG$, the natural map $T_Y \to q^*T_{S}$
induces an injective morphism of sheaves $\sG \to q^*T_{S}$.
There is a line $\ell_0 \subset \p^2$ such that  the divisor of zeroes of 
the induced map $q^*\sO_{\p^2}(2)\cong \det(\sG) \to q^*\det(T_{S})$ is 
$q^*\ell_0$. Note that $\sG$ induces a flat connection 
on $q : Y \to S$ over $q^{-1}(\textup{Supp}(\ell_0))$.

Let $\ell \subset \p^2$ be a general line, set
$Y_\ell:=q^{-1}(\ell)$ and $q_\ell := q_{|Y_\ell} : Y_\ell \to \ell$.
Since $T_{Y/S}\not \subseteq \sG$, $\sG$ induces a foliation by curves $\sC_\ell\subset T_{Y_\ell}$.
As in the claim in the proof of  Proposition \ref{proposition:P-bdle_over_S_hirzebruch}, one checks that 
$\sC_\ell\cong q_\ell^*\sO_{\p^1}(1)$, 
$Y_\ell\cong \mathbb{F}_s$ with $s\ge 1$, 
and $\sC_\ell$ is induced by a pencil  containing $\sigma_0+s f$, 
where $\sigma_0$ denotes the minimal section, and $f$ a fiber of $q_\ell : Y_\ell \to \ell$. 

One then shows that $\sG$ has algebraic leaves, and it is induced by a pencil $\Pi\subset |\sO_Y(1)\otimes q^*\sO_S(-a)|$,
where $a$ is such that $\sK_{|\ell}\cong\sO_{\p^1}(a)\oplus\sO_{\p^1}(a+s)$.

Any member of $\Pi$ is of the form $\Sigma+uq^*{\ell_0}$, where 
$\Sigma$ is irreducible and has relative degree $1$ over $S$, 
and  $u\ge 0$ is an integer.
In particular, the ramification divisor of the rational first integral 
for $\sG$, $\pi : Y \dashrightarrow  \p^1$,
must be of the form $R(\pi)=cq^*\ell_0$, $c \ge 0$.
An easy computation shows that $c=s-1$.
In particular, if $s \ge 2$, then $\Pi$  contains a member of the form 
$\Sigma+sq^*\ell_0$, where 
$\Sigma$ is irreducible and has relative degree $1$ over $S$.

\medskip

\noindent {\bf Case 1:} 
Suppose that $\Pi$  contains a member of the form $\Sigma+uq^*\ell_0$, where 
$\Sigma$ is irreducible and has relative degree $1$ over $S$, 
and  $u\ge 0$ is an integer.
It follows from the description of $\sC_\ell$ above that $u=s$, and 
$\Sigma\cap Y_\ell$ is the minimal section of $q_\ell:Y_\ell\cong \mathbb{F}_s\to \ell$.
If $\Sigma'$ is the closure of a general leaf of $\sG$, then 
$\Sigma\cap \Sigma'\cap Y_\ell= \emptyset$.
This implies that $\Sigma\cap \Sigma' \cap q^{-1}(b)=\emptyset$ for a general point $b \in\ell_0$.
One can find an open subset $V \subset S$, 
with $\textup{codim}_S(S\setminus V)\ge 2$, such that $\Sigma\cap q^{-1}(V)$
and $\Sigma'\cap q^{-1}(V)$ are sections of $q_{|q^{-1}(V)}$, and
$\Sigma\cap\Sigma'\cap q^{-1}(V)=\emptyset$. 
Therefore, there are line bundles $\sB_1$ and $\sB_2$ on $S$ such that 
$\sK \cong \sB_1\oplus \sB_2$, and $\Sigma$ corresponds to the surjection 
$\sB_1\oplus\sB_2\twoheadrightarrow \sB_1$. 
From the description of $\sK_{|\ell}$, we see that 
$\sB_1\cong\sO_S(a)$, and $\sB_2\cong\sO_S(a+s)$. 
This is case (1)  in the statement of Proposition~\ref{proposition:P-bdle_over_S_plane}.

\medskip

\noindent {\bf Case 2:} 
Suppose then that all members of $\Pi$ are irreducible. 
In particular, we must have $s=1$.
Then the section $\Sigma$ gives rise to an exact sequence
$$
0 \ \to \ \sO_S \ \to \
\sK\otimes\sM \ \to \ \sI_{\Lambda}\sO_S(1) \ \to \ 0,
$$
where 
$\Lambda \subset S$ is a closed subscheme with $\textup{codim}_S\Lambda \ge 2$.
If $\Lambda=\emptyset$, then the sequence splits since 
$h^1(S,\sO_S(-1))=0$. But then $\Pi$ contains a reducible member, a contradiction.
Thus $\Lambda\neq\emptyset$, and $\Lambda$ is a local complete intersection subscheme.
This is case (2)  in the statement of Proposition~\ref{proposition:P-bdle_over_S_plane}.
\end{proof}

Next we give examples of foliations of the type described in Proposition \ref{proposition:P-bdle_over_S_plane}(2).

\begin{say}Let $\Lambda \subset \p^2$ be a local complete intersection subscheme of codimension 2. 
By \cite[Theorem 5.1.1]{HuyLehn}, there exists a vector bundle $\sK$ on $\p^2$ fitting into an exact sequence
$$
0 \ \to \ \sO_{\p^2} \ \to \
\sK\ \to \ \sI_\Lambda\sO_{\p^2}(1) \ \to \ 0.
$$
Set $Y:=\p_{\mathbb{F}_e}(\sK)$, with natural projection $q : Y \to \p^2$ and tautological line bundle $\sO_{Y}(1)$.

Note that the map
$H^0(\p^2,\sK)
\to H^0(\p^2,\sI_\Lambda\sO_{\p^2}(1))$
is surjective since 
$H^1(\p^2,\sO_{\p^2})$ vanishes.
We assume moreover that $h^0(\p^2,\sI_\Lambda\sO_{\p^2}(1))\ge 1$, 
and let $\ell_0$ be a line in $\p^2$ such that $\Lambda \subset \ell_0$.
Let $\Sigma$ be the zero locus of the section of 
$\sO_{Y}(1)$
corresponding to 
$\sO_{\p^2} \to \sK$.
Then $\Sigma$ is a section of $q$ over 
$q^{-1}(\textup{Supp}(\Lambda))$, and hence it is irreducible.
Let $s'\in H^0(\p^2,\sO_Y(1))$
be a section lifting the section 
$H^0(\p^2,\sI_\Lambda\sO_{\p^2}(1))$ corresponding
to $\ell_0$, and denote by $\Sigma'$ its zero locus.
We claim that $\Sigma'$ is irreducible. 
Indeed, if $\Sigma'$ is reducible, then $s'$ must vanish along $\ell_0$. 
On the other hand, $h^0(\p^2,\sO_{\p^2}(-1))=h^0(\p^2,\sI_\Lambda)=0$, and hence
$h^0(\p^2,\sK\otimes\sO_{\p^2}(-1))=0$, yielding a contradiction.

Therefore, the pencil $\Pi\subset |\sO_Y(1)|$
generated by $\Sigma$ and $\Sigma'$ induces 
a foliation $\sG$ on $Y$ as in Proposition \ref{proposition:P-bdle_over_S_plane}(2).
\end{say}

\begin{exmp}
Set $Y:=\p_{\p^2}(T_{\p^2})$, denote by $q : Y \to \p^2$ the natural projection, and by  $\sO_Y(1)$ the
tautological line bundle. 
Let $p : Y \to \p^2$ be the morphism
induced by the linear system $|\sO_Y(1)\otimes q^*\sO_{\p^2}(-1)|$.
If $\sC\subset T_{\p^2}$ is a degree zero foliation, then $\sG:= p^{-1}(\sC)$
is a codimension one foliation on $Y$ with 
$\det(\sG)\cong q^*\sO_{\p^2}(2)$ as in
Proposition \ref{proposition:P-bdle_over_S_plane}(2). 
Let $F$ be the closure of a leaf of $\sG$.
Then $q_{|F}: F\to \p^2$
is the blow up of $\p^2$ at a point on the line
$q(p^{-1}(\textup{Sing}(\sC)))$. 
\end{exmp}

\begin{prop}\label{proposition:P-bdle_over_S_plane_2}
Let
$\sK$ be a rank $2$ vector bundle on $\p^2$. Set $Y:=\p_{\p^2}(\sK)$, with natural projection $q : Y \to \p^2$
and tautological line bundle $\sO_{Y}(1)$. 
Let $\sG$ be a codimension one foliation on $Y$ with $\det(\sG)\cong q^*\sO_{\p^2}(1)$, 
and suppose that $\sG$  is not algebraically integrable.
Then there exist
\begin{itemize}
	\item a rational map $\psi: Y\map S=\p_\ell(\sK_{|\ell})$, where $\ell \subset \p^2$ is a general line, 
		giving rise to a foliation by rational curves $\sM\cong q^*\sO_{\p^2}(1)$ on $Y$, which lifts a degree zero foliation on $\p^2$;
	\item a rank $1$ foliation $\sN$ on $S$ induced by a global vector field;
\end{itemize}
such that $\sG$ is the pullback of $\sN$ via $\psi$.
\end{prop}

\begin{proof}
First note that $T_{Y/\p^2}\not \subseteq \sG$. 
Indeed, if $T_{Y/\p^2}\subseteq \sG$,
then $\sG$ would be the pullback via $q$ of a foliation on $\p^2$.
Denote by $f\cong\p^1$ a general fiber of $q : Y \to \p^2$.
Then, by \ref{pullback_foliations},
$\sO_{\p^1}\cong \det(\sG)_{|f}\cong (T_{Y/S})_{|f}\cong \sO_{\p^1}(2)$
 which is absurd.

Let $\sA$ be a very ample line bundle on $Y$. 
Since $\sG$ is not algebraically integrable, by \cite[Proposition 7.5]{fano_fols}, there exists an algebraically integrable subfoliation by curves 
$\sM \subset \sG$ such that 
$\sM \cdot \sA^2 \ge \det(\sG) \cdot \sA^2 \ge 1$. 
Moreover, $\sM$ does not depend on the choice
of the very ample line bundle $\sA$ on $Y$.  
Therefore 
$\sM \cdot (q^*\sO_{\p^2}(k)\otimes\sA)^2 
\ge \det(\sG)\cdot (q^*\sO_{\p^2}(k)\otimes\sA)^2 >0$ for all $k\ge 1$, and
hence $\sM \cdot f \ge 0$.

So we can write $\sM\cong \sO_Y(a)\otimes q^*\sO_{\p^2}(b)$ for some integers $a$ and $b$, with $a\ge 0$.
Since $T_{Y/\p^2}\not \subseteq \sG$, there exists an injection of sheaves 
$\sM \to q^*T_{\p^2}$, and hence we must have $a=0$.
On the other hand, since $\sM \cdot \sA^2 \ge 1$, we must have $b \ge 1$. 
Since the map $\sM \to q^*T_{\p^2}$ induces a nonzero map $\sO_{\p^2}(b) \to T_{\p^2}$, we conclude that 
$b = 1$ by Bott's formulas.  So $\sM\cong q^*\sO_{\p^2}(1)$.

Let $p \in \p^2$ be the singular locus of the degree zero foliation $\sO_{\p^2}(1) \subset T_{\p^2}$
induced by the map $\sM \to q^*T_{\p^2}$.
Then $\sM \subset T_Y$ is a regular foliation (with algebraic leaves) over $q^{-1}(\p^2\setminus \{p\})$. 
By Lemma \ref{lemma:foliation_surface}, a general leaf of $\sM$ maps isomorphically  to a line in $\p^2$ through the point $p$. 
This implies that the space of leaves of $\sM$ can be naturally identified with 
$S=\p_\ell(\sK_{|\ell})$, where $\ell \subset \p^2$ is a general line.
Moreover, the natural morphism
$q^{-1}(\p^2\setminus \{p\}) \to S$ is smooth. Hence, by \eqref{pullback_foliations}, $\sG$ is the pullback of a rank $1$ foliation 
$\sN\cong \sO_S$ on $S$.
 \end{proof}

\begin{rem}

Let $\sE$ be a vector bundle on a smooth complex variety $Z$.
Set $Y:=\p_{Z}(\sE)$, 
with natural projection $q : Y \to Z$.
Let $\sW$ be a line bundle on $Z$, and $V\in H^0(Z,T_Z\otimes\sW)$
a twisted vector field on $Z$.
By \cite[Proposition 1.1]{carrell_lieberman} and \cite[Lemma 9.5]{fano_fols},  
the map $\sW^* \subset T_Z$ induced by $V$ lifts to a 
map $q^*\sW^* \to T_{Y}$ if and only if
$\sE$ is $V$-equivariant, i.e., if there exists a $\mathbb C$-linear map
$\tilde V: \sE\to \sW\otimes\sE$ lifting the derivation $V:\sO_T\to \sW$.
\end{rem}

\begin{exmp}

Set $Y:=\p_{\p^2}(\sO_{\p^2}\oplus \sO_{\p^2}(1))$, 
with natural projection $q : Y  \to \p^2$ and  tautological line bundle $\sO_Y(1)$.
Let $p : Y \to \p^3$ be the morphism 
induced by the linear system $|\sO_Y(1)|$. It is the blow-up of 
$\p^3$ at  a point $x$. 
Let $y \in \p^3\setminus \{x\}$, and denote by $\varpi: \p^3 \dashrightarrow \p^2$
the linear projection from $y$.
Let $\sC\cong \sO_{\p^2}\subset T_{\p^2}$ be a 
degree one foliation on $\p^2$, singular at the point $\varpi(x)$,
and let $\sG\subset T_Y$ be the pullback of $\sC$ via $\varpi\circ p$. 
It is a codimension one foliation on $Y$.
An easy computation shows that $\det(\sG)\cong q^*\sO_{\p^2}(1)$.
The rational map
$\varpi: \p^3 \dashrightarrow \p^2$ induces a foliation by curves 
$\sM\cong q^*\sO_{\p^2}(1)$ on $Y$, which lifts 
the degree zero foliation on $\p^2$
given by the linear projection from the point
$q(p^{-1}(y))\in \p^2$.
The space $T$ of leaves of $\sM$ can be naturally identified with  $S=\mathbb{F}_1$,
and $\sG$ is the pullback via the induced rational map 
$Y \map S$ of a foliation induced by a global vector field on $S$.

\end{exmp}

% 4.5.  Proof of Main Theorem

\subsection{{Proof of Theorem~\ref{main_thm_rho>1}}} \label{subsection:proof_main_rho>0}

\

Let $X$, $\sF$ and $L$ be as in Assumptions~\ref{assumptions}. 
By Theorem~\ref{Thm:KX-KF_not_nef}, $K_X+(n-3)L$ is not nef, i.e., 
$\tau(L)>n-3$.
By Theorem~\ref{tironi}, $\tau(L)\in \{n-2, n-1, n\}$, unless 
$\big(n,\tau(L)\big)\in \big\{(5,\frac{5}{2}),(4,\frac{3}{2}),(4,\frac{4}{3}) \big\}$.

\medskip 

\noindent {\bf Step 1: } We show that $\tau(L)\ge n-2$.

\medskip

Suppose to the contrary that $\big(n,\tau(L)\big)\in \big\{(5,\frac{5}{2}),(4,\frac{3}{2}),(4,\frac{4}{3}) \big\}$
(Theorem~\ref{tironi}(4--6)).

In cases (4), (5b) and (6) described in Theorem~\ref{tironi}, 
$\varphi_L$ makes $X$ a fibration over  a smooth curve $C$. 
Denote by $F$ the general fiber of $\varphi_L$, which is either a projective space or a quadric. 
By Proposition~\ref{prop:Fano_fol_not_fibration}, $\sF\neq T_{X/C}$.
Therefore, if $\ell\subset X$ is a general line on $F$, then $\ell$ is not tangent to $\sF$.
By Lemma \ref{lemma:bound_on_pseudo_index}, 
$$
(n-3)L\cdot \ell \ = \ -K_\sF \cdot\ell \  \le \ -K_F \cdot\ell -2.
$$
One easily checks that this inequality is violated for those $(X,L)$
in Theorem~\ref{tironi}(4), (5b) and (6), yielding a contradiction.

It remains to consider cases (5a) and (5c) described in Theorem~\ref{tironi}.
In both cases, $X$ admits a morphism $\pi:X\to S$ onto a normal surface with general fiber 
$F\cong \p^{2}$, and $\sL_{|F}\cong \sO_{\p^2}(2)$.
Let $\ell\subset X$ be a general line on $F\cong \p^{2}$.
Then $L\cdot \ell = 2$.
It follows from Lemma \ref{lemma:bound_on_pseudo_index} that 
$\ell$ is tangent to $\sF$. 
By \ref{pullback_foliations}, $\sF$ is the pullback via $\pi$ 
of a foliation by curves $\sG$ on $S$. 
Thus
$$
L_{|F} \ = \ (-K_{\sF})_{|F}  \ = \ (-K_{X/S})_{|F},
$$
which is a contradiction. 

We conclude that $\tau(L)\ge n-2$.

\medskip

\noindent {\bf Step 2:} We show that $\tau(L)\ge n-1$.

\medskip 

Suppose to the contrary that $\tau(L)= n-2$.
Then either 
\begin{itemize}
	\item $(X,L)$ is as in Theorem~\ref{tironi}(3a--d); or 
	\item $\varphi_L: X\to X'$ is birational.
\end{itemize}

Suppose that  $(X,L)$ is one of the pairs described in Theorem~\ref{tironi}(3a--d) and Theorem~\ref{Thm:Classification_Mukai}.
Then $X$ admits a morphism $\pi:X\to Y$ onto a normal
variety of dimension $d$, $1\le d\le 3$, with general fiber 
$F$  a Fano manifold of dimension $n-d$, index 
$\iota_F=n-2$, and 
\begin{equation}\label{-KF=(n-2)L}
-K_F\ = \ (n-2)L_{|F}.
\end{equation}
Since $\iota_F\ge \dim(F)-1$, $F$ is covered by rational curves of $L$-degree $1$. 
So we can apply Lemma \ref{lemma:bound_on_pseudo_index}, and conclude that $F$
is tangent to $\sF$.
By  \ref{pullback_foliations}, $\sF$ is the pullback via $\pi$ 
of a codimension $1$ foliation $\sG$ on $Y$. So
$$
(n-3) L_{|F} \ = \ (-K_{\sF})_{|F} \ = \ (-K_{X/Y})_{|F},
$$
which contradicts \eqref{-KF=(n-2)L}.

\medskip

Suppose now that $\varphi_L: X\to X'$ is birational.
By Proposition~\ref{prop:BS},
$\varphi_L$ is the composition of finitely many disjoint divisorial contractions
$\varphi_i:X\to X_{i}$ of the following types:
\begin{itemize}
	\item[(E)] $\varphi_i:X\to X_{i}$ is the blowup of a smooth curve $C_i\subset X_{i}$, 
		with exceptional divisor $E_i$.
		In this case $X_{i}$ is smooth, and the restriction of $\sL$ to a  fiber of $(\varphi_i)_{|E_i}:E_i\to C_i$
		is isomorphic to $\sO_{\p^{n-2}}(1)$.
		
	\item[(F)] $\varphi_i:X\to X_{i}$ contracts a divisor $F_i\cong \p^{n-1}$ to a singular point, and
		$$
		\big(F_i, \sN_{F_i/X}, \sL_{|F_i}\big)\cong \big(\p^{n-1}, \sO_{\p^{n-1}}(-2), \sO_{\p^{n-1}}(1)\big).
		$$
		In this case $X_i$ is $2$-factorial.  In even dimension it is  Gorenstein.
	\item[(G)] $\varphi_i:X\to X_{i}$ contracts a divisor $G_i\cong Q^{n-1}$ to a singular point, and
		$$
		\big(G_i, \sN_{G_i/X}, \sL_{|G_i}\big)\cong \big(Q^{n-1}, \sO_{Q^{n-1}}(-1), \sO_{Q^{n-1}}(1)\big).
		$$
		In this case $X_i$ factorial. 
\end{itemize}
In particular $X'$ is $\bQ$-factorial and terminal.

Set $L':=(\varphi_L)_*(L)$.
The Mukai foliation $\sF$ induces a foliation $\sF'$  on $X'$  such that $-K_{\sF'} \sim (n-3)L'$.

We claim that $K_{X'}+(n-3)L'$ is not pseudo-effective. 
To prove this, let
$\Delta\sim_{\mathbb{Q}} (n-3)L$ be an effective $\mathbb{Q}$-divisor on $X$ such that $(X,\Delta)$ is klt, and set 
$\Delta':=(\varphi_L)_*(\Delta)\sim_{\mathbb{Q}} (n-3)L'$. 
Since $-(K_X+\Delta)$ is $\phi_L$-ample,
$(X',\Delta')$ is also klt.
Suppose that  $K_{X'}+(n-3)L'\sim_{\mathbb{Q}}  K_{X'}+\Delta'$ is pseudo-effective.
Under these assumptions, \cite[Theorem 2.11]{codim_1_del_pezzo_fols} states that 
for any integral divisor $D$ on $X'$ such that $D\sim_{\mathbb{Q}}  K_{X'}+\Delta'$,
$h^0\big(X,\Omega_{X'}^{1}[\otimes]\sO_{X'}(-D)\big)=0$.
On the other hand, by \ref{q-forms}, 
$\sF'$ gives rise to a nonzero global section 
$\omega\in H^0\big(X,\Omega_{X'}^{1}[\otimes]\sO_{X'}\big(-(K_{X'}-(n-3)L')\big)\big)$,
yielding a contradiction and proving the claim. 
In particular,  $K_{X'}+(n-3)L'$ is not nef, and Proposition~\ref{prop:BS}
implies that one of the following holds. 
\begin{enumerate}
	\item $n=6$ and $\big(X',\sO_{X'}(L')\big) \cong \big(\p^6, \sO_{\p^6}(2)\big)$.
	\item $n=5$ and one of the following holds.
		\begin{enumerate}
			\item $\big(X',\sO_{X'}(L')\big) \cong \big(Q^5, \sO_{Q^5}(2)\big)$.
			\item $X'$ is a $\p^4$-bundle over a smooth curve, and the restriction of $\sO_{X'}(L')$ to a general fiber is $\sO_{\p^4}(2)$.
			\item $\big(X,\sO_{X}(L)\big) \cong \Big(\p_{\p^4}\big(\sO_{\p^4}(3)\oplus \sO_{\p^4}(1)\big), \sO_{\p}(1)\Big)$.
		\end{enumerate}
	\item $n=4$ and one of the following holds.
		\begin{enumerate}
			\item $\big(X',\sO_{X'}(L')\big) \cong \big(\p^4, \sO_{\p^4}(3)\big)$.
			\item $X'$ is a Gorenstein del Pezzo $4$-fold and $3L'\sim_{\bQ} -2K_{X'}$.
			\item $\varphi_{L'}$ makes $X'$ a generic quadric bundle over a smooth curve $C$, and for a general fiber $F\cong Q^{3}$ of
				$\varphi_{L'}$, $\sO_F\big(L'_{|F}\big)\cong \sO_{Q^{3}}(2)$.
			\item $\varphi_{L'}$ makes $X'$ a generic $\p^{2}$-bundle  over  a normal  surface $S$, and for a general fiber $F\cong \p^{2}$ 			
				of $\varphi_{L'}$, $\sO_F\big(L'_{|F}\big)\cong \sO_{\p^{2}}(2)$.	
			\item $\big(X',\sO_{X'}(L')\big) \cong \big(Q^4, \sO_{Q^4}(3)\big)$.
			\item $\varphi_L:X\to X'$ factors through $\tilde X$, the blowup of $\p^4$ along a cubic surface $S$ contained in 
				a hyperplane. 
				The exceptional locus of the contraction $\tilde X\to X'$ is the strict transform of the hyperplane of $\p^4$ containing $S$,
				and it is of type (F) above. 
			\item $\varphi_L:X\to X'$ factors through $\tilde X$, a conic bundle over $\p^3$. 
				The exceptional locus of the contraction $\tilde X\to X'$ consists of a single prime divisor of type (F) above. 
			\item $\varphi_{L'}$ makes $X'$ a $\p^{3}$-bundle  over  a smooth curve $C$, 		
				and for a general fiber $F\cong \p^{3}$ of
				$\varphi_{L'}$, $\sO_F\big(L'_{|F}\big)\cong \sO_{\p^{3}}(3)$.	
			\item $\big(X',\sO_{X'}(L')\big) \cong \big(\p^4, \sO_{\p^4}(4)\big)$.
			\item $X'\subset \p^{10}$ is a cone over $\big(\p^3, \sO_{\p^3}(2)\big)$ and $L'\sim_{\bQ}2 H$, where $H$ 
				denotes a hyperplane section in $\p^{10}$.
		\end{enumerate}
\end{enumerate}

If $X'$ is a Fano manifold with $\rho(X')=1$, then $\sF'$ is a Fano foliation with $-K_{\sF'} \sim (n-3)L'$.
By Theorem~\ref{Thm:ADK}, $\iota_{\sF'}\le n-1$, and equality holds only if $X'\cong \p^n$,
in which case $\sF'$ is induced by a pencil of hyperplanes.
As a consequence, $(X',L')$ cannot be as in (1), (2a), (3e) and (3i). 

Suppose that  $\big(X',\sO(L')\big) \cong \big(\p^4, \sO_{\p^4}(3)\big)$, i.e., $(X,L)$ is as in (3a).
Then $\sF'$ is induced by a pencil of hyperplanes in $\p^4$.
Denote by $H\cong \p^2$ the base locus of this pencil.
Since $X'\cong \p^4$ is smooth, by Proposition~\ref{prop:BS}, $\varphi_L: X\to \p^4$ is the blowup of finitely many 
disjoint smooth curves $C_i\subset \p^4$, $1\le i\le k$.
Denote by $E_i\subset X$ the exceptional divisor over $C_i$,
and by $F_i\cong \p^{n-2}$ a fiber of  $(\varphi_L)_{|E_i}$.
Let 
$$
\omega\in H^0\big(\p^4,\Omega^1_{\p^4}\otimes \sO_{\p^4}(2)\big) 
$$
be the $1$-form 
defining $\sF'$.
An easy computation shows that $(\varphi_L)^*\omega$ vanishes along $E_i$ (with multiplicity exactly $2$) if and only if $C_i\subset H$.
So $(\varphi_L)^*\omega$ induces a section that does not vanish in codimension $1$ 
$$
\omega_{\sF}\ \in \ H^0\big(X,\Omega^1_X\otimes (\varphi_L)^*(\sO_{\p^4}(2)) \otimes \sO_X(-\sum_{i=1}^{k}\epsilon_iE_i)\big),
$$
where $\epsilon_i=2$ if $C_i\subset H$, and  $\epsilon_i=0$ otherwise. 
This is precisely the $1$-form defining $\sF$.
Hence,
$$
N_\sF\cong \sO_X(-K_X+K_{\sF})\cong (\varphi_L)^*(\sO_{\p^4}(2)) \otimes \sO_X(-\sum_{i=1}^{k}\epsilon_iE_i),
$$
and thus
$$
\sO_{\p^{n-2}}(1) \cong \sO_X(-K_X+K_{\sF})_{|F_i}\cong \sO_{\p^{n-2}}(\epsilon_i),
$$
yielding a contradiction.
We conclude that  $(X',L')$ cannot be as in (3a).

\medskip

Next we consider the cases in which $X'$ admits a morphism $\pi:X'\to C$ onto a smooth curve, with general fiber
$F$ isomorphic to either $\p^{n-1}$ or $Q^{n-1}$
(these are cases (2b), (3c) and (3h)). 
By Proposition~\ref{prop:Fano_fol_not_fibration}, $\sF$ is not the relative tangent to the 
composed morphism $\varphi_L\circ \pi:X\to C$.
Hence, $\sF'\neq T_{X'/C}$, and a general line $\ell\subset F$ is not tangent to $\sF'$.
By Lemma \ref{lemma:bound_on_pseudo_index_bis}, 
$$
(n-3)L'\cdot \ell \ = \ -K_{\sF'} \cdot\ell \  \le \ -K_F \cdot\ell -2.
$$
One easily checks that this inequality is violated for those $(X',L')$
in cases (2b), (3c) and (3h).

\medskip

Next we show that $(X,L)$ cannot be as in (2c).
Suppose to the contrary that 
$$
\big(X,\sO_X(L)\big) \cong \Big(\p_{\p^4}\big(\sO_{\p^4}(3)\oplus \sO_{\p^4}(1)\big), \sO_{\p}(1)\Big).
$$
Let $\ell\subset X$ be a general fiber of the natural projection $\pi:X=\p_{\p^4}\big(\sO_{\p^4}(3)\oplus \sO_{\p^4}(1)\big)\to \p^4$.
Since $-K_{\sF} \cdot \ell=2$, $\ell$ is tangent to $\sF$ by Lemma \ref{lemma:bound_on_pseudo_index}.
By \ref{pullback_foliations}, $\sF$ is the pullback via $\pi$ of a codimension $1$ foliation $\sG$ on $\p^4$.
By \eqref{K_pullback_fol},  $\det(\sG)\cong \sO_{\p^4}(4)$, which is impossible 
by Theorem~\ref{Thm:ADK}.

\medskip

We show that $(X',L')$ cannot be as in (3b).
Suppose to the contrary that $X'$ is a  Gorenstein  del Pezzo $4$-fold and $3L'\sim_{\bQ} -2K_{X'}$.
Then there is an ample Cartier divisor $H'$ on $X'$ such that $L'\sim_{\bQ} 2H'$ and $-K_{X'}\sim 3H'$.
Notice that $X'$ has isolated singularities. 
Let $Y \in |H'|$ be a general member. 

We claim that $Y$ is a smooth $3$-fold. Suppose first
that $(H')^4 \ge 2$. Then $|H'|$ is basepoint free by \cite[Corollary 1.5]{fujita90}, and hence $Y$ is smooth by Bertini's Theorem. Suppose now that $(H')^4 = 1$. By \cite[Theorem 4.2]{fujita_classification} (see also 
\cite[6.3 and 6.4]{fujita_classification}), $\dim(\textup{Bs}(H'))\le 0$. Thus, if 
$H'_1,\ldots,H'_4$ are general members of $|H'|$, then 
$H'_1,\ldots, H'_4$ meets properly in a (possibly empty) finite set of points, and
$(H')^4=\deg(H'_1\cap\cdots\cap H'_4)$ (see \cite[Example 2.4.8]{fulton}).
Since $(H')^4=1$, $H'_1\cap\cdots\cap H'_4$ is a reduced point $x$,
$X'$ is smooth at $x$, and the local equations of 
$H'_1,\ldots, H'_4$ at $x$ form a regular sequence in $\sO_{X',x}$. In particular, 
$H'_i$ is smooth at $x$ for all $i\in\{1,\ldots,4\}$. 
By Bertini's Theorem again, we conclude that $Y$ is smooth. 

Set $H_Y:={H'}_{|Y}$, and denote by $\sH$ the codimension one foliation on $Y$ induced by $\sF'$.
By the adjunction formula, $-K_Y=2H_Y$, and hence $Y$ is a del Pezzo threefold.
By \ref{restricting_fols}, there exists a non-negative integer $b$ such that 
$
-K_{\sH} \ = \ (1+b) H_Y.
$
By Theorem~\ref{Thm:ADK}, we must have $b\in\{0,1\}$, and hence $\sH$ is a Fano foliation.
It follows from Thereom \ref{Thm:codim1_dP} and the classification of 
del Pezzo manifolds that
$Y \cong \p^3$, $b=0$, and $(H_Y)^3=(H')^4=8$. 
Therefore, $H'$ is very ample by \cite[6.2.3]{fujita_classification}, so that 
we can apply \cite[Theorem 3]{fujita82a} to conclude that one of the following holds.

\begin{itemize}
\item $X'$ is a cubic hypersurface in $\p^5$.
\item $X'$ is a complete intersection of two quadric hypersurfaces in $\p^6$.
\item $X'$ is a cone over a Gorenstein del Pezzo $3$-fold.
\item $\dim(\textup{Sing}(X'))\ge 1$. 
\end{itemize}
In the first two cases, $(H')^4=3$ and $(H')^4=4$, respectively. Since $X'$ has isolated singularities, we conclude that 
$X'$ must be a cone over $\big(\p^3,\sO_{\p^3}(2)\big)$.
Denote by $\pi:X\to  \p^3$ the induced map, and by $\ell$ a general fiber of $\pi$.
One computes that $L \cdot \ell =1$.
By Lemma \ref{lemma:bound_on_pseudo_index}, $\ell$ is tangent to $\sF$.
So, by  \ref{pullback_foliations}, $\sF$ is the pullback via $\pi$ 
of a codimension $1$ foliation $\sG$ on $\p^3$. Thus
$$
L_{|\ell} \ = \ (-K_{\sF})_{|\ell}  \ = \ (-K_{X/\p^3})_{|\ell} \ = \ -K_\ell,
$$
which is absurd. This shows that  $(X',L')$ cannot be as in (3b).

\medskip

We show that $(X',L')$ cannot be as in (3d).
Suppose to the contrary that $\varphi_{L'}$ makes $X'$ a generic $\p^{2}$-bundle  over  a normal  surface $S$, 
and for a general fiber $F\cong \p^{2}$ of $\varphi_{L'}$, $\sO_F\big(L'_{|F}\big)\cong \sO_{\p^{2}}(2)$.	
Lemma \ref{lemma:bound_on_pseudo_index_bis} implies that $F$ is tangent to $\sF'$.
By \ref{pullback_foliations}, $\sF'$ is the pullback via $\varphi_{L'}$
of  foliation by curves $\sG$ on $S$, and thus
$$
L'_{|F} \ = \ (-K_{\sF'})_{|F} \ = \ (-K_{X'/S})_{|F},
$$
which is a contradiction.

\medskip

In cases (3f), (3g) and (3j), $\varphi_L:X\to X'$ factors through a factorial $4$-fold $\tilde X$.
Denote by $\tilde L$ the push-forward of $L$ to  $\tilde X$.
The Mukai foliation $\sF$ induces a foliation $\tilde \sF$  on $\tilde X$  such that $-K_{\tilde \sF} \sim \tilde L$.

In case (3f),  $\tilde X$ is the blowup of $\p^4$ along a cubic surface $S$ contained in 
a hyperplane $F\subset \p^4$. 
Denote by $\tilde F\subset \tilde X$ the strict transform of $F$,
so that $N_{\tilde F/ \tilde X}\cong \sO_{\p^{3}}(-2)$ and $\sO_{\tilde F}\big(\tilde L_{|\tilde F}\big)\cong \sO_{\p^{3}}(1)$.
We will reach a contradiction by exhibiting a family $H$ of rational curves on $\tilde X$ such that
\begin{enumerate}
	\item the general member of $H$ is a curve tangent to $\tilde \sF$; and
	\item two general points of $\tilde X$ can be connected by a chain of curves from $H$ 
		avoiding the singular locus of $\tilde \sF$.
\end{enumerate}
We take $H$ to be the family of strict transforms of lines in $\p^4$ meeting $S$ and not contained in $F\cong \p^3$.
It is a minimal dominating family of rational curves on $\tilde X$ satisfying condition (2) above. 
Let $\ell\subset \tilde X$ be a general member of $H$.
One computes that $-K_{\tilde X}\cdot \ell =4$ and $-K_{\tilde \sF}\cdot \ell = \tilde L \cdot \ell \ge 3$.
For the latter, notice that $\ell \equiv \ell_1+ 2\ell_2$, where $\ell_1\subset \tilde F$ is a line under the isomorphism 
$\tilde F\cong\p^3$, and $\ell_2$ is a nontrivial fiber of the blowup $\tilde X\to \p^4$.
Condition (1) above then follows from Lemma \ref{lemma:bound_on_pseudo_index}.
We conclude that $(X',L')$ cannot be as in (3f).

In case (3g), $\tilde X$ is a conic bundle over $\p^3$.  
Moreover, there is a divisor $F\subset X$ mapping isomorphically onto its image by $X\to \tilde X$,  
and such that $(F, N_{F/ X})\cong \big(\p^3, \sO_{\p^{3}}(-2)\big)$.
Denote by $\pi:X\to \tilde X\to \p^3$ the composite map, and by $\ell$ a general fiber of $\pi$.
By Lemma \ref{lemma:bound_on_pseudo_index}, $\ell$ is tangent to $\sF$.
So, by  \ref{pullback_foliations}, $\sF$ is the pullback via $\pi$ 
of a codimension $1$ foliation $\sG$ on $\p^3$.
Let $C\subset \p^3$ be a general line, and set $S:=\pi^{-1}(C)$.
Note that $S$ is smooth, $\pi_C:=\pi_{|C}:S\to C$ is a conic bundle, and the foliation on $S$ induced by 
$\sF$ is precisely $T_{S/C}$.
Hence, by \ref{restricting_fols}, there is an effective divisor $B$ on $S$ such that 
\begin{equation}\label{eq13g}
-K_S \ = L_{|S} \ + \ B.
\end{equation}
On the other hand, by  \ref{pullback_foliations}, 
\begin{equation}\label{eq23g}
L_{|S} \ = -K_{S/C} \ + \ (\pi_C)^*c_1\big(\sG_{|C}\big).
\end{equation}
Equations \eqref{eq13g} and \eqref{eq23g} together imply that $B=(\pi_C)^*B_C$ for some
effective divisor $B_C$ on $C$, and thus $-K_S$ is ample. 
We will reach a contradiction by exhibiting a curve $\sigma \subset S$ such that  
$-K_S\cdot \sigma \le 0$.
We take $\sigma:= F\cap S$. 
Using that $\sO_{F}(F) \cong \sO_{\p^{3}}(-2)$, the adjunction formula implies that $-K_S\cdot \sigma \le 0$.
We conclude that $(X',L')$ cannot be as in (3g).

In case (3j), $\tilde X=\p_{\p^3}(\sO_{\p^3}\oplus \sO_{\p^3}(-2))$.
Denote by $\pi:X\to \tilde X\to \p^3$ the composite map, and by $\ell$ a general fiber of $\pi$.
One computes that $L \cdot \ell =1$.
By Lemma \ref{lemma:bound_on_pseudo_index}, $\ell$ is tangent to $\sF$.
So, by  \ref{pullback_foliations}, $\sF$ is the pullback via $\pi$ 
of a codimension $1$ foliation $\sG$ on $\p^3$. Thus
$$
L_{|\ell} \ = \ (-K_{\sF})_{|\ell}  \ = \ (-K_{X/\p^3})_{|\ell} \ = \ -K_\ell,
$$
which is a contradiction.

We conclude that $\tau(L)\ge n-1$.

\medskip 

\noindent {\bf Step 3: } We show that if $\tau(L)= n-1$, then one of the following conditions hold. 
\begin{enumerate}
	\item[(i)] $X$ admits a structure of quadric bundle over a smooth curve. In this case, by Proposition~\ref{proposition:Q-bdle_over_C},
		$X$ and $\sF$ are as described in Theorem~\ref{main_thm_rho>1}(3).
	\item[(ii)] $X$ admits a structure of $\p^{n-2}$-bundle over a smooth surface. In this case, by section \ref{subsection:P-bdles/surfaces},
		$X$ and $\sF$ are as described in Theorem~\ref{main_thm_rho>1}(4).		
	\item[(iii)] $n=5$, $\varphi_L: X\to \p^5$ is the blowup of one point $P\in \p^5$, and $\sF$ is induced by a pencil 
		of hyperplanes in $\p^5$ containing $P$ in its base locus. This gives Theorem~\ref{main_thm_rho>1}(5).
	\item[(iv)] $n=4$, $\varphi_L: X\to \p^4$ is the blowup of at most $8$ points in general position on a plane $\p^2 \cong S\subset \p^4$, 
		and $\sF$ is induced by the pencil of hyperplanes in $\p^4$ with base locus $S$. This gives Theorem~\ref{main_thm_rho>1}(6).
	\item[(v)] $n=4$, $\varphi_L: X\to Q^4$ is the blowup of at most $7$ points in general position on a 
		codimension $2$ linear section $Q^2 \cong S\subset Q^4$, 
		and $\sF$ is induced by the pencil of hyperplanes sections of $Q^4\subset \p^5$ with base locus $S$.
		This gives Theorem~\ref{main_thm_rho>1}(7).
\end{enumerate}

\medskip

Suppose that  $\tau(L)= n-1$.
By Theorem~\ref{tironi}, one of the following holds.
\begin{itemize}
	\item $X$ admits a structure of quadric bundle over a smooth curve. This is case (i).
	\item $X$ admits a structure of $\p^{n-2}$-bundle over a smooth surface. This is case (ii).
	\item $\varphi_L: X\to X'$ is the blowup of a  smooth projective variety at finitely many points $P_1, \ldots, P_k\in X'$.
\end{itemize}

Suppose that we are in the latter case, and 
denote the exceptional prime divisors of $\varphi_L$ by $E_i$, $1\le i\le k$.
Set $L':=(\varphi_L)_*(L)$. It  is an ample divisor on $X'$ and
\begin{equation}\label{L+E=L'}
L \ +\ \sum_{i=1}^k E_i \ =  \ (\varphi_L)^*L'.
\end{equation}

The Mukai foliation $\sF$ induces a Fano foliation $\sF'$  on $X'$  such that $-K_{\sF'} \sim (n-3)L'$.
By Theorem~\ref{Thm:KX-KF_not_nef}, $K_{X'}+(n-3)L'$ is not nef, i.e., $\tau(L') > n-3$.
On the other hand, since $K_{X'}+(n-1)L'= (\varphi_L)_*\big(K_{X}+(n-1)L\big)$ is ample, $\tau(L') < n-1$.
It follows from Theorem~\ref{tironi}, together with steps 1 and 2 above, that $\rho(X')=1$.
Let $H'$ be the ample generator of $\Pic(X')$, and 
write $L'\sim \lambda H'$, $\lambda\ge 1$.
Then $-K_{X'}= \iota_{X'}H'$,  $-K_{\sF'}= \lambda (n-3)H'$ and  
\begin{equation} \label{lambda1}
\tau(L') =\frac{\iota_{X'}}{\lambda} < n-1.  	
\end{equation}
By Lemma~\ref{lemma:bound_on_index2},
\begin{equation} \label{lambda2}
\iota_{X'} \ge \lambda (n-3) +2.  	
\end{equation}
Inequalities \eqref{lambda1} and \eqref{lambda2} together yield that $\lambda \ge 2$. 
On the other hand, by Theorem~\ref{Thm:ADK}, $\iota_{\sF'}=\lambda (n-3)\le n-1$.
Thus $(n,\lambda)\in \big\{(5,2),(4,3),(4,2)\big\}$.

Let $\omega' \in H^0\big(X', \Omega_{X'}^1\big( -K_{X'}-(n-3)L' \big)\big)$
be the  twisted $1$-form defining $\sF'$. 
The induced twisted $1$-form $(\varphi_L)^*\omega' \in H^0\big(X, \Omega_X^1\big( (\varphi_L)^*(-K_{X'}-(n-3)L') \big)\big)$
saturates to give the  twisted $1$-form defining $\sF$, 
$\omega_{\sF} \in H^0\big(X, \Omega_{X}^1\big( -K_{X}-(n-3)L \big)\big)$.
Using \eqref{L+E=L'}, one computes that 
$$
-K_{X}-(n-3)L= (\varphi_L)^*\big( -K_{X'}-(n-3)L' \big) + 2 \sum_{i=1}^k E_i.
$$
Thus $(\varphi_L)^*\omega'$ must vanish along each $E_i$ with multiplicity exactly $2$.

\medskip

Suppose that $(n,\lambda)=(5,2)$. Then $\iota_{\sF'}=\rank(\sF')$ and, by Theorem~\ref{Thm:ADK},
$X'\cong \p^5$, and $\sF'$ is 
a foliation  induced by a pencil of hyperplanes in $\p^5$.
We claim that 
$\varphi_L: X\to \p^5$ is the blowup of only one point.
Indeed, if $\varphi_L: X\to \p^5$ blows up at least 2 points $P$ and $Q$, let $\ell$ be a line connecting $P$ and $Q$, and
$\tilde \ell\subset X$ its strict transform. 
We get a contradiction by intersecting \eqref{L+E=L'} with $\tilde \ell$, and conclude that 
$\varphi_L: X\to \p^5$ is the blowup of a single point $P\in \p^5$, with exceptional divisor $E$.
Moreover,   $(\varphi_L)^*\omega'$ vanishes along $E$ with multiplicity exactly $2$.
A local computation shows that this happens precisely when $P$ is in the base locus of
the pencil of hyperplanes defining $\sF'$.

\medskip

Suppose that $(n,\lambda)=(4,3)$. Then $\iota_{\sF'}=\rank(\sF')$ and, by Theorem~\ref{Thm:ADK},
$X'\cong \p^4$, and $\sF'$ is a foliation  induced by a pencil of hyperplanes in $\p^4$.
Moreover,   $(\varphi_L)^*\omega'$ vanishes along each $E_i$ with multiplicity exactly $2$.
A local computation shows that this happens precisely when the points $P_i$ all lie in the base locus of
the pencil of hyperplanes defining $\sF'$, which is a codimension $2$ linear subspace . 
Since $L =  (\varphi_L)^*L' - \sum_{i=1}^k E_i$ is ample, we must have $k\le 8$ 
and the $P_i$'s are in general position by
Lemma~\ref{number_of_pts} below.

\medskip

Suppose that $(n,\lambda)=(4,2)$. 
Then $\sF'$ is a codimension $1$ del Pezzo foliation on $X'$.
By Theorem~\ref{Thm:codim1_dP},
either $X'\cong\p^4$ and $\sF'$ is a degree $1$ foliation, or 
$X'\cong Q^4\subset \p^{5}$ and $\sF'$ is induced by a pencil of hyperplane sections.

Suppose first that $X'\cong\p^4$ and $\sF'$ is a degree $1$ foliation.
The same argument used in the case $(n,\lambda)=(5,2)$ shows that 
$\varphi_L: X\to \p^4$ is the blowup of only one point $P\in \p^4$, with exceptional divisor $E$.
Moreover,   $(\varphi_L)^*\omega'$ vanishes along $E$ with multiplicity exactly $2$.
A local computation shows that this cannot happen.

Finally suppose that $X'\cong Q^4\subset \p^{5}$ and $\sF'$ is induced by a pencil of hyperplane sections.
Moreover,   $(\varphi_L)^*\omega'$ vanishes along each $E_i$ with multiplicity exactly $2$.
A local computation shows that this happens precisely when the points $P_i$ all lie in the base locus of
the pencil  defining $\sF'$, which is a codimension $2$ linear section of $Q^4$. 
Since $L =  (\varphi_L)^*L' - \sum_{i=1}^k E_i$ is ample, we must have $k\le 7$ 
and the $P_i$'s are in general position by
Lemma~\ref{number_of_pts} below.

\medskip 

\noindent {\bf Step 4: } Finally suppose that $\tau(L)= n$.
By  Theorem~\ref{tironi}, $X$ admits a structure of $\p^{n-1}$-bundle over a smooth curve. 
In this case, by Proposition~\ref{lemma:P-bdle_over_C},
$X$ and $\sF$ are as described in Theorem~\ref{main_thm_rho>1}(1) or (2).
\qed

\

\begin{lemma}\label{number_of_pts}
\begin{enumerate}
	\item Let $\pi:X\to \p^4$ be the blowup of finitely many points $P_1, \ldots, P_k$ contained in a codimension 2 linear subspace $S\cong \p^2$,
		and denote by $E_i$ the exceptional divisor over $P_i$.
		Then the line bundle $\pi^*\sO_{\p^4}(3) \otimes \sO_X(- \sum_{i=1}^k E_i)$ is ample if and only if 
		$k\le 8$ and the $P_i$'s are in general position in $\p^2$.
	\item  Let $\pi:X\to Q^4$ be the blowup of a smooth quadric at finitely many points $P_1, \ldots, P_k$ 
		contained in a codimension 2 linear section $S$,
		and denote by $E_i$ the exceptional divisor over $P_i$.
		Then the line bundle $\pi^*\sO_{Q^4}(2) \otimes \sO_X(- \sum_{i=1}^k E_i)$ is ample if and only if 
		$k\le 7$ and the $P_i$'s are in general position in $S$. 
\end{enumerate}
\end{lemma} 

\begin{proof}
Under the assumptions of (1) above, write 
$L= \pi^*3H - \sum_{i=1}^k E_i$, where $H$ is a hyperplane in $\p^4$.
The divisor $L$ is ample if and only if $|mL|$ separates points in $X$ for $m\gg 1$. 
Notice that $|L|$ always separates points outside the strict transform $\tilde S$ of the plane 
$S\subset \p^4$. 
Moreover, for $m\ge 0$, any global section of $\sO_{\tilde S}\big(mL_{|\tilde S}\big)$ extends to a 
global section 
of $\sO_X(mL)$.
Hence,   $L$ is ample if and only if $L_{|\tilde S}$ is ample. 
Now notice that 
$$
L_{|\tilde S}\ = \ p^*(-K_S)\ -\ \sum_{i=1}^k {E_i}_{|\tilde S} \ = \ -K_{\tilde S},
$$
where $p=\pi_{|\tilde S}:\tilde S\to S\cong \p^2$ is the blowup of $\p^2$ at $P_1, \ldots, P_k$.
Therefore, $L$ is ample if and only if $-K_{\tilde S}$ is ample, i.e., 
$k\le 8$ and the $P_i$'s are in general position in $\p^2$.

Now we proceed to prove (2). Let $\tilde S$ be the strict transform of the (possibly singular) 
irreducible quadric surface $S\subset Q^4$. 
Write as above $L= \pi^*2H - \sum_{i=1}^k E_i$, where $H\subset Q^4$ is a hyperplane section.
Notice that $|L|$ always separates points outside $\tilde S$.
Moreover, any global section of $\sO_{\tilde S}\big(L_{|\tilde S}\big)$ extends to a 
global section of $\sO_X(L)$.
Suppose that $L$ is ample, so that $L_{|\tilde{S}}=-K_{\tilde{S}}$ is ample as well. Then $k\le 7$ and the $P_i$'s are in general position in $S$.
Conversely, suppose that $k\le 7$, and the $P_i$'s are in general position in $S$.
Then $\dim(\textup{Bs}(-K_{\tilde{S}}))\le 0$, and hence
$\dim(\textup{Bs}(L))\le 0$.
We conclude that 
$L$ is ample by Zariski's Theorem (\cite[Remark 2.1.32]{lazarsfeld1}).
\end{proof}

\bibliographystyle{amsalpha}
\bibliography{foliation}

\end{document}